\documentclass[11pt, twoside, leqno]{amsart}  

\usepackage{scalerel,stackengine}

\makeatletter
\@namedef{subjclassname@2020}{\textup{2020} Mathematics Subject Classification}
\makeatother 

\usepackage{lipsum}
\usepackage{amsfonts}
\usepackage{graphicx}
\usepackage{epstopdf}
\usepackage{algorithmic}
\usepackage{calligra}
\usepackage[dvipsnames]{xcolor}
\usepackage{amsfonts,amsmath,amsthm,amssymb}
\usepackage{mathtools}
\usepackage{hyperref}
\usepackage[makeroom]{cancel}
\usepackage{autonum}
\usepackage{hhline}
\usepackage{array}
\usepackage{diagbox}

\usepackage{mdframed}
\usepackage{multicol}
\usepackage{graphicx}
\usepackage{subcaption}
\usepackage{moreverb}
\usepackage{bbm}
\usepackage[margin=1.15in]{geometry}
\usepackage{todonotes}
\usepackage{scalerel,amssymb}
\allowdisplaybreaks
\usepackage{mathrsfs}  
\usepackage{lineno}
\usepackage{todonotes}
\usepackage{tikz}
\usepackage{appendix}
\usepackage{enumitem}
\usepackage{pgfplots}
\usetikzlibrary{arrows.meta}
\usepackage{siunitx}
\usepackage[numbers,sort&compress]{natbib}
\definecolor{mygreen}{HTML}{43a047}
\usepackage{subcaption}
\usepackage{doi}
\usepackage{soul}
\usepackage[normalem]{ulem}
\usepackage{wrapfig}
\newcommand{\betaa}{\beta_{\textup{a}}}
\newcommand{\Gronwall}{Gr\"onwall}

\newcommand{\calT}{{\mathcal{T}}}

\newcommand{\calG}{\mathcal{G}}
\newcommand{\frakG}{\mathfrak{G}}
\newcommand{\ulb}{\underline{b}}
\newcommand{\olb}{\overline{b}}
\newcommand{\ulq}{\underline{q}}
\newcommand{\olq}{\overline{q}}

\newcommand{\op}{\overline{p}}
\newcommand{\opt}{\overline{p}_t}
\newcommand{\opstar}{\overline{p}_*}
\newcommand{\opstart}{\overline{p}_{*t}}

\newcommand{\otheta}{\overline{\theta}}
\newcommand{\othetat}{\overline{\theta}_t}

\newcommand{\Om}{\Omega}
\newcommand{\D}{\Delta}
\newcommand{\pn}{p^n}

\newcommand{\rhob}{\rho_{\textup{b}}}
\newcommand{\rhoa}{\rho_{\textup{a}}}
\newcommand{\Ca}{C_{\textup{a}}}
\newcommand{\Cb}{C_{\textup{b}}}
\newcommand{\wb}{\omega_{\textup{b}}}
\newcommand{\kappaa}{\kappa_{\textup{a}}}
\newcommand{\Thetaa}{\Theta_{\textup{a}}}

\newcommand{\ftheta}{f_\theta}
\newcommand{\fp}{f_p}
\newcommand{\fc}{f_c}

\newcommand{\pstar}{p_\ast}
\newcommand{\thetastar}{\theta_\ast}

\def\Rp{R_p}

\def\Rtheta{R_\theta}

\def\Xtheta{\mathcal{X}_\theta}

\newcommand{\pt}{p_t}
\newcommand{\ptt}{p_{tt}}

\newcommand{\ddt}{\frac{\textup{d}}{\textup{d}t}}

\newcommand{\dt}{\, \textup{d} t}
\newcommand{\ds}{\, \textup{d} s }
\newcommand{\dx}{\, \textup{d} x}

\newcommand{\dxs}{\, \textup{d}x\textup{d}s}

\newcommand{\intTO}{\int_0^T \int_{\Omega}}
\newcommand{\inttO}{\int_0^t \int_{\Omega}}
\newcommand{\intt}{\int_0^t}
\newcommand{\intT}{\int_0^T}
\newcommand{\intO}{\int_{\Omega}}
\newcommand{\nLtwo}[1]{\|#1\|_{L^2(\Omega)}}

\newcommand{\change}[1]{{\textcolor{black}{#1}}}

\newcommand{\R}{\mathbb{R}} 
\newcommand{\N}{\mathbb{N}} 

\newcommand{\Htwo}{H^2(\Omega)}
\newcommand{\Ltwo}{L^2(\Omega)}
\newcommand{\Linf}{L^\infty(\Omega)}

\newcommand{\Hneg}{H^{-1}(\Omega)}

\newcommand{\Lfour}{L^4(\Omega)}
\newcommand{\Hone}{H^1(\Omega)}

\newcommand{\Honezero}{H_0^1(\Omega)}
\newcommand{\Hthree}{H^3(\Omega)}
\def\LtwoLsix{L^2(\Lsix)}
\newcommand{\Honetwo}{{H^2(\Omega)\cap H_0^1(\Omega)}}

\newcommand{\Honethree}{{H_\diamondsuit^3(\Omega)}}

\def\LinfLtwo{L^\infty(L^2(\Omega))}

\def\LtwoLfour{L^2(L^4(\Omega))}
\def\Lthree{L^3(\Omega)}
\def\LinfLfour{L^\infty(L^4(\Omega))}
\def\LonetLtwo{L^1_t(\Ltwo)}
\def\Linf{L^\infty(\Omega)}
\def\LinftLfour{L^\infty_t(\Lfour)}
\def\LtwotLfour{L^2_t(\Lfour)}
\def\LinftHone{L^\infty_t(\Hone)}

\newcommand{\Xp}{\mathcal{X}_p}

\newcommand{\Yp}{\mathcal{X}^0_p}

\newcommand{\Xc}{\mathcal{X}_c}

\newcommand{\pstart}{p_{*t}}

\newcommand{\Ctheta}{C_{\theta}}

\newtheorem{theorem}{Theorem}
\newtheorem{lemma}{Lemma}
\newtheorem{proposition}{Proposition}

\newtheorem{assumption}{Assumption}
\newtheorem{definition}{Definition}

\newtheorem{corollary}{Corollary}
\numberwithin{lemma}{section}
\numberwithin{proposition}{section}
\numberwithin{theorem}{section}
\numberwithin{equation}{section}
\makeatletter
\newcommand{\leqnomode}{\tagsleft@true}
\newcommand{\reqnomode}{\tagsleft@false}
\makeatother

\definecolor{grey}{rgb}{0.5,0.5,0.5}

\definecolor{darkgreen}{rgb}{0,0.5,0}

\def\Honez{\Honezero}
\newcommand{\LinfTLinf}{L^\infty(0,T; \Linf)}

\def\LinfLinf{L^\infty(\Linf)}

\def\LtwoHone{L^2(\Hone)}
\def\LtwoTHonez{L^2(0,T;\Honezero)}

\def\LinfHone{L^\infty(\Hone)}
\def\LtwoLtwo{L^2(L^2(\Omega))}
\def\LtwoTLtwo{L^2(0,T; L^2(\Omega))}
\def\LoneTLtwo{L^1(0,T;\Ltwo)}
\def\LoneLtwo{L^1(\Ltwo)}

\def\LtwoHtwo{L^2(\Htwo)}
\def\LtwoLfour{L^2(\Lfour)}

\def\LinftLtwo{L^\infty_t(\Ltwo)}

\def\LtwotHone{L^2_t(\Hone)}
\def\LinfHtwo{L^\infty(\Htwo)}

\def\LinfTLinf{L^\infty(0,T; \Linf)}
\def\LinfTLtwo{L^\infty(0,T; \Ltwo)}
\def\LoneLtwo{L^1(\Ltwo)}
\def\LoneTHneg{L^1(0,T; \Hneg)}
\def\LoneLinf{L^1(\Linf)}
\def\LtwotLtwo{L^2_t(\Ltwo)}

\def\Lsix{L^6(\Omega)}

\def\Lthree{L^3(\Omega)}

\def\LtwoLthree{L^2(\Lthree)}

\def\LinfLthree{L^\infty(\Lthree)}
\def\LinfLsix{L^\infty(\Lsix)}

\def\LinfTHonetwo{L^\infty(0,T; \Honetwo)}
\def\LinfHthree{L^\infty(\Hthree)}
\def\LinfTHonethree{L^\infty(0,T; \Honethree)}
\def\HoneTHonezero{H^1(0,T; \Honezero)}
\def\HoneHone{H^1(\Hone)}
\def\LtwoTHonetwo{L^2(0,T; \Honetwo)}
\def\Hfour{H^4(\Omega)}

\def\HoneHtwo{H^1(\Htwo)}
\def\HoneLtwo{H^1(\Ltwo)}
\def\HoneTHtwo{H^1(0,T; \Htwo)}

\def\pzero{p_0}

\def\deltap{\delta_p}

\def\deltatheta{\delta_{\theta}}

\def\thetazero{\theta_0}

\def\bfv{\boldsymbol{v}}

\def\CL{C_{\textup{L}}}

\def\eps{\varepsilon}

\def\thetat{\theta_t}

\def\ball{B_{\Rp, \Rtheta}}

\def\thetaone{\theta^{(1)}}
\def\thetatwo{\theta^{(2)}}
\def\pone{p^{(1)}}
\def\ptwo{p^{(2)}}

\def\thetaonestar{\theta_*^{(1)}}
\def\thetatwostar{\theta_*^{(2)}}
\def\ponestar{p_*^{(1)}}
\def\ptwostar{p_*^{(2)}}
\def\ponestart{p_{*t}^{(1)}}

\def\bRtheta{\bar{R}_\theta}
\def\bRp{\bar{R}_p}

\def\ponestart{p^{(1)}_{*t}}
\def\ptwostart{p^{(2)}_{*t}}
\def\otheta{\overline{\theta}}
\def\othetastar{\overline{\theta}_*}

\def\qstar{q_*}
\def\kstar{k_*}
\def\bstar{b_*}

\def\qstart{q_{*t}}

\def\wbstar{\omega_{\textup{b}, *}}
\def\wbstart{\omega_{\textup{b}, *t}}

\def\thetastart{\theta_{*t}}
\def\konestar{\kstar^{(1)}}
\def\ktwostar{\kstar^{(2)}}

\def\bonestar{\bstar^{(1)}}
\def\btwostar{\bstar^{(2)}}
\def\qonestar{\qstar^{(1)}}
\def\qtwostar{\qstar^{(2)}}
\def\wbonestar{\wbstar^{(1)}}
\def\wbtwostar{\wbstar^{(2)}}

\def\ptwostartt{p^{(2)}_{*tt}}
\def\Cp{C_p}
\def\Ctheta{C_\theta}
\def\solspaceWestPennes{\boldsymbol{X}_{p, \theta}}

\def\fc{f_c}
\def\ct{c_t}
\def\calM{\mathcal{M}}
\def\psionestar{\psi_*^{(1)}}
\def\psitwostar{\psi_*^{(2)}}
\def\solspaceWestPennesconcentr{\boldsymbol{X}_{p, \theta, c}}

\def\czero{c_0}
\def\Xc{\mathcal{X}_c}
\def\calD{\mathcal{D}}

\def\Czerooneloc{C^{0,1}_{\textup{loc}}}

\def\Ctwooneloc{C^{2,1}_{\textup{loc}}}
\def\Cthreeoneloc{C^{3,1}_{\textup{loc}}}

\def\frakK{\mathfrak{K}}
\def\Dtalpha{\textup{D}_t^\alpha}

\def\frakm{\mathfrak{m}}
\def\frakn{\mathfrak{n}}
\def\frakl{\mathfrak{l}}

\def\ulfrakm{\underline{\frakm}}
\def\olfrakm{\overline{\frakm}}
\def\ulfrakn{\underline{\frakn}}
\def\olfrakn{\overline{\frakn}}
\def\CfrakK{C_{\frakK}}
\def\fpt{f_{p,t}}
\def\frakmt{\frakm_t}

\def\Xfrakm{X_\frakm}
\def\Xfrakn{X_\frakn}
\def\Xfrakl{X_\frakl}

\def\calL{\mathcal{L}}
\def\ft{f_t}
\def\fraknt{\frakn_t}

\def\frakmone{\frakm^{(1)}}
\def\frakmtwo{\frakm^{(2)}}

\def\fraknone{\frakn^{(1)}}
\def\frakntwo{\frakn^{(2)}}

\def\fraklone{\frakl^{(1)}}
\def\frakltwo{\frakl^{(2)}}

\def\optt{\op_{tt}}

\def\Yp{\mathcal{Y}_p}
\def\Ytheta{\mathcal{Y}_\theta}

\newcommand{\bfmu}{\boldsymbol{\mu}}
\newcommand{\bfxi}{\boldsymbol{\xi}}

\def\pnt{p^n_t}
\def\pntt{p^n_{tt}}

\def\Mfrakm{\mathbb{M}_\frakm}
\def\Lfrakl{\mathbb{L}_\frakl}
\def\Nfrakn{\mathbb{N}_\frakn}
\def\Mfrakmij{\mathbb{M}_{\frakm, ij}}
\def\Nfraknij{\mathbb{N}_{\frakn, ij}}
\def\Lfraklij{\mathbb{L}_{\frakl, ij}}

\def\ph{p_h}
\def\pht{p_{h,t}}

\def\frakf{\mathfrak{f}}       
 
\title[Westervelt-based modeling of ultrasound-enhanced drug delivery]{Westervelt-based modeling of ultrasound-enhanced drug delivery}  
\subjclass[2020]{35L70, 35R11, 35M30, 65M22}         

\keywords{Westervelt's equation, Pennes equation, nonlocal attenuation, multiphysics systems, local well-posedness, ultrasound-enhanced drug delivery}   

\author{Julio Careaga$^\dagger$}
\thanks{$^\dagger$Department of Mathematics, University of B\'io-B\'io, Avenida Collao 1202, Casilla 5-C, Concepci\'on, Chile  (\href{jcareaga@ubiobio.cl}{jcareaga@ubiobio.cl})}

\author{Vanja Nikoli\'c$^\ddag$}
\thanks{$^\ddag$Department of Mathematics,
	Radboud University, 
	Heyendaalseweg 135, 
	6525 AJ Nijmegen, The Netherlands (\href{vanja.nikolic@ru.nl}{vanja.nikolic@ru.nl})}
\author{Belkacem Said-Houari$^\S$}
\thanks{$^S$Department of Mathematics, College of Sciences, University of
	Sharjah, P. O. Box: 27272, Sharjah, United Arab Emirates (\href{bhouari@sharjah.ac.ae}{bhouari@sharjah.ac.ae})}
\begin{document}
	\vspace*{8mm}
	\begin{abstract}
	We investigate a nonlinear multiphysics model motivated by ultrasound-enhanced drug delivery.  The acoustic pressure field is modeled by Westervelt's quasilinear wave equation to adequately capture the nonlinear effects in ultrasound propagation. The nonlocal attenuation characteristic for soft biological media is modeled by acoustic damping of the time-fractional type.  Additionally, acoustic medium parameters are allowed to depend on the temperature of the medium. The wave equation is coupled to the nonlinear Pennes heat equation with a pressure-dependent source to account for ultrasound waves heating up the tissue.  Finally, the drug concentration is obtained as the solution to an advection-diffusion equation with a pressure-dependent velocity. Toward gaining a rigorous understanding of this system, we set up a fixed-point argument in the analysis combined with devising energy estimates that can accommodate the time-fractional damping. The energy arguments are, in part, carried out by employing time-weighted test functions to reduce the regularity assumptions on the initial temperature. The analysis reveals that different smoothness of the initial pressure, temperature, and concentration fields is needed as well as smallness of the pressure-temperature data in order to ensure  non-degeneracy of the system and establish well-posedness. 
	  Our theoretical considerations are complemented by a numerical investigation of the system under more realistic boundary conditions.  The numerical experiments, performed in different computational scenarios, underline the importance of considering nonlinear effects when modeling ultrasound-targeted drug delivery. 	\end{abstract}   
	\vspace*{-7mm}  
	\maketitle             
	\section{Introduction}
Over the last years, the potential for using ultrasound waves as a way of improving drug delivery has gained recognition in various medical treatments.   In particular, High-Intensity Focused Ultrasound (HIFU) waves have found important applications in enhancing tumor therapy, treatments of central nervous diseases, and cardiovascular issues due to their non-invasive nature and the ability to target a particular area of interest in the body; see~\cite{mcclureUsingHighIntensityFocused2016, phenixHighIntensityFocused2014, delaneyMakingWavesHow, pittUltrasonicDrugDelivery2004, pitt2004ultrasonic}.  The transport of drugs is enhanced through the oscillating motion of the medium under the influence of ultrasound~\cite{pitt2004ultrasonic}, which influences the speed of drug delivery. \\
	\indent In this work, we investigate a mathematical model that captures several prominent effects of HIFU-enhanced drug delivery. At high intensities,  ultrasonic waves exhibit nonlinear effects, including wavefront steepening and, related to it, generation of higher harmonics; see, e.g.,~\cite{hamiltonNonlinearAcoustics1998}. Secondly, acoustic waves that propagate through complex media such as soft biological tissue are known to follow nonlocal attenuation patterns of time-fractional type; see, e.g.~\cite{prieurNonlinearAcousticWave2011, holmWavesPowerlawAttenuation2019}. Thirdly, the acoustic medium parameters in general depend on the temperature of the medium. These effects in ultrasound propagation can be modeled by Westervelt's wave equation, which is given in terms of the acoustic pressure $p=p(x,t)$ by
	\begin{equation}\label{Westervelt_eq}
		p_{tt}-q(\Theta)\Delta p - b(\Theta) \frakK *\Delta \pt = k(\Theta)\left(p^2\right)_{tt}.
	\end{equation}
The nonlocal attenuation is captured by the term $- b(\Theta) \frakK *\Delta \pt$,	where $*$ is the Laplace convolution operator:
	\begin{equation}
		(\frakK* v)(t) = \int_0^t \frakK(t-s) v(s) \ds
	\end{equation}
and $\frakK$ the memory kernel. In the analysis we impose relatively non-restrictive conditions on the memory kernel $\frakK$ which allow us to cover, among others, the practically most relevant Abel kernel:
\begin{equation} \label{Abel kernel}
	\frakK(t) = \frac{1}{\Gamma(1-\alpha)}t^{-\alpha}\qquad \text{for }\alpha \in (0,1),
\end{equation} 
where $\Gamma(\cdot)$ is the Gamma function. This choice leads to the Westervelt equation with the damping involving the Caputo--Djrbashian derivative $\Dtalpha(\cdot)$ of order $\alpha \in (0,1)$:
\begin{subequations}\label{Main_system}
\begin{equation}\label{Westervelt_eq Caputo}
p_{tt}-q(\Theta)\Delta p - b(\Theta) \Delta \Dtalpha p = k(\Theta)\left(p^2\right)_{tt}.
\end{equation}
For the definition of $\Dtalpha(\cdot)$, we refer to \eqref{def Caputo derivative} below. The non-isothermal setting is reflected by having temperature-dependent acoustic medium parameters in \eqref{Westervelt_eq}; namely, the speed of sound squared $q=q(\Theta)$, damping coefficient $b=b(\Theta)$, and nonlinearity coefficient $k=k(\Theta)$.  Possible forms and the assumptions made on these functions of temperature in the analysis are discussed in Section~\ref{Sec:Preliminaries}. We note that in particular allowing for the sound speed to depend on the temperature is important in modeling therapeutic HIFU applications as the changes in the propagation speed of pressure waves with the changed temperature may lead to the shifting of the focal area; that is, the region with the peak pressure values. This effect is known as \emph{thermal lensing}; see, e.g.~\cite{connorBioacousticThermalLensing2002}.  \\
	\indent  At high intensities, focused ultrasound waves generate heat. This heat transfer is commonly modeled using the Pennes equation \cite{pennesAnalysisTissueArterial1948} for the temperature $\Theta=\Theta(x,t)$:
		\begin{equation} \label{Pennes_eq notscaled}
	\rhoa \Ca\Theta_t -\kappaa\Delta \Theta+ \wb(\Theta) \rhob \Cb (\Theta-\Thetaa) = \tilde{\calG}(\pt, \Theta)+f_\Theta. 
\end{equation}
In the above equation, the medium parameters $\rhoa, \Ca$ and $\kappaa$ stand, respectively, for the ambient density, the ambient heat capacity, and thermal conductivity of the tissue. The additional term $\wb(\Theta) \rhob \Cb (\Theta-\Thetaa)$ accounts for the heat loss due to blood circulation, with  $\rhob$ and $\Cb$ being the density and specific heat capacity of the blood, respectively. Further, $\wb$ is the mass flow rate of blood per unit volume of tissue. This blood perfusion is known to depend on the temperature in practice and its qualitative behavior is influenced by the tissue being healthy or cancerous; see,  e.g.,~\cite{dengParametricStudiesPhase2000, langImpactNonlinearHeat1999, kimNonlinearFiniteelementAnalysis1996, songImplicationBloodFlow1984, tompkinsnTemperaturedependentConstantrateBlood1994}. The dependence of $\wb$ on the temperature $\Theta$ is outlined in Section \ref{Section_Assumptions}.
	The source term in the Pennes equation contains $\tilde{\calG}$ which models the energy absorbed by the tissue; we discuss its form in detail in Section~\ref{Sec:Preliminaries}. In the theory, we allow for an additional pressure-independent source of heat  $f_\Theta=f_\Theta(x,t)$.   \\
	\indent  Having modeled the wave-heat interaction, the ultrasound-enhanced drug transport is captured with a convection-diffusion equation for the drug concentration $c=c(x,t)$:
	\begin{equation} \label{concentration equation}  
		c_t + \nabla \cdot \big(\bfv(p, \nabla p) c\big)- \nabla \cdot \big(D(p, \Theta) \nabla c\big) = f_c
	\end{equation}
	\end{subequations}
with the source term $\fc=\fc(x,t)$. 
Equation \eqref{concentration equation} models physical phenomena where particles, energy, or other quantities   are transported within a system through both diffusion and convection processes.
	The second term  in \eqref{concentration equation} is a convection term, where
	$\bfv=\bfv(p, \nabla p)$ is the convective velocity that depends on the acoustic pressure and its gradient, and it measures the influence of the acoustic wave in drug transport.   
	The function  $D(p, \Theta)=D_0 + \calD(p,  \Theta)$ is split into a diffusion matrix $D_0$, modeled according to Fick's law, and a perturbation involving $\calD(p, \Theta)$
	which depends on both the acoustic pressure and the temperature.  	

\subsection*{Main contributions and relation to the existing literature}
	The main contributions of this work pertain to the rigorous study of the multiphysics system in \eqref{Main_system} and a further numerical investigation of its properties. \\
	\indent To the best of our knowledge, coupled systems involving wave, heat, and advection-diffusion problems based on nonlinear acoustic models of Westervelt type have not been mathematically investigated before, even in much simpler local-in-time  settings. In general, multiphysics problems involving time-fractional evolution are under-investigated in terms of rigorous analysis. We provide an overview below of helpful related works on the Westervelt equation and linear or local versions of the problem.
	\begin{itemize}[leftmargin=0.4cm]
		\item \emph{Westervelt's equation with strong damping}.
		Equation \eqref{Westervelt_eq} can be seen as a nonlocal isothermal version of the Westervelt equation with strong damping:
		\begin{equation}\label{Westervelt_eq_1}
p_{tt}-c^2\Delta p - b \Delta p_t = k\left(p^2\right)_{tt},
\end{equation}
which models the propagation of sound waves through thermoviscous fluids; see, e.g.,~\cite{kaltenbacherMathematicsNonlinearAcoustics2015}.  This equation has been studied analytically in various settings; we mention, e.g., the global well-posedness analysis   in~\cite{kaltenbacherGlobalExistenceExponential2009, meyerOptimalRegularityLongTime2011}.  More precisely, in~\cite{kaltenbacherGlobalExistenceExponential2009}, the authors established both local and global well-posedness, along with exponential decay rate, under Dirichlet boundary conditions, provided that the initial data are small and smooth enough. By employing the maximal $L^p$-regularity for abstract quasilinear parabolic equations, the authors in \cite{meyerOptimalRegularityLongTime2011} extended the result of \cite{kaltenbacherGlobalExistenceExponential2009}  to $L^p$-based Sobolev spaces thereby improving upon the regularity assumptions on data in \cite{kaltenbacherGlobalExistenceExponential2009}. \vspace*{1mm}
		\item \emph{Westervelt's equation with nonlocal damping}.  Well-posedness of the isothermal Westervelt equation with time-fractional attenuation given by
				\begin{equation}
			p_{tt}-c^2\Delta p - b \Delta \Dtalpha p = k\left(p^2\right)_{tt},
		\end{equation}
		 has been studied  in~\cite{kaltenbacherInverseProblemNonlinear2021, bakerNumericalAnalysisTimestepping2024} under homogeneous Dirichlet conditions, and, in the integro-differential framework that we adopt here with the memory kernel $\frakK$,  in~\cite{kaltenbacher2024limiting}, together with a study of the limiting behavior of solutions as $b \searrow 0 $.  \vspace*{1mm}
		\item \emph{The Westervelt--Pennes  system with strong acoustic damping}. Mathematical literature on local nonlinear wave-heat systems is quite rich, in particular, many results have been obtained for models arising in thermoelasticity; see the book~\cite{zheng2020nonlinear} and, e.g.,~\cite{rackeGlobalSolvabilityExponential1993, jiangEvolutionEquationsThermoelasticity2000}, and the references provided therein.  In the context of nonlinear thermo-acoustics, the results are more recent. The Westervelt--Pennes coupling with strong acoustic damping (that is, $\frakK=\delta_0$), which has a parabolic-parabolic nature, has been recently analyzed with homogeneous Dirichlet--Dirichlet conditions  in a simplified setting of the Pennes equation being linear in terms of the temperature (that is with $\wb=\,$const. and $\tilde{\calG}=\tilde{\calG}(p)$) in~\cite{nikolicLocalWellposednessCoupled2022, nikolicWesterveltPennesModel2022} using energy methods.  
		More precisely,  in \cite{nikolicLocalWellposednessCoupled2022} the authors proved a  local well-posedness result by  using the energy method together with a fixed point argument. In~\cite{nikolicWesterveltPennesModel2022} and  under a smallness assumption on the initial data, they  established the existence of a global-in-time solution  and proved an exponential decay of the solution.
		The  regularity assumptions in \cite{nikolicLocalWellposednessCoupled2022} have been  improved
		using an $L^p$- maximal regularity approach in~\cite{wilkeL_pL_qTheory2023}. Local well-posedness analysis of the system with mixed Neumann and nonlinear absorbing boundary   conditions for the pressure can be found in~\cite{nikolicSolvabilityWesterveltPennes2024}.  We also point out that the local Westervelt equation coupled with the hyperbolic Pennes equation resulting in what is known as the Westervelt--Cattaneo--Pennes system has been recently considered in~\cite{benabbasLocalWellposednessCoupled2024}, where a local well-posedness has been established and in \cite{benabbasGlobalExistenceAsymptotic2024} where the authors proved global existence and time-asymptotic behavior of the solution.  \vspace*{1mm}
				\item \emph{Linear and local wave-heat-concentration systems}. We mention also that a finite element analysis of linear versions of the system  with $k\equiv0$ in the wave equation and temperature-independent acoustic parameters as well as local-in-time acoustic damping can be found in \change{~\cite{ferreiraDrugDeliveryEnhanced2022}}. 
	\end{itemize}
	These available results however leave open the important questions of the influence of the fractional dissipation on the Westervelt--Pennes system as well as the properties of the concentration field, which are the main aims of the present work. In addition, we discuss the numerical approximation of this complex problem and illustrate the properties of this multiphysics system with numerical experiments. 
	
		\subsection*{Challenges and the present approach}
	The nonlinear and nonlocal multiphysics nature of the problem presents us with several challenges in the analysis. Observe, however, that the concentration $c$ does not appear in equations \eqref{Westervelt_eq} and \eqref{Pennes_eq notscaled},  and we can thus investigate the local-well posedness of \eqref{Westervelt_eq}--\eqref{Pennes_eq notscaled} separately first and then proceed to study the concentration field. \\
	\indent The analysis of the Westervelt--Pennes sub-system will be based on a fixed-point argument applied on a linearization of {\eqref{Westervelt_eq}--\eqref{Pennes_eq notscaled}}. Here deriving energy estimates (at first in a Galerkin semi-discrete  setting) for a suitable linearization of the system is crucial but highly non-trivial. In the testing procedure, we naturally need to exploit the coercivity property of the memory term. That is, having at our disposal an estimate of the form
	\begin{equation}
		\begin{aligned}
			\int_0^{t} \intO \left(\frakK* y \right)(s) \,y(s)\dxs\geq 	\CfrakK
			\int_0^{t} \|(\frakK* y)(s)\|^2_{\Ltwo} \ds,  \quad y\in L^2(0, T; \Ltwo)
		\end{aligned}
	\end{equation}
	is one of the key tools. However, this coercivity is difficult to achieve for kernels of practical interest with a space-time dependent coefficient $b=b(\Theta)$ in front of $\frakK*\Delta \pt$. For this reason, in the analysis we rewrite the Westervelt equation in the following form: 
	\begin{equation}
		\frac{1}{b(\Theta)}	p_{tt}-\frac{q(\Theta)}{b(\Theta)}\Delta p -  \frakK*\Delta  \pt = \frac{k(\Theta)}{b(\Theta)}\left(p^2\right)_{tt}+\fp,
	\end{equation}    
	where for the sake of mathematical generalization we also allow for the source term $\fp=\fp(x,t)$. Note that this equation can be further rewritten using $(p^2)_{tt}= 2p \ptt+2 \pt^2$ as
	\begin{equation}
		\begin{aligned}
			\frac{1}{b(\Theta)}	(1- 2 k(\Theta) p)p_{tt}-\frac{q(\Theta)}{b(\Theta)}\Delta p -  \frakK*\Delta  \pt = 2\frac{k(\Theta)}{b(\Theta)} \pt^2+\fp.
		\end{aligned}
	\end{equation}
	Looking at this form, we see that in the analysis one has to prove non-degeneracy (positivity) of the following terms: $b(\Theta)$, $1- 2k(\Theta) p$, and $q(\Theta)$.  These will be assumed to be positive in the linearized problem and then proven for the solution in the fixed-point argument using smallness of data and sufficient smoothness of the pressure and temperature fields.
	\\
	\indent  In deriving energy estimates for a linearized nonlocal Westervelt--Pennes system, we have to address several additional difficulties. First, the linearized system is decoupled and nonlocal in time. Therefore, we cannot take advantage of the dissipation coming from the parabolic Pennes equation, and, at the same time, the available test functions for the acoustic equations are restricted by the coercivity property of the memory kernel. Furthermore, as the dissipation induced by the kernel $\frakK$ is very weak,  we are forced  to estimate the coefficients $b(\Theta)$ and $q(\Theta)$ and their temporal or spatial derivatives in the $\LinfTLinf$ norm.  This requirement necessitates, for instance, controlling the norm    $\|\theta_ t\|_{\LinfLinf}$,  which by using  Sobolev embedding theorem mandates initial data in $\Hfour$ for the Pennes equation and requires the source term $\mathcal{G}$ (which depends on $p_t$) to be estimated in a more regular space. \\
	\indent To
	circumvent this difficulty, and in order to reduce the regularity of the initial data for the Pennes equation, we introduce \emph{time-weighted} energy estimates for the parabolic Pennes equation. Hence, instead of estimating   $\| \theta_ t\|_{\LinfLinf}$ we estimate  the time-weighted norm $\|\sqrt{t} \theta_ t\|_{\LinfLinf}$ which allows for the initial temperature data to be in $\Hthree$ rather than $\Hfour$. \\
	\indent	Having the linear estimates in hand, we can then employ the Banach fixed-point theorem to transfer the results to the nonlinear Westervelt--Pennes problem, provided that the pressure and temperature data are sufficiently small and smooth. This result is contained in Theorem~\ref{thm: West Pennes}. \\
	\indent Subsequently, we can analyze the equation for the concentration $c$. There we  rely on Ladyzhenskaya's inequality \eqref{Lady_Inequality_2} which helps with reducing the assumptions on the coefficients in \eqref{concentration equation} by requiring $\bfv$ to be in $L^4(0,T; L^4(\Omega))$ rather than $L^\infty(0,T; L^\infty(\Omega))$.  This reduction is helpful for the case $\frakK=\delta_0$ with the strong damping $- \Delta \pt$ (covered also by our theory although not in an optimal way), where the assumptions on the pressure-temperature data could be further reduced to $(p_0,p_1,\theta_0)\in  \left(\Honetwo\right)\times  \Honezero \times \left(\Honetwo\right)$ by exploiting the parabolic character of the pressure problem; see~\cite{wilkeL_pL_qTheory2023}. The main theoretical result of the work for the Westervelt--Pennes-advection-diffusion system is contained in Theorem~\ref{thm: West Pennes concentration}.
	
	\subsection*{Outline of the rest of the paper} The rest of the exposition is organized as follows. In Section~\ref{Sec:Preliminaries}, we introduce the modeling assumptions and present the main contribution of this work together with some background theoretical results that we rely on in the analysis.  In Section~\ref{sec: WestPennes}, we conduct the analysis of the Westervelt--Pennes system with nonlocal acoustic damping of time-fractional type. Section~\ref{sec: WestPennes concentration} is dedicated to the local well-posedness analysis of the full Westervelt--Pennes-advection diffusion system. Finally, in Section~\ref{sec: Numerics} we discuss the numerical discretization of the problem and illustrate  its properties and nonlinear effects with several numerical experiments. 
	
	\subsection*{Notation} Below we use $C>0$ to denote a generic constant which may change its value from one line to the next. We use $x \lesssim y$ to denote $x \leq Cy$. We write $\lesssim_T$ when the hidden constant depends on the final time $T$ in such a manner that it tends to $+\infty$ as $T \rightarrow +\infty$. We further introduce the following notation on smooth domains:
	\begin{equation} 
		\Honethree =\left\{p\in H^3(\Omega)\,:\, \mbox{tr}_{\partial\Omega} p = 0, \  \mbox{tr}_{\partial\Omega} \D p = 0\right\}.
	\end{equation}  
	 When writing norms in Bochner spaces, we omit the temporal domain $(0,T)$; for example, $\|\cdot\|_{L^p(L^q(\Omega))}$ denotes the norm  in $L^p(0,T; L^q(\Omega))$.  We use subscript $t$ to emphasize that the temporal interval is $(0,t)$ for some $t \in (0,T)$; for example, $\|\cdot\|_{L^p_t(L^q(\Omega))}$ denotes the norm  in $L^p(0,t; L^q(\Omega))$ for $t \in (0,T)$. 
	
	\section{Theoretical preliminaries, assumptions, and statement of the main theoretical result} \label{Sec:Preliminaries}
	
	In this section, we set the assumptions, state the main result, as well as some useful background results that are needed in the analysis. Throughout this work we assume that $\Omega \subset \R^d$ with $d \in \{1,2,3\}$ is a bounded domain that is $C^2$ regular.  \\
	\indent  We consider the following system for the pressure-temperature-concentration $(p, \Theta, c)$:
	\begin{subequations} \label{coupled_problem original}    
		\begin{equation} \label{coupled_problem original system}
			\left\{ \begin{aligned}
				&\frac{1}{b(\Theta)}p_{tt}-\frac{q(\Theta)}{b(\Theta)}\Delta p - \frakK* \Delta \pt = \frac{k(\Theta)}{b(\Theta)}\left(p^2\right)_{tt}+f_p, \qquad &&\text{in} \ \Omega \times (0,T), \\[1mm]
				& 	\rhoa \Ca\Theta_t -\kappaa\Delta \Theta+ \wb(\Theta) \rhob \Cb (\Theta-\Thetaa) = \tilde{\calG}(\pt, \Theta)+f_\Theta, \qquad &&\text{in} \ \Omega \times (0,T), \\[1mm]
				& c_t + \nabla \cdot \big(v(p, \nabla p) c\big)- \nabla \cdot \big(D(p, \Theta) \nabla c \big) = f_c, \qquad &&\text{in} \ \Omega \times (0,T).
			\end{aligned} \right.
		\end{equation}
		We couple \eqref{coupled_problem original system} with Dirichlet boundary conditions 
		\begin{eqnarray} \label{coupled_problem_original_BC}
			p\vert_{\partial \Om}=0, \qquad \Theta\vert_{\partial \Om}= \Thetaa, \qquad c\vert_{\partial \Om}=0,
		\end{eqnarray}
		and the initial data
		\begin{eqnarray} \label{coupled_problem_original_IC}
			(p, p_t)\vert_{t=0}= (p_0, p_1), \qquad \Theta \vert_{t=0}=\Theta_0, \qquad c\vert_{t=0}=c_0.
		\end{eqnarray}
	\end{subequations} 
	We next discuss and then set the assumptions on the involved temperature and pressure-dependent coefficients as well as the absorbed energy.
	
	\subsection{Assumptions on the memory kernel} Following~\cite{kaltenbacher2024limiting}, where the Westervelt equation with nonlocal damping of time-fractional type is analyzed, we consider the problem in an integro-differential form where we impose a set of assumptions on the memory kernel $\frakK$ which, among others, cover the case of having dissipation of Caputo--Djrbashian fractional type. More precisely, we impose the following regularity and coercivity assumptions on the memory kernel in \eqref{coupled_problem original system}. 
	
	\begin{assumption}\label{assumptions kernel}	We assume that the kernel satisfies
			\begin{equation}\label{assumption regularity kernel} 
				\frakK \in L^1(0,T) \cup \{\delta_0\}  .
			\end{equation}
				Furthermore, there exists $\CfrakK>0$, such that 
			\begin{equation}\label{assumption coercivity kernel} 
				\begin{aligned}
					\int_0^{t} \intO \left(\frakK* y \right)(s) \,y(s)\dxs\geq 	\CfrakK
					\int_0^{t} \|(\frakK* y)(s)\|^2_{\Ltwo} \ds,  \quad y\in L^2(0, T; \Ltwo),
				\end{aligned}
			\end{equation}
		 for all $t\in(0,T)$.	
\end{assumption} 
\indent The coercivity assumed in \eqref{assumption coercivity kernel} caters to the need to have $(-\Delta)^{\nu} p_t$ with $\nu \in \{0, 1 , \ldots \}$ as the test function for the wave equation. We allow for $\frakK= \delta_0$ in \eqref{assumption regularity kernel}, where $\delta_0$ is the Dirac delta distribution, to cover the case of having strong damping in the equation (i.e., $-  \Delta \pt$), 
although we focus our efforts in the analysis on $\frakK \in L^1(0,T)$. We use $\|\cdot\|_{\calM(0,T)}$ to denote the total variation norm which for the examples considered can be understood as: 
\begin{equation} \label{def_calM}
	\|\frakK\|_{\mathcal{M}(0,T)}=\begin{cases}
		1 &\text{ if }\frakK=  \delta_0,\\
		\|\frakK\|_{L^1(0,T)}&\text{ if }\frakK \in L^1(0,T).
	\end{cases}
\end{equation}
The coercivity assumption \eqref{assumption coercivity kernel} is there in order to achieve sufficient damping from the $\frakK$ term. 	As discussed in \cite{kaltenbacher2024limiting}, various classes of kernels satisfy these assumptions. A prominent example is the Abel kernel given in \eqref{Abel kernel}, which results in the time-fractional damping $-\Delta \Dtalpha p$ in the equation. Here $\Dtalpha v$ denotes the Caputo--Djrbashian derivative of order $\alpha \in (0,1)$:
\begin{equation} \label{def Caputo derivative}
	\Dtalpha v = \frac{1}{\Gamma(1-\alpha)} \intt (t-s)^{-\alpha} v_t \ds \quad \text{for }\ v \in W^{1,1}(0,T), \, t \in (0,T);
\end{equation}
see, e.g.,~\cite{kubicaTimefractionalDifferentialEquations2020} for details. Exponential kernels in the form of
\begin{equation} \label{exponential_kernel}
	\frakK(t) = \frac{1}{\tau} \exp\left(-\frac{t}{\tau}\right), 
\end{equation}
where $\tau>0$ is the relaxation time, also satisfy Assumption~\ref{assumptions kernel}. More generally, the kernels of Mittag-Leffler type given by
\begin{equation} \label{ML_kernels}
	\begin{aligned}
		\frakK = \left(\frac{1}{\tau}\right)^{a-b}\frac{1}{\tau^b}t^{b-1}E_{a,b}\left(-\Big(\frac{t}{\tau}\Big)^a\right)
	\end{aligned}
\end{equation}
 fit within the theoretical framework of this work, where the generalized Mittag-Leffler function is given by 
\begin{align} \label{def_MittagLeffler}
	E_{a, b}(t) = \sum_{k=0}^\infty \frac{t^k}{\Gamma(a k + b)}, \qquad a > 0,\ t,\, b \in \mathbb{R};
\end{align}
see, e.g.,~\cite[Ch.\ 2]{kubicaTimefractionalDifferentialEquations2020}. For details on verification of assumed properties for these kernels, we refer to~\cite[Sec.\ 4]{kaltenbacher2024limiting}.

	\subsection{Assumptions on the temperature-dependent parameters and absorbed energy}\label{Section_Assumptions}
	 In the modeling literature (see, for example~\cite{bilaniukSpeedSoundPure1993}), the functions $q$ and $b$ are typically fitted to data assuming a polynomial dependency on the temperature $\Theta$. That is, these functions usually have the following form:
	\begin{equation} 
		\begin{aligned}
			q(s) =\, \sum_{i=0}^{m_1} q_i s^{i}, \qquad b(s) =\, \sum_{i=0}^{m_2} b_i s^{i},
		\end{aligned}
	\end{equation}
	for some $m_1$,  $m_2\in \N_0$ and $q_i$, $b_i \in \R$.
	 The coefficient of nonlinearity in the Westervelt equation is given by
	\begin{equation} \label{def_k}
		k(\Theta) = \frac{\beta_{a}}{\rhoa q(\Theta)},
	\end{equation}
	where $\betaa$ and $\rhoa$ are the nonlinearity parameter and (as mentioned) the medium density, respectively. Concerning the blood perfusion $\wb$, several different forms appear in the literature.  It is often assumed to be linear~\cite{dengParametricStudiesPhase2000} or higher-order polynomial~\cite{tompkinsnTemperaturedependentConstantrateBlood1994}: 
	\begin{subequations}\label{Assumption_omega}  
		\begin{equation}\label{wb_form_1}
			\wb(s) = a_1 + a_2 s+ \ldots a_{l} s^l, \ l \in \N;
		\end{equation}
		it also appears in a power-law form~\cite{tompkinsnTemperaturedependentConstantrateBlood1994}
		\begin{equation}
				\wb(s) = a_1 + a_2 (s-a_3)^{a_4}, \ a_{1,2,3,4} \in \R, a_4>1;
		\end{equation}
		or as the Gaussian function of the temperature in certain temperature intervals 
		\begin{equation}\label{wb_form_2}
			\wb(s) = a_1+a_2 e^{-a_3 (s-a_4)^2}, \quad s \in [s_0, s_1],
		\end{equation}
	\end{subequations}
	and constant in others; see~\cite{langImpactNonlinearHeat1999}. For the purposes of the analysis, we need sufficient smoothness of these functions in the sense of local Lipschitz continuity.
\begin{assumption}\label{assumptions: q b k wb calG}
	We make the following assumptions on the temperature-dependent parameters:
	\begin{itemize}
		\item $q$, $b \in \Cthreeoneloc(\R)$ and 
		$
		k(s) =\frac{\upsilon}{q(s)} 
		$ for some $\upsilon \in \R$; \vspace*{1mm}
		\item $q(\Thetaa)>0$ and $b(\Thetaa)>0$; \vspace*{1mm}
		\item $\wb \in \Ctwooneloc(\R)$.
	\end{itemize}
	
\end{assumption} 
The conditions on the positivity of $q(\Thetaa)$ and $b(\Thetaa)$ correspond to the usual assumptions on the positivity of these parameters in a setting of constant (that is, ambient) temperature. \\
\indent The function $\tilde{\calG}$ in the Pennes equation models the absorbed acoustic energy. We assume it has the following form, as commonly done in the literature; see, e.g.\change{~\cite{norton2016westervelt}}.
\begin{assumption}\label{assumtion calG} The absorbed acoustic energy is given by 
\begin{equation}\label{calG form}
		\tilde{\calG} (\pt, \Theta)= \tilde{\zeta} \frakG(\Theta)  \pt^2 \quad \text{with } \quad  \frakG(\Theta) =\frac{b(\Theta)}{q^2(\Theta)} \ \text{for some } \ \tilde{\zeta}>0.
	\end{equation}
\end{assumption}

\noindent	We remark that in the literature also time-averaged versions of this functional appear in the form of
	\begin{equation}\label{Q_Form_2}
		\tilde{\calG} (\pt, \Theta)=   \frac{\tilde{\zeta}}{t_2-t_1} \int_{t_1}^{t_2} \frakG(\Theta) \pt^2(s) \ds, 
	\end{equation}  where $0 \leq t_1 < t_2 \leq T$; see, e.g.,~\cite{connorBioacousticThermalLensing2002}. The theory presented in this work can be extended in a straightforward manner to this setting as well, however for the sake of streamlining the arguments we assume here the form in \eqref{calG form}.
	\subsection{The diffusion matrix and the transport velocity}   In the advection-diffusion equation for the concentration $c$, we make the following assumptions on its pressure- and temperature-dependent functions.
	\begin{assumption} \label{assumptions: D v} 
	We assume that the coefficient $D$ can be decomposed into a diffusive and a perturbation component as follows:
	\begin{equation}
	D(p, \Theta) \equiv 	D(p, \theta+\Thetaa)= D_0 + \calD(p, \theta),
\end{equation}
 where there exists $d_0>0$ such that for all $\mathbf{u}\in \R^d$, 
				\begin{equation} \label{positivity D0}
						\int_\Omega D_0 \mathbf{u}\cdot \mathbf{u}\dx\geq d_0\int_\Omega |\mathbf{u}|^2\dx,
					\end{equation}
	and $\calD \in  \Czerooneloc(\R \times \R)$ with $\calD(0,0)=0$. \\
	\indent Concerning the velocity $\bfv$, we assume that its components $v_i$ are bounded in the following sense
	\begin{equation}
		\begin{aligned}
			|v_i(z_1, z_2) | \leq C\big(1+ |z_1|+|z_2|\big), \quad\text{for } \ i=1, \ldots, d.
		\end{aligned}
		\end{equation} 		

	\end{assumption}
	These assumptions are in accordance with those made for in available studies of linear models of drug delivery; see, e.g.,~\cite{ferreiraDrugDeliveryEnhanced2022}, where $D_0$ is the identity matrix.  
	
	\subsection{Rewriting the problem} We next introduce the change of variables $\theta = \Theta-\Thetaa$ and divide the Pennes equation by $\rhoa \Ca$ to arrive at the scaled problem 
	\begin{subequations} \label{coupled_problem West Pennes concentration}    
		\begin{equation} \label{coupled_system West Pennes concentration} \tag{W--P-c}
			\boxed{ 
				\begin{aligned}
					&\ \frac{1}{b(\theta+\Thetaa)} p_{tt}-\frac{q(\theta+\Thetaa)}{b(\theta+\Thetaa)}\Delta p -  \frakK*\Delta \pt =\frac{ k(\theta+\Thetaa)}{b(\theta+\Thetaa)}\left(p^2\right)_{tt}+f_p \ &&\text{in} \ \Omega \times (0,T), \\[2mm]
					&\ \theta_t -\kappa\Delta \theta+ \nu \wb(\theta+\Thetaa)  \theta = \calG(\pt, \theta+\Thetaa)+f_\theta \ &&\text{in} \ \Omega \times (0,T), \\[2mm]
					&\ c_t + \nabla \cdot \big(\bfv(p, \nabla p) c\big)- \nabla \cdot \big(D(p, \theta+\Thetaa) \nabla c\big) = f_c \ &&\text{in} \ \Omega \times (0,T),
				\end{aligned} 
			}
		\end{equation}
		with
		\begin{equation}
			f_\theta = \frac{1}{\rhoa \Ca} f_{\Theta-\Thetaa}, \quad	\nu = \frac{\rhob \Cb}{\rhoa \Ca}, \quad \ \kappa = \frac{\kappaa}{\rhoa \Ca}, \quad \ \calG = \frac{1}{\rhoa \Ca} \tilde{\calG}. 
		\end{equation} 
		Now we have homogeneous Dirichlet boundary conditions
		\begin{equation} \label{coupled_problem_BC_WPc} \tag{W--P-c bc}
			p\vert_{\partial \Om}=0, \qquad \theta\vert_{\partial \Om}= 0, \qquad c\vert_{\partial \Om}=0,
		\end{equation}
		and the initial data 
		\begin{equation} \label{coupled_problem_IC_WPc}\tag{W--P-c ic}
			(p, p_t)\vert_{t=0}= (p_0, p_1), \qquad \theta \vert_{t=0}=\Theta_0- \Thetaa:= \theta_0, \qquad c\vert_{t=0}=c_0.
		\end{equation}
	\end{subequations}

	\subsection{Main theoretical result}  Next, we state and discuss the main theoretical result of the present work before proceeding to its proof carried out in several stages. 
To this end, we introduce the solutions space for the pressure by
\begin{equation} \label{def_Xp}
	\begin{aligned}
		\mathcal{X}_p 
		= \Big\{
		p \in L^{\infty}(0, T; \Honethree):\, &p_t \in L^{\infty}(0, T;  \Honetwo), \\
		&\, \ptt \in \LtwoTHonez, \\
		&\,  \frakK*\nabla \D \pt \in \LtwoTLtwo \Big\},
	\end{aligned}
\end{equation}
with the corresponding norm
\begin{equation}
	\begin{aligned}
		\|p\|_{\Xp} = \left(\Vert p\Vert^2_{L^\infty(\Hthree)}+\Vert p_t\Vert_{L^\infty(\Htwo)}^2 +\|\ptt\|^2_{\LtwoHone}+ \|\frakK* \nabla \D \pt\|^2_{\LtwoLtwo}\right)^{1/2}.
	\end{aligned}
\end{equation}	
Further, the solution space for the temperature is given by 
	\begin{equation} \label{def_Xtheta}
	\begin{aligned}  
		\mathcal{X}_\theta = \Big\{ \theta \in  \LinfTHonethree:\quad \thetat \in&\ \LtwoTHonetwo, \\
		\sqrt{t}\thetat \in& \LinfTHonetwo  \Big\},
	\end{aligned}
\end{equation}
with the norm
\begin{equation}
	\|\theta\|_{\Xtheta} = \left(\|\theta\|^2_{\LinfHthree}+\|\theta_t\|^2_{\LtwoHtwo}+\|\sqrt{t}\theta_t\|^2_{\LinfHtwo}\right)^{1/2}.
\end{equation}
Thirdly, the solution space for the concentration is given by
\begin{equation} \label{def_Xc}
	\begin{aligned}
		\Xc = \LinfTLtwo \cap \LtwoTHonez \cap W^{1,1}(0,T; \Hneg),
	\end{aligned}
\end{equation}
with the norm
\begin{equation}
	\|c\|_{\Xc} = \left(\|c\|^2_{\LinfLtwo}+\|\nabla c\|^2_{\LtwoLtwo}+\|c_t\|^2_{L^1(\Hneg)} \right)^{1/2}.
\end{equation}
Then the solution space for the multiphysics system is given by
\begin{equation} \label{def_Xpthetac}
	\solspaceWestPennesconcentr = \Xp \times \Xtheta \times \Xc,
\end{equation}
with the norm
\begin{equation}
	\|(p, \theta, c)\|_{\solspaceWestPennesconcentr} = \left( \|p\|^2_{\Xp}+\|\theta\|^2_{\Xtheta}+\|c\|^2_{\Xc} \right)^{1/2}.
\end{equation}
We can now define the notion of the solution in $\solspaceWestPennesconcentr$.
\begin{definition}\label{def solution} 
	The triplet $(p, \theta, c) \in \solspaceWestPennesconcentr$ is a solution of \eqref{coupled_system West Pennes concentration}, \eqref{coupled_problem_BC_WPc}, \eqref{coupled_problem_IC_WPc} if it satisfies 
	\begin{equation}
		\begin{aligned}
	\intTO \left(\frac{1}{b(\theta+\Thetaa)}p_{tt}-\frac{q(\theta+\Thetaa)}{b(\theta+\Thetaa)}\Delta p -  \frakK*\Delta p_t - \frac{k(\theta+\Thetaa)}{b(\theta+\Thetaa)}\left(p^2\right)_{tt}-f_p \right) \phi_p \dxs= 0,		
		\end{aligned}	
	\end{equation}
	as well as
	\begin{equation}
		\begin{aligned}
			\intTO \left(\theta_t -\kappa\Delta \theta+  \wb(\theta+\Thetaa) \nu \theta - \calG(\pt, \theta+\Thetaa)-f_\theta\right) \phi_\theta \dxs=0,
		\end{aligned}
	\end{equation}
	and
	\begin{equation}
		\begin{aligned}
			\begin{multlined}[t]\int_0^T \langle \ct, \phi_c\rangle_{\Hneg, \Honezero} \ds +\intTO  \bfv (p, \nabla p) c \cdot \nabla \phi_c\dxs\\\hspace*{4cm}+ \intTO D(p, \theta+\Thetaa) \nabla c \cdot \nabla \phi_c \dxs
			=\, \intTO f_c \phi_c \dxs \end{multlined}
		\end{aligned}
	\end{equation}
	for all $\phi_p \in \LtwoTLtwo$, $\phi_\theta \in \LtwoTLtwo$, and $\phi_c \in L^\infty(0,T; \Honezero)$, with
	\begin{equation}
		(p, p_t)\vert_{t=0}= (p_0, p_1), \qquad \theta \vert_{t=0}= \theta_0, \qquad c\vert_{t=0}=c_0.
	\end{equation}
	
\end{definition}
\noindent We are now ready to state the main result.
\begin{theorem} \label{thm: West Pennes concentration}
	Let Assumptions~\ref{assumptions kernel}--\ref{assumptions: D v} hold. Let $\fp$, $\ftheta \in \HoneTHonezero$ and $\fc \in \LoneTLtwo$. Furthermore, let 
	\begin{equation}
		(\pzero, p_1) \in \Honethree \times \left(\Honetwo \right), \quad \thetazero \in \Honethree, \quad c_0 \in \Ltwo.
	\end{equation}
	There exist $\deltap>0$ and $\deltatheta>0$  small enough, such that if
	\begin{equation}
		\begin{aligned}
			&	\|\pzero\|_{\Hthree} + \|p_1\|_{\Htwo} +\|\fp\|_{\HoneHone}\leq \deltap,\\
			&		\|\thetazero\|_{\Hthree} + \|\ftheta\|_{\HoneHone}\leq \deltatheta,
		\end{aligned}
	\end{equation}
	then the system \eqref{coupled_system West Pennes concentration}, \eqref{coupled_problem_BC_WPc}, \eqref{coupled_problem_IC_WPc} has a unique solution in the sense of Definition~\ref{def solution}. The solution satisfies
	\begin{equation}
		\begin{aligned}
			\|(p, \theta, c)\|_{\solspaceWestPennesconcentr} \lesssim\, &\, \begin{multlined}[t]
				\|\pzero\|_{\Hthree} + \|p_1\|_{\Htwo} +\|\thetazero\|_{\Hthree} +\|c_0\|_{\Ltwo}\\\hspace*{1cm}+\|\fp\|_{\HoneHone}+ \|\ftheta\|_{\HoneHone}+ \|\fc\|_{\LoneLtwo}.
			\end{multlined}
		\end{aligned}
	\end{equation}
\end{theorem}

\noindent Before proceeding to the details of the analysis, let us first discuss the main result.
\begin{itemize}[leftmargin=*]
	\item The relatively high regularity assumptions on the pressure and temperature data are due to the weak damping in the acoustic equation. Indeed, the regularity assumptions made here on the initial pressure conditions $(p_0, p_1)$ agree with those for the  fractionally damped Westervelt equation with constant coefficients; see~\cite{kaltenbacher2024limiting, kaltenbacherInverseProblemNonlinear2021, bakerNumericalAnalysisTimestepping2024}. \vspace*{1mm}	 
\item  The time weights play a significant role in reducing the needed regularity of solutions of the Pennes equation (that is, in optimizing the assumptions on $\theta_0$).  In fact, for $\theta_0\in H^3(\Omega)$, we obtain 
\[
\sqrt{t}\thetat \approx \sqrt{t} \Delta \theta \in \LinfTHonetwo;
\]
see also Proposition~\ref{prop: linear temperature higher order} below.  This enhancement in integrability plays a vital role in ``closing" the fixed-point argument for the Westervelt--Pennes sub-system in Section~\ref{sec: WestPennes}. 
 We do not employ the time-weighted energy estimate for the pressure equation since it is hyperbolic and the method would not enhance the integrability of $p$.\vspace*{1mm}
\item  The requirement for the initial data to be small  arises naturally within the proof of the local well-posedness, and it is related to the non-degeneracy issue that characterizes the Westervelt equation and other nonlinear  acoustic wave models (see~\cite{nikolicLocalWellposednessCoupled2022, kaltenbacherParabolicApproximationQuasilinear2022, kaltenbacherGlobalExistenceExponential2009}). Here, however, degeneracy needs to be avoided not only in the term $1-2 k(\Theta)p$ but also in the coefficients $q=q(\Theta)$ and $b=b(\Theta)$ and thus the smallness is imposed not only on the pressure but also the temperature data. We point out that no smallness assumption on the final time $T$ is made in Theorem~\ref{thm: West Pennes concentration}. We expect, however, that one could relax the smallness assumption on data to involve a lower-order norm by requiring also $T$ to be small, as done in the analysis of the fractional Westervelt equation with constant coefficients in~\cite{kaltenbacher2024limiting}. \vspace*{1mm}
	\item Note that, although Theorem~\ref{thm: West Pennes concentration} covers the case $\frakK= \delta_0$ where the acoustic damping is $-\Delta \pt$, the assumptions on data are not optimal in that setting where one could exploit the parabolic character of the pressure problem as well. For the analysis of the Westervelt--Pennes system with strong acoustic damping using a maximal $L^p$ regularity approach, we refer to~\cite{wilkeL_pL_qTheory2023}. \vspace*{1mm}
	\item We note also that with more smoothness assumed of the source function $\fc$, one could have more regular concentration field $c=c(x, t)$ as the velocity and diffusion terms are quite smooth thanks to the high regularity of the pressure and temperature fields.
\end{itemize}

	\subsection{Background embeddings and useful inequalities} Before proceeding to the proof of Theorem~\ref{thm: West Pennes concentration}, we here several helpful results on which we will often rely. In the analysis, we frequently use the embeddings  $\Htwo \hookrightarrow \Linf$ and $\Hone \hookrightarrow L^p(\Om)$, where $1 \leq p \leq 6$.  We also recall Ladyzhenskaya's inequality for functions  $u \in H^1(\Om)$:
	\begin{equation}\label{Lady_Inequality}
		\| u\|_{L^4(\Omega)} \leq  \CL \|u\|_{H^1(\Omega)}^{d/4}\|u\|_{L^2(\Omega)}^{1-d/4}, \qquad 1\leq d\leq 4.
	\end{equation}
\noindent	Lastly, we state here the precise form of \Gronwall's inequality that we use below.
	\begin{lemma} \label{lemma: Gronwall}
		Let $u$, $v \in L^\infty(0,T)$ be almost everywhere non-negative functions that satisfy the integral inequality
		\begin{equation} \label{assumption: Gronwall_ineq}
			u(t) + v(t) \leq a_1+\intt a_2(s) u(s) \ds \quad \text{for a.e.} \ t \in [0,T],
		\end{equation}  
		where $a_1 \geq 0$ and $a_2 \in L^1(0,T)$ is an almost everywhere  non-negative function. Then
		\begin{equation}\label{Gronwall_ineq}
			\begin{aligned}
				u(t) + v(t) \leq a_1e^{\int_0^t a_2(s)\ds} \quad \text{for a.e.} \ t \in [0,T].
			\end{aligned}
		\end{equation}
	\end{lemma}
	\begin{proof}  
		Let $w = u+v$. By \eqref{assumption: Gronwall_ineq} and the fact that $v$ is non-negative, we have
		\begin{equation} \label{Gronwall_ineq_w}
			w(t) \leq a_1+\intt a_2(s) w(s) \ds \quad \text{for all } \ t \in [0,T].
		\end{equation}
		The statement then follows by applying \cite[Lemma 1]{dautrayMathematicalAnalysisNumerical2000}.
	\end{proof}
	
	\section{Analysis of the nonlocal Westervelt--Pennes system} \label{sec: WestPennes}   
Toward the proof of Theorem~\ref{thm: West Pennes concentration}, we first recall that the concentration $c$ does not appear in the wave-heat sub-system. Therefore, we can treat the equation for $c$ sequentially after first analyzing the Westervelt and Pennes equations in \eqref{coupled_system West Pennes concentration}.	In this section, we thus perform the analysis of the Westervelt--Pennes subsystem:   
	\begin{subequations} 
		\begin{equation} \label{coupled_problem West Pennes} \tag{W--P}
			\boxed{
				\begin{aligned}
					&\ \frac{1}{b(\theta+\Thetaa)} p_{tt}-\frac{q(\theta+\Thetaa)}{b(\theta+\Thetaa)}\Delta p -  \frakK*\Delta p_t = \frac{k(\theta+\Thetaa)}{b(\theta+\Thetaa)}\left(p^2\right)_{tt}+f_p \ &&\text{in} \ \Omega \times (0,T), \\[2mm]
					&\ \theta_t -\kappa\Delta \theta+  \wb(\theta+\Thetaa) \nu \theta = \calG\big(\pt, b(\theta+\Thetaa), q(\theta+\Thetaa)\big)+f_\theta \ &&\text{in} \ \Omega \times (0,T),
				\end{aligned} 
			}
		\end{equation}
		supplemented with homogeneous Dirichlet boundary conditions  
		\begin{equation} \label{coupled_problem_BC_WP} \tag{W--P bc}
			p\vert_{\partial \Om}=0, \qquad \theta\vert_{\partial \Om}=0,
		\end{equation}  
		and the initial data
		\begin{equation} \label{coupled_problem_IC_WP} \tag{W--P ic}
			(p, p_t)\vert_{t=0}= (p_0, p_1), \qquad \theta \vert_{t=0}=\theta_0. 
		\end{equation} 
	\end{subequations}
	The analysis is conducted through a fixed-point approach for the Westevelt--Pennes equations; for this purpose we need a suitable linearization of the pressure-temperature system. 
		\subsection{Bounds for the linearized  problem}	The linearized problem is obtained by linearizing the quadratic pressure terms and ``freezing" the temperature-dependent parameters $b$, $q$, and $\wb$ in the Westervelt and Pennes equations:
	\begin{equation} \label{linearized WP system} \tag{W--P lin}
		\boxed{
			\begin{aligned}  
				&\frac{1}{\bstar(x,t)} p_{tt}-\frac{\qstar(x,t)}{\bstar(x,t)}\Delta p -  \frakK*\Delta \pt = 2\frac{\kstar(x,t)}{\bstar(x,t)}\left(\pstar \ptt+ \pstart \pt\right)+\fp , \\[1mm]
				& \theta_t -\kappa\Delta \theta+  \nu \wbstar(x,t)  \theta = \calG +\ftheta . 
			\end{aligned} 
		}
	\end{equation}
With this linearization, the linearized pressure equation is decoupled from the heat equation, which can be analyzed in a subsequent step.  The precise assumptions on $\bstar$, $\qstar$, $\pstar$, and $\wbstar$ will be made in the fixed-point analysis.
For the moment, it is sufficient to consider the following general form of the linear acoustic problem with variable coefficients: 
	\begin{equation}\label{ibvp_linWest}
		\left\{
		\begin{aligned}
			&\frakm(x,t)\ptt - \frakn(x,t) \Delta p - \frakK* \Delta \pt -\frakl \pt = \fp,\\
			&(p, p_t)\vert_{t=0}= (p_0, p_1),\\
			&p|_{\partial\Omega}=0.
		\end{aligned}
		\right. 
	\end{equation}
We conduct the linear acoustic investigations under the following non-degeneracy assumption on the variable coefficients: there exist $\ulfrakm$, $\olfrakm$, $\ulfrakn$, and $\olfrakn$ such that
	\begin{equation}\label{nondegeneracy assumption frakm frakn}
	0<\ulfrakm \leq \frakm(x,t)\leq \olfrakm, \quad 0<\ulfrakn \leq \frakn(x,t)\leq \olfrakn
	\end{equation}
	for all $(x,t)\in \Omega  \times (0,T)$, as well as under sufficient regularity of these coefficients; see~\eqref{regularity frakm frakn frakl}.
	 Well-posedness analysis of linear acoustic problems of the form  in \eqref{ibvp_linWest} with  space-time variable coefficients   can be found in the recent works on fractional acoustic equations; see, e.g.,~\cite{kaltenbacher2024limiting}. Compared to the existing works, here we need to incorporate time weights in the definition of spaces to which $\frakm$ and $\frakn$ belong as a preparation for the later analysis of the nonlinear system.

\begin{proposition} \label{prop: linear pressure higher order}
	Given $T>0$, let $\fp \in H^1(0,T; \Honezero)$ and
	\[
	(p_0, p_1) \in \Honethree \times \left(\Honetwo \right). 
	\]   Furthermore, assume that the functions $\frakm$, $\frakn$, and $\frakl$ satisfy 
		\begin{equation} \label{regularity frakm frakn frakl}
					\begin{aligned}
							&\frakm  \in  \Xfrakm:= L^\infty(0,T; \Honethree),\qquad \sqrt{t}\frakm_t \in \LinfTLinf,\\ 
							& \frakn \in  \Xfrakn:= L^\infty (0,T; \Honethree),\qquad \sqrt{t}\frakn_t \in \LinfTLinf,\\
							& \frakl \in \Xfrakl:=  L^\infty(0,T; \Honetwo),
						\end{aligned}
				\end{equation}
	and that the non-degeneracy assumptions for $\frakm$ and $\frakn$ in \eqref{nondegeneracy assumption frakm frakn} hold.  
	Then there is a unique $p$ in $\mathcal{X}_p$ which solves   \eqref{ibvp_linWest}. This solution satisfies   the following estimate: 
	\begin{equation} \label{est_p_linear_1}
		\begin{aligned}
			\|p\|_{\mathcal{X}_p} 	\lesssim\, &\,
			\begin{multlined}[t]
			   \big(\|p_0\|_{\Hthree}+\| p_1\|_{\Htwo}
				+ \|\fp\|_{H^1(\Hone)} \big)   \exp \big(C(T)\Lambda_1(\frakm, \frakl, \frakn)\big),
			\end{multlined}
		\end{aligned}
	\end{equation}
		where the function $\Lambda_1$ is given by 
		\begin{equation} \label{def Lambda1}
				\begin{aligned}
				\begin{multlined}[t]
				\Lambda_1(\frakm, \frakl, \frakn) =\,  1+\|\frakn\|_{\Xfrakn}+\calL^2(\frakm, \frakn, \frakl)
					+ \int_0^T \frac{1}{\sqrt{t}}\Big( \|\sqrt{t}\frakm_t\|_{\Linf} +\|\sqrt{t}\frakn_t\|_{\Linf}\Big)\dt
					\end{multlined}
					\end{aligned}
			\end{equation}
			with
			\begin{equation}
			\calL(\frakm, \frakn, \frakl) =   \|\frakm\|_{\Xfrakm}\big(1+\|\frakm\|_{\Xfrakm}\big)\big(1+\|\frakn\|_{\Xfrakn}+\|\frakl\|_{\Xfrakl}\big).
			\end{equation}
\end{proposition}

\begin{proof}
The proof is carried out using a Faedo--Galerkin discretization in space combined with energy estimates, analogously to~\cite[Proposition 3.1]{kaltenbacher2024limiting}. We approximate the solution by
\begin{equation}
	\begin{aligned}
		\pn(x,t) =\sum_{i=1}^n \xi^n_i(t)v_i(x),
	\end{aligned}
\end{equation}
where $\{v_i\}_{i=1}^\infty$ are smooth eigenfunctions of the Dirichlet-Laplacian operator. Fix $n\geq 1$ and let  $V_n=\textup{span}\{v_1, \ldots, v_n\}$. The semi-discrete acoustic problem is then given by
\begin{equation} \label{semi-discrete}
	\begin{aligned}
		\begin{multlined}[t] \big(\frakm \pntt- \frakn \D \pn - \frakK*\D \pnt-\frakl \pnt-\fp, v_j\big) = 0 \end{multlined}
	\end{aligned}
\end{equation}
for all $j=1, \ldots, n$, with approximate initial conditions $(\pn, \pnt)\vert_{t=0}=(p_{0}^n, p_{1}^n)$ taken as $\Ltwo$ projections of $(p_0, p_1)$ onto $V_n$. 
With $\boldsymbol{\xi}=[\xi^n_1 \ \xi^n_2 \ \ldots \ \xi^n_n]^T$, the approximate problem can be rewritten in matrix form
\begin{equation} \label{matrix_eq}
	\begin{aligned}
		\Mfrakm(t)\boldsymbol{\xi}_{tt}+ \Nfrakn(t) \boldsymbol{\xi}+ \mathbb{K}\, \frakK*\boldsymbol{\xi}_t+\Lfrakl(t)\bfxi_t= \boldsymbol{f},
	\end{aligned}
\end{equation}
where the entries of matrices $\Mfrakm(t)=[\Mfrakmij]$, $\Nfrakn(t)=[\Nfraknij]$, $\mathbb{K}=[\mathbb{K}_{ij}]$, $\Lfrakl=[\Lfraklij]$, and vector $\boldsymbol{f}=[\boldsymbol{f}_{i}]$ are given by
\begin{equation} \label{matrices}
	\begin{aligned}
		& \Mfrakmij(t)= (\frakm  v_i, v_j)_{L^2}, \quad && \Nfraknij(t)= -(\frakn \Delta  v_i, v_j)_{L^2},   \\
		& \mathbb{K}_{ij}= -(\Delta v_i,  v_j)_{L^2},\quad  && \Lfraklij(t)= - (\frakl v_i, v_j)_{L^2}, \hspace{1.5cm} \boldsymbol{f}_i(t)=(\fp(t), v_i)_{L^2}.
	\end{aligned}
\end{equation}
We introduce the vectors of coordinates of the approximate initial data in the basis:
\[
\boldsymbol{\xi}_0=[\xi^n_{0,1}, \ \xi^n_{0,2}, \ \ldots \ ,\xi^n_{0,n}]^T, \quad \boldsymbol{\xi}_1=[\xi^n_{1,1}, \ \xi^n_{1,2}, \ \ldots \ ,\xi^n_{1,n}]^T.
\]
By setting $\bfmu=\boldsymbol{\xi}_{tt}$, we have
\begin{equation} \label{eq_xi}
	\bfxi_t(t)=1*\bfmu+\bfxi_1, \quad	\boldsymbol{\xi}(t)=\boldsymbol{\xi}_0 +t \boldsymbol{\xi}_1+1*1*\bfmu.
\end{equation}
Therefore, the semi-discrete problem can be rewritten as
\begin{equation}
	\begin{aligned}
		\begin{multlined}[t]
			\Mfrakm(t)\bfmu + \Nfrakn\left(\boldsymbol{\xi}_0 +t \boldsymbol{\xi}_1+1*1*\bfmu\right)+ \mathbb{K}\,\frakK *(1*\bfmu+\bfxi_1)+ \Lfrakl(t)(1*\bfmu+\bfxi_1) =\boldsymbol{f}. \end{multlined}
	\end{aligned}
\end{equation}
Note that $\Mfrakm \in L^\infty(0,T)$ is positive definite due to assumption \eqref{nondegeneracy assumption frakm frakn} on $\frakm$. Thanks to this, this matrix equation can be seen as a matrix Volterra integral equation of second kind: 
\begin{equation}
	\begin{aligned}
		\bfmu+ \boldsymbol{K} * \bfmu = \boldsymbol{\tilde{f}},
	\end{aligned}
\end{equation}
with
\[
\boldsymbol{K}= \Mfrakm^{-1}(t)\big(\Nfrakn 1*1 + \mathbb{K}\,\frakK *1+\Lfrakl(t)1\big)
\]
and
\[
\boldsymbol{\tilde{f}} = \Mfrakm^{-1}(t)\big(-\Nfrakn \left(\boldsymbol{\xi}_0 +t \boldsymbol{\xi}_1\right)- \mathbb{K}\,\frakK *\bfxi_1-\Lfrakl(t) \bfxi_1+\boldsymbol{f}(t)\big).
\]
The existence theory for systems of Volterra integral equations of the second kind~\cite[Ch.\ 2, Theorem 4.5]{gripenberg1990volterra} yields a unique solution $\bfmu \in L^\infty(0,T)$. From $\bfxi_{tt}=\bfmu$, we then conclude that there exists a unique $\bfxi \in W^{2, \infty}(0,T)$ and thus $\pn \in W^{2, \infty}(0,T; V_n)$. \\
\indent In the next step of the proof, we need a uniform bound on $\pn$. To this end, testing is conducted using the test functions $\Delta^2 \pt^n$, $- \Delta \ptt^n$, and $\pntt$ as in the analysis of the problem with $\frakn=1$ in~\cite[Proposition 3.1]{kaltenbacher2024limiting}. The two main differences compared to~\cite{kaltenbacher2024limiting} lie in handling the additional terms arising from the fact that $\frakn \neq $const.\ and in introducing the time weights when estimating the $\frakm$ and $\frakn$ terms. This approach leads to the following bound on the approximate pressure:
	\begin{equation} \label{est pn}
	\begin{aligned}
		\|\pn\|_{\mathcal{X}_p} 	\lesssim&\, \begin{multlined}[t]  \left(\|p_0\|_{\Hthree}+\| p_1\|_{\Htwo} 
			+ \|\fp\|_{H^1(\Hone)} \right)   \exp \left(C(T)\Lambda_1(\frakm, \frakl, \frakn)\right), \end{multlined}
	\end{aligned}
\end{equation}
with the function $\Lambda_1$ defined in \eqref{def Lambda1}.  For completeness, we present the detailed derivation of \eqref{est pn} in Appendix~\ref{Appendix: linear pressure estimates}. The $n$-uniform bounds in turn allows us to extract a weakly convergent subsequence of $\{\pn\}_{n \geq 1}$ whose limit can be shown to be a solution of the original problem using standard arguments; see, e.g.,~\cite{evansPartialDifferentialEquations2010}.\\
\indent  Uniqueness follows by proving that the only solution of the homogeneous problem (that is, with $\fp=0$ and $p_0=p_1=0$) is zero. This can be shown by testing this problem with $\pt$. Note that by assumption \eqref{assumption coercivity kernel}, we have $\intt (\frakK*\nabla \pt, \nabla \pt) \geq 0$. Therefore, testing
\[
\frakm(x,t)\ptt - \frakn(x,t) \Delta p - \frakK* \Delta \pt -\frakl \pt =0
\]
with $\pt$, and using the fact that $p\vert_{t=0} = p_t \vert_{t=0}=0$ in the homogeneous problem yields
\begin{equation} \label{est uniqueness}
	\begin{aligned}
		\|\sqrt{\frakm} \pt(t)\|^2_{\Ltwo} + \|\sqrt{\frakn} \nabla p(t)\|^2_{\Ltwo}
		\leq\, &\, \frac12 \inttO \frakm_t \pt^2 \dxs - \frac12 \inttO \nabla \frakn \cdot (\nabla p) \pt \dxs\\
		& + \frac12 \inttO\frakn_t |\nabla p|^2 \dxs + \inttO \frakl \pt^2 \dxs.
	\end{aligned}
\end{equation}
Next, on account of the assumed regularity of $\frakm$, we have
	\begin{equation} \label{est frakmt}
	\begin{aligned}
		\frac12 \int_0^t (\frakm_t \pt,  \pt)_{L^2}\ds\,
		\leq\,&\, \frac12 \intt \|\frakm_t\|_{\Linf}  \| \pt\|^2_{\Ltwo} \ds\\
		\leq\,&\, \frac12\intt \frac{1}{\sqrt{s}} \|\sqrt{s}\frakm_t\|_{\Linf}  \| \pt\|^2_{\Ltwo} \ds\\
		\leq\,&\, \frac12\|\sqrt{t}\frakm_t\|_{\LinfLinf} \intt \frac{1}{\sqrt{s}}   \| \pt\|^2_{\Ltwo} \ds.
	\end{aligned}
\end{equation} 
Note that $s^{-1/2} \in L^1(0,t)$ since $\intt \frac{1}{\sqrt{s}}  \ds = 2 \sqrt{t} \leq 2 \sqrt{T}$. Similarly, we have 
\[
 \displaystyle \frac12 \inttO\frakn_t |\nabla p|^2 \dxs \leq \frac12 \|\sqrt{t}\frakn_t\|_{\LinfLinf} \intt \frac{1}{\sqrt{s}}   \| \nabla p\|^2_{\Ltwo} \ds.
 \]
The remaining terms on the right-hand side of \eqref{est uniqueness} can be bounded as follows:
\begin{equation}
	\begin{aligned}
&\frac12 \inttO \nabla \frakn \cdot (\nabla p) \pt \dxs +	\inttO \frakl \pt^2 \dxs\\
 \leq \,&\,  \|\nabla \frakn\|_{\LinfLinf} \Big(\intt \|\nabla p\|^2_{\Ltwo}\ds+\intt \|\pt\|^2_{\Ltwo}\ds\Big)+ \|\frakl\|_{\LinfLinf}\intt \|\pt\|^2_{\Ltwo}\ds.
	\end{aligned}
\end{equation}
Employing these estimates in  \eqref{est uniqueness} together with \Gronwall's inequality yields 
\[
	\| \pt\|^2_{\LinfLtwo} + \|\nabla p\|^2_{\LinfLtwo} \leq 0,
\]
from which we obtain that $p=0$.
\end{proof}   

%
	
	We next  consider the linearized initial-boundary value problem for the temperature: 
\begin{equation}\label{linearized Pennes problem}
	\left\{	\begin{aligned}
		&\theta_t -\kappa \Delta \theta+\nu\wbstar(x,t) \theta = f \quad \text{ in } \Omega \times (0,T),\\  
		&\theta\vert_{\partial \Om}=0, \\
		&\theta \vert_{t=0}=\theta_0,
	\end{aligned} \right.
\end{equation} 
and wish to establish its well-posedness with sufficient regularity of $\theta$ to be able to later ``close" the fixed-point argument in the analysis of the Westervelt--Pennes system. In view of   the assumptions made on the functions $\frakm$, $\frakn$, and $\frakl$ in Proposition~\ref{prop: linear pressure higher order} and the fact that they can be seen as placeholders for
	\begin{align}\label{functions:m:n:l}
	\frakm = \frac{1}{b(\theta+\Thetaa)}\left(1-2 k(\theta+\Thetaa) p\right), \quad  \frakn= \frac{q(\theta+\Thetaa)}{b(\theta+\Thetaa)}, \quad \frakl = 2\frac{k(\theta+\Thetaa)}{b(\theta+\Thetaa)}\pt,
	\end{align}
we know that we need a solution of \eqref{linearized Pennes problem}, such that (at least) $\theta \in \LinfTHonethree$ and $\sqrt{t}\thetat \in \LinfTHonetwo$.  This regularity is guaranteed by the next result. 

\begin{proposition} \label{prop: linear temperature higher order}
	Let $\kappa>0$. Given $T>0$, assume  that  \[
	 \wbstar \in  \HoneTHtwo, \quad f \in \HoneTHonezero, \quad \theta_0 \in \Honethree.
	\]
	Then initial boundary-value problem \eqref{linearized Pennes problem} has a unique solution $\theta \in \Xtheta$, where $\Xtheta$ is defined in \eqref{def_Xtheta}.
	The solution satisfies
	\begin{equation}\label{lin_est_Pennes}
		\begin{aligned}
			\|\theta\|_{\mathcal{X}_\theta} 
			\lesssim&\,\begin{multlined}[t] \left(\|\theta_0\|_{\Hthree}+ \|f\|_{\HoneHone}\right)
				 \exp(C(T) \Lambda_2(T))
			\end{multlined}
		\end{aligned}
	\end{equation} 
	with
	\begin{equation}\label{def Lambda2}
		\Lambda_2(T) = \intT \Big(1+ \|\wbstar(t)\|^2_{\Htwo}+\|\wbstart(t)\|^2_{\Htwo}\Big) \dt.
	\end{equation}
\end{proposition}

\begin{proof}

The proof can be rigorously justified using a Faedo--Galerkin approach; we present here only the energy estimate as it is the non-standard part of the proof. Testing the time differentiated (semi-discrete) Pennes equation with $s\Delta^2 \theta_{t}(s)$ for $s \in (0,t)$ and integrating over $(0,t)$ for $t \in (0,T)$ yields
\begin{equation}
	\begin{aligned}
		&\frac12 \|\sqrt{t}\Delta \thetat(t)\|^2_{\Ltwo}  + \kappa \intt \|\sqrt{s}\nabla \Delta \thetat(s)\|^2_{\Ltwo} \ds \\
		=\,&\,\frac{1}{2}\intt\|\Delta \thetat(s)\|^2_{\Ltwo}\ds - \nu\inttO s\Delta (\wbstar \theta)_t \Delta \thetat \dxs + \inttO s\ft \Delta^2 \thetat \dxs.
	\end{aligned}
\end{equation}
We can treat the $\ft$ term using Young's inequality as follows:
\begin{equation}
	\begin{aligned}
		\inttO s \ft \Delta^2 \thetat \dxs =\,&\, - \inttO s\nabla \ft \cdot \nabla \Delta \thetat \dxs \\
		\lesssim\,&\,  T\|\nabla \ft\|^2_{\LtwoLtwo}+ \eps \intt \|\sqrt{s}\nabla \Delta \thetat\|^2_{\Ltwo} \ds
	\end{aligned}
\end{equation}
for any $\eps>0$. To estimate the $\nu$ term, we note that the following identity holds:
\begin{equation}
	\begin{aligned}
\Delta (\wbstart \theta+ \wbstar \thetat)
	=	\Delta \wbstart \theta+2 \nabla \wbstar\cdot \nabla \theta+ \wbstart \Delta \theta 
		+\Delta\wbstar \thetat+2 \nabla \wbstar \cdot \nabla \thetat+\wbstar\Delta \thetat .
	\end{aligned}
\end{equation}
Employing H\"older's inequality together with elliptic regularity yields
\begin{equation}
	\begin{aligned}
		&\inttO {s}\Delta (\wbstart \theta+ \wbstar \thetat) \Delta \thetat \dxs \\
		\lesssim\,&\, \intt \sqrt{s} \|\wbstart\|_{\Htwo} \|\Delta \theta\|_{\Ltwo}\|\sqrt{s}\Delta \thetat\|_{\Ltwo} \ds +\intt \|\wbstar\|_{\Htwo}\|\sqrt{s}\Delta \thetat\|^2_{\Ltwo} \ds\\
		\lesssim\,&\,T \intt  \|\wbstart\|^2_{\Htwo} \|\Delta \theta\|^2_{\Ltwo} \ds +\intt \big(1+\|\wbstar\|_{\Htwo}\big)\|\sqrt{s}\Delta \thetat\|^2_{\Ltwo} \ds.
	\end{aligned}
\end{equation}
Thus we are led to the following estimate:
\begin{equation} \label{est theta lin 1}
	\begin{aligned}
	&\frac12 \|\sqrt{t}\Delta \thetat(t)\|^2_{\Ltwo}  + \kappa \intt \|\sqrt{s}\nabla \Delta \thetat(s)\|^2_{\Ltwo} \ds \\
	\lesssim&\, \begin{multlined}[t]
\intt\|\Delta \thetat\|^2_{\Ltwo}\ds+	T \intt  \|\wbstart\|^2_{\Htwo} \|\Delta \theta\|^2_{\Ltwo} \ds\\ +\intt \big(1+\|\wbstar\|_{\Htwo}\big)\|\sqrt{s}\Delta \thetat\|^2_{\Ltwo} \ds
	+ T\|\nabla \ft\|^2_{\LtwoLtwo}\\
	+ \eps \intt \|\sqrt{s}\nabla \Delta \thetat\|^2_{\Ltwo} \ds.
	\end{multlined}
	\end{aligned}
\end{equation} 
Since we have established that for the later fixed-point analysis, we need $\theta \in \LinfTHonethree$, we additionally test the semi-discrete Pennes problem with $\Delta^2 \thetat$, which together with integration by parts leads to
\begin{equation}
	\begin{aligned}
		&\intt \|\Delta \thetat\|^2_{\Ltwo} \dxs +\frac{\kappa}{2} \|\nabla \Delta \theta(s)\|^2_{\Ltwo} \Big \vert_0^t\\
		\leq&\, \inttO \Delta(\wbstar \theta) \Delta \thetat \dxs + \inttO f \Delta^2 \thetat \dxs  \\
			=&\, \inttO \Delta(\wbstar \theta) \Delta \thetat \dxs - \intO \nabla f \cdot \nabla \Delta \theta \dx \Big \vert_0^t + \intt \nabla \ft \cdot \nabla \Delta \theta \dxs .
	\end{aligned}
\end{equation} 
Employing once again H\"older's inequality yields
\begin{equation}
	\begin{aligned}
		&\intt \|\Delta \thetat\|^2_{\Ltwo} \dxs +\frac{\kappa}{2} \|\nabla \Delta \theta(s)\|^2_{\Ltwo} \Big \vert_0^t\\
		\lesssim&\, \intt \|\wbstar\|_{\Htwo}\|\Delta \theta\|_{\Ltwo} \|\Delta \thetat\|_{\Ltwo} \ds - \intO \nabla f \cdot \nabla \Delta \theta \dx \Big \vert_0^t + \intt \nabla \ft \cdot \nabla \Delta \theta \dxs.  	
	\end{aligned}
\end{equation} 
From here we thus have
\begin{equation} \label{est theta lin 2}
	\begin{aligned}
		&\intt \|\Delta \thetat\|^2_{\Ltwo} \dxs +\frac{\kappa}{2} \|\nabla \Delta \theta(s)\|^2_{\Ltwo} \Big \vert_0^t\\
			\lesssim&\, \begin{multlined}[t]\intt  \|\wbstar\|_{\Htwo}^2\|\Delta \theta\|^2_{\Ltwo}\ds+ \eps \intt \|\Delta \thetat\|^2_{\Ltwo} \ds +\eps\|\nabla \Delta\theta(t)\|^2_{\Ltwo}\\
				+\|\nabla \Delta\thetazero\|^2_{\Ltwo} + \intt \|\nabla \Delta \theta\|^2_{\Ltwo} \ds + \|f\|^2_{\HoneHone},
				\end{multlined}
	\end{aligned}
\end{equation}
for any $\eps>0$, where we have relied on the embedding $H^1(0,T; \Honez) \hookrightarrow C([0,T]; \Honez)$ to bound $\|f(s)\|_{\Hone}$ for $s \in [0,t]$. \\
\indent Adding the two estimates \eqref{est theta lin 1} and \eqref{est theta lin 2}, fixing $\eps>0$ small enough so that the $\eps$ terms can be absorbed by the left-hand side, and applying \Gronwall's inequality then yields \eqref{lin_est_Pennes}. 
\end{proof}
\indent	We can now consider a more concrete right-hand side in \eqref{linearized Pennes problem}. Let $p \in \Xp$ and assume that the right hand-side in the heat equation has the special form $f= \calG+\ftheta$ dictated by our model with 
	\begin{equation}\label{calG linearized}
		\calG (\pt)= \frakG(x,t) p_t^2,
	\end{equation}
	where we recall that $\frakG$ acts as a placeholder for $\frakG = \zeta \dfrac{b(\theta+\Thetaa)}{q^2(\theta+\Thetaa)}$.
\noindent If $\frakG$ is sufficiently smooth, we have
	\begin{equation} \label{est_calG H1H1}
		\begin{aligned}
			\|\calG(\pt)\|_{\HoneHone} \lesssim\,&\,
			\|\nabla(\calG(\pt))\|_{L^2(\Ltwo)} +	\|\nabla(\calG(\pt))_t\|_{L^2(\Ltwo)}.
		\end{aligned} 
	\end{equation} 
We can estimate the two contributions on the right-hand side above as follows:
		\begin{equation} \label{est_calG L2H1}
		\begin{aligned}
			\|\nabla(\calG(\pt))\|_{L^2(\Ltwo)}
			\lesssim&\, \begin{multlined}[t]
			\left \|\frakG \right \|_{\LinfLthree}\| \pt\|_{\LinfLinf}\|\nabla \pt\|_{\LtwoLsix}\\
				+\left\|\nabla \frakG \right\|_{\LtwoLtwo} \|\pt\|^2_{\LinfLinf}.
				 \end{multlined}
		\end{aligned} 
	\end{equation} 
	For the second-one, we note that
	\begin{equation}
		\begin{aligned}
				\nabla(\calG(\pt))_t
				=\,&\, \nabla\big(\frakG_t \pt^2+ 2 \frakG \pt \ptt \big) \\
				=\,&\, \pt^2 \nabla \frakG_t+ 2\frakG_t \pt \nabla \pt+ 2 \nabla \frakG \pt \ptt+2\frakG \nabla \pt \ptt +2 \frakG \pt \nabla \ptt.
		\end{aligned}
	\end{equation}
	Thus, by H\"older's inequality, we have the following bound:
	\begin{equation}
		\begin{aligned}
		&\|\nabla(\calG(\pt))_t\|_{L^2(\Ltwo)}\\
		\lesssim&\, \begin{multlined}[t]\|\pt\|_{\LinfLinf}^2 \|\nabla \frakG_t\|_{\LtwoLtwo}+ \|\frakG_t\|_{\LtwoLthree}\|\pt\|_{\LinfLinf}\|\nabla \pt\|_{\LinfLsix} \\
			+\|\frakG\|_{\LinfLinf}\|\nabla \pt\|_{\LinfLfour}\|\ptt\|_{\LtwoLfour}\\
			+\|\frakG\|_{\LinfLinf}\|\pt\|_{\LinfLinf}\|\nabla \ptt\|_{\LtwoLtwo}
			\end{multlined}\\
			\lesssim&\, \big(\|\frakG\|_{\HoneHone}+\|\frakG\|_{\LinfLinf}\big)\|p\|^2_{\Xp}.
		\end{aligned}
	\end{equation}
By combining this estimate with the bound in \eqref{lin_est_Pennes}, we arrive at the following corollary of Proposition~\ref{prop: linear temperature higher order}.

\begin{corollary} \label{corollary: Pennes}
	Let the assumptions of Proposition~\ref{prop: linear temperature higher order} hold and let $p \in \Xp$ be the solution of \eqref{ibvp_linWest}. Let additionally 
	\begin{equation} \label{frakG reg assumption}
	\frakG \in \LinfTLinf \cap H^1(0,T; H_0^1(\Omega))
	\end{equation}
and $\ftheta \in \HoneTHonezero$. Then the solution of \eqref{linearized Pennes problem} with $f=\frakG \pt^2+\ftheta$ satisfies 
		\begin{equation}\label{lin_est_Pennes p}
		\begin{aligned}
			\|\theta\|_{\mathcal{X}_\theta} 
			\lesssim&\,\begin{multlined}[t] \left(\|\theta_0\|_{\Hthree}+ (\|\frakG\|_{\LinfLinf}+\|\frakG\|_{\HoneHone})\|\pt\|^2_{\Xp}+ \|\ftheta\|_{\HoneHone}\right) \\ \times
				\exp(C(T)\Lambda_2(T)),
			\end{multlined}
		\end{aligned}
	\end{equation}
where $\Lambda_2$ is defined in \eqref{def Lambda2}.	
\end{corollary}
Having completed the linear pressure-temperature analysis, we can transform the problem of showing the existence of solutions to the Westervelt--Pennes system into a fixed-point problem for a suitably defined mapping, which we can then analyze using the Banach fixed-point theorem.

\subsection{Analysis of the Westervelt--Pennes system with nonlocal acoustic damping}  
 To state the well-posedness of the Westervelt--Pennes system, we introduce the solution space 
 \begin{equation} \label{def sol space WestPennes}
  \solspaceWestPennes=\mathcal{X}_p \times \mathcal{X}_\theta
 \end{equation}
for the pressure-temperature, where $\Xp$ is defined in \eqref{def_Xp} and $\Xtheta$ in \eqref{def_Xtheta}. Consider the mapping 
\begin{equation} \label{def mapping calT}
\mathcal{T}: \ball \ni (\pstar, \thetastar) \mapsto (p, \theta)
\end{equation}
that associates to each $(\pstar, \thetastar) \in \ball \subset \solspaceWestPennes$ the solution $(p, \theta) \in \solspaceWestPennes$ of the linear  problem 
\begin{equation} \label{linearized problem fixed point}
	\left\{
	\begin{aligned}
		& \frakm \ptt-\frakn \Delta p - \frakK* \Delta \pt -\frakl \pt=\fp \ &&\text{in} \ \Omega \times (0,T), \\[1mm]
		& \theta_t -\kappa\Delta \theta +\nu \wbstar \theta = \calG(p_t)+\ftheta \ &&\text{in} \ \Omega \times (0,T), \\[1mm]
	\end{aligned} 
	\right.
\end{equation}
with boundary and initial conditions given in \eqref{coupled_problem_BC_WP} and \eqref{coupled_problem_IC_WP}, respectively, and the variable coefficients chosen as
\begin{equation} \label{def frakm frakn frakl}
\begin{aligned}	
&\frakm = \frac{1}{\bstar}\left(1-2 \kstar \pstar\right), \quad  \frakn= \frac{\qstar}{\bstar}, \quad \frakl = 2\frac{\kstar}{\bstar}\pt.
\end{aligned}
\end{equation}
Above, we have introduced the short-hand notation
\begin{equation} \label{def qstar bstar kstar}
\qstar = q(\thetastar+\Thetaa), \quad \bstar = b(\thetastar+\Thetaa), \quad \kstar = k(\thetastar+\Thetaa),
\end{equation}	
and, in the heat equation,
\begin{equation} \label{def calG wbstar}
\begin{aligned}	
	& \calG = \frakG\pt^2,  \qquad  \frakG= \zeta\frac{\bstar}{\qstar}, \qquad \wbstar = \wb(\thetastar+\Thetaa).
\end{aligned}
\end{equation}
The ball in \eqref{def mapping calT} is defined as follows: 
\begin{equation} \label{def ball}
	\begin{aligned}
		\ball= \big\{(\pstar, \thetastar)\in \solspaceWestPennes:&\,\|\pstar\|_{\mathcal{X}_p} \leq \Rp,\quad \|\thetastar\|_{\Xtheta} \leq \Rtheta, \\
		& (\pstar, \pstart,  \thetastar) \vert_{t=0}=(p_0, p_1, \theta_0) \big\},
	\end{aligned}
\end{equation}
where the radii $\Rp \in (0, \bRp]$ and $\Rtheta \in (0, \bRtheta]$ should be small enough, as required by the well-posedness analysis below.

\begin{theorem} \label{thm: West Pennes}
	Let $T>0$. Let Assumption~\ref{assumptions kernel} on the memory kernel $\frakK$ and Assumption~\ref{assumptions: q b k wb calG} on the functions $q$, $b$, $\wb$, and $\calG$ hold. Let $\fp$, $\ftheta \in \HoneTHonezero$. Additionally, assume that
	\begin{equation}
		(\pzero, p_1) \in \Honethree \times \left(\Honetwo\right), \ \thetazero \in \Honethree.
	\end{equation}
	There exist $\deltap=\deltap(T)>0$ and $\deltatheta=\deltatheta(T)>0$  small enough such that if
		\begin{equation}
				\begin{aligned}
						&	\|\pzero\|_{\Hthree} + \|p_1\|_{\Htwo} +\|\fp\|_{\HoneHone}\leq \deltap,\\
						&		\|\thetazero\|_{\Hthree} + \|\ftheta\|_{\HoneHone}\leq \deltatheta,
					\end{aligned}
			\end{equation}
	then the nonlocal Westervelt--Pennes system \eqref{coupled_problem West Pennes}, \eqref{coupled_problem_BC_WP}, \eqref{coupled_problem_IC_WP} has a unique solution in $\solspaceWestPennes$, where $\solspaceWestPennes$ is defined in \eqref{def sol space WestPennes}. The solution satisfies the following bound:
	\begin{equation}
		\begin{aligned}
		\|(p, \theta)\|_{\solspaceWestPennes} \lesssim\,&\, \begin{multlined}[t]
			\|\pzero\|_{\Hthree} + \|p_1\|_{\Htwo} +\|\thetazero\|_{\Hthree} + \|\ftheta\|_{\HoneHone}+\|\fp\|_{\HoneHone}.
		\end{multlined}
		\end{aligned}
	\end{equation}
\end{theorem}
~\\
We proceed through several steps to prove the theorem, in which we verify the conditions of the Banach fixed-point theorem. \\

\noindent \underline{Step I}: The ball $\ball$, defined in \eqref{def ball}, is non-empty.\\

	Let $\Cp$ and $\Ctheta$ be the hidden constants in the estimates  \eqref{est_p_linear_1}  and \eqref{lin_est_Pennes}, respectively. The space $\ball$ is non-empty because the solution of the linear problem in the particular case where 
	\begin{equation} \label{special choice alpha beta f g}
		\frakm=1,\quad   \frakn = 1,\quad \fp=\ftheta =0, \quad \wbstar =0,\quad \zeta =0
	\end{equation}
	belongs to it, provided that $\deltap$ and $\deltatheta$ are chosen correctly. 
	Indeed, the solution $(\pstar, \thetastar) \in \solspaceWestPennes$ of such a linear problem will belong to the ball $\ball$ provided that we select $\deltap$ such that
	\begin{equation} \label{est p simple}
		\begin{aligned}
	 		&\|\pstar\|_{\mathcal{X}_p} \leq\Cp \left(\|p_0\|_{\Hthree}+\| p_1\|_{\Htwo} 
			\right)   e^ {C(T)} \leq\Cp \deltap   e^ {C(T)}\leq \Rp.
		\end{aligned}
	\end{equation}
Furthermore, with the particular choice \eqref{special choice alpha beta f g} in the linear problem, by Proposition~\ref{prop: linear temperature higher order}, we have 
	\begin{equation}
		\begin{aligned}
			\|\thetastar\|_{\Xtheta} \leq& \,  \Ctheta \|\theta_0\|_{\Hthree}e^ {C(T)} \leq \Ctheta \deltatheta e^ {C(T)}.
		\end{aligned}
	\end{equation}
Thus by choosing $\deltatheta \leq (\Ctheta e^ {C(T)})^{-1}\Rtheta$, 	we have $\|\thetastar\|_{\Xtheta} \leq \Rtheta$.  We thus conclude that $(\pstar, \thetastar) \in \ball$.

~\\

\noindent \underline{Step II}: {Verifying the assumptions of Propositions~\ref{prop: linear pressure higher order} and~\ref{prop: linear temperature higher order}}\\ 
~\\
We wish to apply Propositions~\ref{prop: linear pressure higher order} and~\ref{prop: linear temperature higher order} to prove the self-mapping property of $\calT$. To this end, we first verify that the assumptions of these propositions are satisfied. 
We begin by establishing the properties of functions $\qstar$ and $\bstar$ defined in \eqref{def qstar bstar kstar}.
\begin{lemma}\label{lemma: estimates bstar qstar}
	Let the assumptions of Theorem~\ref{thm: West Pennes} hold and let $(\pstar, \thetastar) \in \ball$.  For small enough $\Rtheta>0$, there exist $\ulq$, $\olq$, $\ulb$, $\olb>0$, such that
	\begin{equation}\label{bound_b_q}
		0< \ulq \leq \qstar \leq \olq , \quad 0 < \ulb \leq \bstar \leq \olb \quad \text{in } \Omega \times (0,T). 
	\end{equation}
Furthermore, the following bounds hold:
\begin{equation} \label{bounds qstar bstar LinfH3}
	\begin{aligned}
				\|\qstar\|_{\LinfHthree} \lesssim&\, 1+ \Rtheta , \qquad && \|\sqrt{t} (\qstar)_t\|_{\LinfHtwo} \lesssim \Rtheta,\\
				\|\bstar\|_{\LinfHthree} \lesssim&\, 1+ \Rtheta , \qquad && \|\sqrt{t} (\bstar)_t\|_{\LinfHtwo} \lesssim \Rtheta.
	\end{aligned}
\end{equation}
\end{lemma}
	\begin{proof}
			On account of $q$, $b \in C^1(\R)$ and $\|\thetastar\|_{\LinfLinf} \lesssim \Rtheta$, we immediately have
		\begin{equation}
		 \|\qstar\|_{\LinfLinf} \leq \olq, \quad	\|\bstar\|_{\LinfLinf} \leq \olb
		\end{equation}
for some $\olq$, $\olb>0$. The lower (non-degeneracy) bounds hold   
		provided $\Rtheta$ is sufficiently small relative to $q(\Thetaa)$ and $b(\Thetaa)$. Indeed, since  
		\begin{equation}
			q(\thetastar+\Thetaa) = q(0+\Thetaa)+ \int_{\Thetaa}^{\thetastar+\Thetaa} q'(s) \ds   \quad \text{in} \ \Omega \times (0,T),
		\end{equation}
		we have 
		\[
			q(\thetastar+\Thetaa) \geq q(\Thetaa)- C|\thetastar| \geq q(\Thetaa)- C\Rtheta  \quad \text{in} \ \Omega \times (0,T).
			\]
			Thus, there exists $\ulq>0$, such that
				\[
			q(\thetastar+\Thetaa) \geq q(\Thetaa)- C\Rtheta \geq \ulq >0
			\]
			 for small enough $\Rtheta$. We can reason similarly for $b$. Further, since $q \in C^{3,1}(\R)$, we have
		\begin{equation}
			\begin{aligned}
				\|\qstar\|_{\LinfHthree} \lesssim&\, 1+ \Rtheta, 
			\end{aligned}	
		\end{equation}
where we have used $\Rtheta \leq \bRtheta$ to estimate higher-order terms with respect to $\Rtheta$ by $ C\Rtheta$. Furthermore,
		\begin{equation}
			\begin{aligned}
				\|\sqrt{t}\, (\qstar)_t\|_{\LinfHtwo}=\,&\,	\|\sqrt{t} \, q'(\thetastar+\Thetaa)\thetastart\|_{\LinfHtwo} \\
				\lesssim\,&\,  \| q'(\thetastar+\Thetaa)\|_{\LinfHtwo} \|\sqrt{t}\,\thetastart\|_{\LinfHtwo}\lesssim R_\theta.
			\end{aligned}	
		\end{equation} 	
	 The bounds on $\bstar$  follows in the same manner. 
	\end{proof}
\noindent We next investigate the properties of the function $\kstar$.
\begin{lemma}\label{lemma: estimates kstar}
		Let the assumptions of Theorem~\ref{thm: West Pennes} hold and let $(\pstar, \thetastar) \in \ball$.  Under the assumptions of Lemma~\ref{lemma: estimates bstar qstar}, we have
		\begin{equation} \label{bound kstar}
			\|\kstar\|_{\LinfHthree} \lesssim 1+ \Rtheta.
		\end{equation}
\end{lemma}		
\begin{proof}
	On account of  Assumption~\ref{assumptions: q b k wb calG}, we know that $\kstar$ has the form $\kstar = {\upsilon}{\qstar}^{-1}$
for some $\upsilon \in \R$. Toward estimating $\kstar$, we note that
\begin{equation}
	\begin{aligned}
	 \|\qstar^{-1}\|_{\LinfHthree} \lesssim\,&\, \|\qstar^{-1}\|_{\LinfLtwo}+\|\nabla \qstar^{-1}\|_{\LinfHtwo} \\
		\lesssim\,&\, 1 +  \|\qstar^{-2}\|_{\LinfHtwo}\|\nabla \qstar\|_{\LinfHtwo} \\
		\lesssim\,&\,  1+  \big(\|\qstar^{-2}\|_{\LinfLtwo}+\|\nabla \qstar^{-2}\|_{\LinfHone}\big)\|\nabla \qstar\|_{\LinfHtwo}.
	\end{aligned}
\end{equation}	
Therefore,
\begin{equation}
	\begin{aligned}
	 \|\qstar^{-1}\|_{\LinfHthree} \lesssim&\, 
		  1+  \big(1+\|\qstar^{-3}\nabla \qstar\|_{\LinfHone}\big)\|\nabla \qstar\|_{\LinfHtwo}.
	\end{aligned}
\end{equation}	
The term in the bracket can be further bounded as follows:
\begin{equation}
	\begin{aligned}
		&\|\qstar^{-3}\nabla \qstar\|_{\LinfHone}\\
		\lesssim &\, \|\qstar^{-3}\|_{\LinfLinf}\|\nabla \qstar\|_{\LinfLtwo}+\|\qstar^{-4}(\nabla \qstar)^2\|_{\LinfLtwo}+\|\qstar^{-3}\nabla^2 \qstar\|_{\LinfLtwo}\\
		\lesssim &\,\|\nabla \qstar\|_{\LinfLtwo}+\|\nabla \qstar\|_{\LinfLfour}^2+\|\nabla^2 \qstar\|_{\LinfLtwo}. 
	\end{aligned}
\end{equation}
Putting the above estimates together and using \eqref{bound_b_q}, we obtain: 
\begin{equation}
	\begin{aligned}
		\|(\qstar)^{-1}\|_{\LinfHthree}
		\lesssim\, 1+\| \qstar\|_{\LinfHthree}^3.
	\end{aligned}
\end{equation}
Using the bound in  \eqref{bounds qstar bstar LinfH3} then yields $\|(\qstar)^{-1}\|_{\LinfHthree} \lesssim 1+ \Rtheta$,
where we have employed $\Rtheta \lesssim \bRtheta \lesssim 1$ for higher-order terms in $\Rtheta$. The bound in \eqref{bound kstar} then immediately follows.
\end{proof}
Equipped with the knowledge of properties of functions $\qstar$, $\bstar$, and $\kstar$, we can now determine the properties of functions $\frakm$, $\frakn$, and $\frakl$, defined in \eqref{def frakm frakn frakl}.
\begin{lemma}\label{estimates frakm frakn frakl}		
	Let the assumptions of Theorem~\ref{thm: West Pennes} hold and let $(\pstar, \thetastar) \in \ball$.  Furthermore, let the assumptions of Lemma~\ref{lemma: estimates bstar qstar} hold. For sufficiently small $\Rp>0$, the functions 
	\[
	\begin{aligned}	
		&\frakm = \frac{1}{\bstar}\left(1-2 \kstar \pstar\right), \quad  \frakn= \frac{\qstar}{\bstar}, \quad \frakl = 2\frac{\kstar}{\bstar}\pstart
	\end{aligned}
	\]
	satisfy the regularity and non-degeneracy assumptions of Proposition~\ref{prop: linear pressure higher order} with the following bounds:
	\begin{equation}
		\begin{aligned}
		&0 < \ulfrakm \leq \frakm \leq \olfrakm\quad \text{in } \Omega \times (0,T), \hspace*{2cm}
			\|\frakm\|_{\LinfHthree} \leq\, 1 + \Rtheta +\Rtheta \Rp, \\
		&	\|\sqrt{t}\frakmt\|_{\LinfLinf} \leq \Rtheta (1+\Rp)+\sqrt{T}\Rp, 
		\end{aligned}
	\end{equation}
	and
	\begin{equation}
	\begin{aligned}
		&0 < \ulfrakn \leq \frakn \leq \olfrakn\quad \text{in } \Omega \times (0,T), \quad
			\|\frakn\|_{\LinfHthree} \leq\, 1+ \Rtheta, \quad
			\|\sqrt{t}\fraknt\|_{\LinfLinf} \leq \Rtheta,
	\end{aligned}
\end{equation}
as well as $\|\frakl\|_{\LinfHtwo} \leq (1+\Rtheta)^2 \Rp$.
	\end{lemma}
\begin{proof}
We begin the proof by observing that
\begin{equation}
	\begin{aligned}
0<\ulfrakm: = \frac{1}{\ulb}\bigg(1 - 2 \frac{|\upsilon|}{\olq} \Rp\bigg)\leq	\frakm	\leq \frac{1}{\ulb}\bigg(1+2 \frac{|\upsilon|}{\olq} \Rp\bigg) := \olfrakm,
	\end{aligned}
\end{equation}
if $\Rp$ is small enough so that $1  -  (2 |\upsilon|/\olq) \Rp >0$.
Further,
\begin{equation} 
	0 < \olfrakn := \ulq\, /\, \olb \leq \frakn \leq \olq \,/\, \ulb.
\end{equation}
	We next use the product estimate in $\Hthree$ to conclude that 
	\begin{equation}
\|\frakm\|_{\LinfHthree}\lesssim \Big\|\frac{1}{\bstar}\Big\|_{\LinfHthree}(1+\|\kstar\|_{\LinfHthree}\|\pstar\|_{\LinfHthree}).
\end{equation}
Employing the bounds in Lemmas~\ref{lemma: estimates bstar qstar} and~\ref{lemma: estimates kstar} then leads to
	\begin{equation}
	\|\frakm\|_{\LinfHthree}\lesssim (1+\Rtheta) \big(1+(1+\Rtheta)\Rp\big) \lesssim 1 + \Rtheta +\Rtheta \Rp.
\end{equation}
Similarly,
\begin{equation}
	\begin{aligned}
		\|\frakn\|_{\LinfHthree} \lesssim \|(\bstar)^{-1}\|_{\LinfHthree}\|\qstar\|_{\LinfHthree} \lesssim (1+ \Rtheta)^2 \lesssim 1+ \Rtheta.
	\end{aligned}
\end{equation}
Next, we note that the following identity holds:
\begin{equation}
\frakmt=\frac{-\bstar'(\thetastar)\theta_{\ast t}}{\bstar^2}\left(1-2 \kstar \pstar\right)-\frac{1}{\bstar}\Big(\kstar'(\thetastar)\theta_{\ast t}\pstar+\kstar(\thetastar)p_{\ast t}\Big).
\end{equation}
Thus, we can estimate $\sqrt{t}\frakmt$ as follows:
\begin{equation}
\begin{aligned}
\|\sqrt{t}\frakmt\|_{\LinfLinf}\lesssim &\,
\begin{multlined}[t]
\frac{1}{\ulb ^2} \|\bstar'\|_{\LinfLinf}\|\sqrt{t}\theta_{\ast t}\|_{\LinfLinf}\Big(1+\|\kstar\|_{\LinfLinf}\|p_\ast\|_{\LinfLinf}\Big)\\
+\frac{1}{\ulb}\|\kstar'\|_{\LinfLinf}\|\sqrt{t}\theta_{\ast t}\|_{\LinfLinf}\|p_\ast\|_{\LinfLinf}\\
+\frac{1}{\ulb}\sqrt{T}\|\kstar\|_{\LinfLinf}\|p_{\ast t}\|_{\LinfLinf}.
\end{multlined}
\end{aligned}
\end{equation}
Hence, by $(\pstar, \thetastar) \in \ball$, we obtain 
\begin{equation}
\|\sqrt{t}\frakmt\|_{\LinfLinf}\lesssim R_\theta (1+\Rp)+\sqrt{
T}\Rp. 
\end{equation}
In an analogous manner, we have 
\begin{equation}
\begin{aligned}
\|\sqrt{t}\fraknt\|_{\LinfLinf}\lesssim\, &\,
\begin{multlined}[t]
\frac{1}{\ulb ^2} \|\qstar '\|_{\LinfLinf}\|\bstar\|_{\LinfLinf}\|\sqrt{t}\theta_{\ast t}\|_{\LinfLinf}\\
+\frac{1}{\ulb ^2}\|\bstar '\|_{\LinfLinf}\|\qstar\|_{\LinfLinf}\|\sqrt{t}\theta_{\ast t}\|_{\LinfLinf} 
\lesssim\, \Rtheta. \end{multlined} 
\end{aligned}
\end{equation}
Finally, we can estimate the function $\frakl$ as follows:
\begin{equation}
	\begin{aligned}
			\|\frakl\|_{\LinfHtwo} \lesssim&\, \|\kstar\|_{\LinfHtwo} \|(\bstar)^{-1}\|_{\LinfHtwo}\|\pstart\|_{\LinfHtwo} \\
			\lesssim&\, (1+\Rtheta)^2 \Rp,
	\end{aligned}
\end{equation}
which completes the proof.
\end{proof}
We have thus verified all assumptions of Proposition~\ref{prop: linear pressure higher order}. We next check that the assumptions of Proposition~\ref{prop: linear temperature higher order} and Corollary~\ref{corollary: Pennes} hold as well. To this end, the next lemma establishes the bounds on the function $\wbstar$, defined in \eqref{def calG wbstar}.
	\begin{lemma}\label{lemma: bound wbstar}
		Let the assumptions of Theorem~\ref{thm: West Pennes} hold and let $(\pstar, \thetastar) \in \ball$.  Then $\wbstar \in \HoneTHtwo$ and the following bound holds:
		\begin{equation} \label{est beta}
			\begin{aligned}
				\|\wbstar\|_{\HoneHtwo} \lesssim 1+\Rtheta.
			\end{aligned}		
		\end{equation}	
	\end{lemma}
	\begin{proof}
	Since $\wb \in \Ctwooneloc(\R)$ and $\|\thetastar\|_{\Xtheta} \leq \Rtheta$, we immediately have
		\begin{equation}
			\begin{aligned}
				\|\wbstar\|_{\LinfHtwo} \lesssim 1+ \Rtheta.
			\end{aligned}
		\end{equation}
	Further,
	\begin{equation}
		\begin{aligned}
			\|(\wbstar)_t\|_{\LtwoHtwo} \leq \|\wb'\|_{\LinfHtwo} \|\qstart\|_{\LtwoHtwo} \lesssim \Rtheta,
		\end{aligned}
	\end{equation}	
which completing the proof.	
	\end{proof}
	
\noindent Furthermore, we have the following estimate for $\frakG$, defined in \eqref{def calG wbstar}.
\begin{lemma} \label{lemma: est frakG} 		Let the assumptions of Theorem~\ref{thm: West Pennes} and Lemma~\ref{lemma: estimates bstar qstar} hold and let $(\pstar, \thetastar) \in \ball$. Then
	\begin{equation} \label{frakG reg assumption}
	\frakG \in \LinfTLinf \cap H^1(0,T; H_0^1(\Omega))
\end{equation}
and the following bound holds:		
\[
\|\frakG\|_{\LinfLinf}+\|\frakG\|_{\HoneHone} \lesssim\, 1+ \Rtheta.
\]
\end{lemma}
\begin{proof}
Calling upon Lemma~\ref{lemma: estimates bstar qstar}, we obtain
\begin{equation}
	\begin{aligned}
		&\|\frakG\|_{\LinfLinf}+\|\frakG\|_{\HoneHone} \\
		\lesssim&\, \frac{\olb}{\ulq^2}+\|(\qstar)^{-2} \bstar\|_{\HoneLtwo}+  \| \nabla((\qstar)^{-2} \bstar)\|_{\HoneLtwo} \\
		\lesssim&\, \begin{multlined}[t] 1+ \|((\qstar)^{-2} \bstar)_t\|_{\LtwoLtwo} + 	\| \nabla((\qstar)^{-2} \bstar)\|_{\LtwoLtwo}\\\hspace*{4.5cm}+\| \nabla((\qstar)^{-4} (\bstar' (\qstar)^2-\bstar ((\qstar)^2)'))\|_{\LtwoLtwo}	\end{multlined} \\	
		\lesssim&\, 1 + \|\thetastart\|_{\LtwoLtwo}+\|\nabla \thetastar\|_{\LtwoLtwo}+ \|\nabla \thetastar\|_{\LinfLfour}\|\nabla \thetastart\|_{\LtwoLfour} \\
		\lesssim&\, 1+ \Rtheta ,
	\end{aligned} 
\end{equation} 
where we have used $\Rtheta \leq \bRtheta \lesssim 1$. 
\end{proof}

	We can therefore conclude that the assumptions of Propositions~\ref{prop: linear pressure higher order} and~\ref{prop: linear temperature higher order} as well as Corollary~\ref{corollary: Pennes} are satisfied, and we can rely on them to prove the properties of the mapping $\calT$, the first of which is the self-mapping property.  \\
	
\noindent \underline{Step III}: {$\calT$ is a self-mapping.}\\
	\begin{lemma} \label{lemma: self-mapping}
		Let  the assumptions of Theorem~\ref{thm: West Pennes} hold. For sufficiently small $\deltap>0$ and $\Rp>0$, the mapping $\calT$ satisfies $\calT(\ball) \subset \ball$.
	\end{lemma} 
	\begin{proof}
Thanks to verifying all assumptions, we know that we can employ Proposition~\ref{prop: linear pressure higher order} on problem \eqref{linearized problem fixed point} to obtain the solution $(p, \theta) \in \solspaceWestPennes$. Furthermore, the following bound holds:
\begin{equation} \label{bound p}
	\begin{aligned}
		\|p\|_{\Xp} \lesssim \deltap \exp\big(C(T)\Lambda_1(\frakm, \frakl, \frakn)\big),
	\end{aligned}
\end{equation}
where we recall that
\begin{equation}
	\begin{aligned}
		\begin{multlined}[t]
		\Lambda_1(\frakm, \frakl, \frakn) =\,  1+\|\frakn\|_{\Xfrakn}+\calL^2(\frakm, \frakn, \frakl)
			+\int_0^T \frac{1}{\sqrt{t}}\Big( \|\sqrt{t}\frakm_t\|_{\Linf} +\|\sqrt{t}\frakn_t\|_{\Linf}\Big)\dt
		\end{multlined}
	\end{aligned}
\end{equation}
with
\begin{equation}
	\calL(\frakm, \frakn, \frakl) = \|\frakm\|_{\Xfrakm}\big(1+\|\frakm\|_{\Xfrakm}\big)\big(1+\|\frakn\|_{\Xfrakn}+\|\frakl\|_{\Xfrakl}\big).
\end{equation}
We can use Lemma~\ref{estimates frakm frakn frakl} to find that
\[
\begin{aligned}
	\calL(\frakm, \frakn, \frakl) \lesssim \, &\, \big(1+\Rtheta+\Rtheta \Rp \big)\big(1+\Rtheta+\Rtheta \Rp\big) \big(1+\Rtheta +(1+\Rtheta)^2 \Rp\big) \\
	\lesssim&\, 1 + \Rtheta + \Rp.
\end{aligned}
\]
	Then further utilizing Lemma~\ref{estimates frakm frakn frakl} and noting that $\int_0^T t^{-1/2}\dt = 2\sqrt{T}$, we conclude that
	\begin{equation}
		\begin{aligned}
		\Lambda_1(\frakm, \frakl, \frakn) \lesssim&	\,	\begin{multlined}[t]		 1+ \Rtheta +\Rp +(1+\Rtheta + \Rp) \int_0^T t^{-1/2}\dt 
		\end{multlined} \\
		\lesssim_T&\,  1+\Rtheta + \Rp.
		\end{aligned}
	\end{equation}
	Thus, for small enough $\deltap$,  \eqref{bound p} implies that
		\begin{equation}
		\begin{aligned}
			\|p\|_{\Xp} 
			 \lesssim \,&\,  \deltap \exp\Big\{C(T)\big(1+ \bRtheta +\bRp\big)\Big\} \leq \Rp.
		\end{aligned}
	\end{equation}
		We next estimate the temperature $\theta$. By Corollary~\ref{corollary: Pennes},  we have
\begin{equation} \label{est theta}
	\begin{aligned}
		\|\theta\|_{\mathcal{X}_\theta} 
		\lesssim&\,\begin{multlined}[t]
		\left(\|\theta_0\|_{\Hthree}+ \big(\|\frakG\|_{\LinfLinf}+\|\frakG\|_{\HoneHone}\big)\|\pt\|^2_{\Xp}+ \|\ftheta\|_{\HoneHone}\right) \\
		\times\exp\big(C(T)\Lambda_2(T)\big),
		\end{multlined}
	\end{aligned}
\end{equation}
where we recall that
	\begin{equation} \label{est Lambda_2}
	\Lambda_2(T) = \intT \Big(1+ \|\wbstar(t)\|^2_{\Htwo}+\|\wbstart(t)\|^2_{\Htwo}\Big) \dt
\end{equation}
and that thanks to Lemma~\ref{lemma: est frakG}, we have
\begin{equation}
	\begin{aligned}
		&\|\frakG\|_{\LinfLinf}+\|\frakG\|_{\HoneHone} \lesssim\, 1+ \Rtheta.
	\end{aligned} 
\end{equation} 
Using Lemma~\ref{lemma: bound wbstar} to estimate \eqref{est Lambda_2}, we find that
\begin{equation} 
	\begin{aligned}
		\Lambda_2(T) \lesssim_T 1+ \Rtheta.
	\end{aligned}
\end{equation}
Employing these bounds in \eqref{est theta} immediately yields
\begin{equation}
	\begin{aligned}
		\|\theta\|_{\Xtheta} 
		\lesssim \,&\, \Big(\deltatheta+ (1+\bRtheta)\Rp^2\Big)\exp\big(C(T)(1+\bRtheta)\big).
	\end{aligned}
\end{equation}
We can therefore also guarantee that $\|\theta\|_{\Xtheta} \leq \Rtheta$ provided $\deltatheta$ and $\Rp$ are sufficiently small. Thus we conclude that $(p, \theta) \in \ball$ for properly adjusted radii and data size.
	\end{proof}	  
~\\
\noindent \underline{Step IV}: {$\calT$ is strictly contractive.} \\

\noindent We next prove contractivity of the mapping.  
	Take $(\ponestar, \thetaonestar)$, $(\ptwostar, \thetatwostar) \in \ball$.		Let 
	\[
	\mathcal{T}(\ponestar, \thetaonestar)=(\pone, \thetaone), \qquad \mathcal{T}(\ptwostar, \thetatwostar)=(\ptwo, \thetatwo)\]
	and denote the differences as
	\[
	\opstar = \ponestar-\ptwostar, \ \othetastar = \thetaonestar-\thetatwostar.
	\]
	Let us also introduce the short-hand notation
	\begin{equation}
		\begin{aligned}
			\psionestar =&\,  \psi(\thetaonestar+\Thetaa), \quad 	&&\psitwostar =\,  \psi(\thetatwostar+\Thetaa)
		\end{aligned}
	\end{equation}
	for $\psi \in \{b, q, k, \wb\}$ and
	\[
	\begin{aligned}	
		\frakm^{(i)} = \frac{1}{\bstar^{(i)}}\left(1-2 \kstar^{(i)} \pstar^{(i)}\right), \quad  \frakn^{(i)}= \frac{\qstar^{(i)}}{\bstar^{(i)}}, \quad \frakl^{(i)} = 2\frac{\kstar^{(i)}}{\bstar^{(i)}}\pstart^{(i)},\qquad \text{for }i=1,2.
	\end{aligned} 
	\] 
	We can then see the difference
	\[\big(\op, \otheta\big)= \big(\pone-\ptwo, \thetaone-\thetatwo\big)\]
	as the solution of the following wave-heat system:
	\begin{equation}\label{Equation_West_diff}
		\left\{    
		\begin{aligned}
			&\begin{multlined}[t] \frakmone \optt - \fraknone \Delta \overline{p}- \frakK* \Delta \opt -\fraklone \opt =  F ,  \end{multlined} \\[2mm] 
			& \othetat - \kappa \Delta \otheta + \nu\wbonestar \otheta = G, 
		\end{aligned}
		\right.    
	\end{equation}
	supplemented with homogeneous initial and boundary conditions. 	The acoustic right-hand side in \eqref{Equation_West_diff} is given by
	\begin{equation} \label{rhs_contractivity}
		\begin{aligned}
			F = \begin{multlined}[t]
			    -\big(\frakmone-\frakmtwo\big) \ptwostartt
				+\big(\fraknone-\frakntwo\big) \Delta \ptwostar
				+\big(\fraklone-\frakltwo\big)\ptwostart
			\end{multlined}
		\end{aligned} 
	\end{equation}  
	and the right-hand side of the temperature equation by
	\begin{equation} \label{def G}
		\begin{aligned}
			G =  -\nu \Big(\wbonestar-\wbtwostar\Big) \thetatwostar + \zeta \frac{\bonestar}{\big(\qonestar\big)^2}\big(\pone_{t}\big)^2- \zeta \frac{\btwostar}{\big(\qtwostar\big)^2}\big(\ptwo_{t}\big)^2.
		\end{aligned}
	\end{equation}
	Toward proving contractivity, we need the following estimate of $F$.
	\begin{lemma} \label{lemma: est F}	Let the assumptions of Theorem~\ref{thm: West Pennes} hold. Then
			\begin{equation} \label{est F} 	
	\begin{aligned}
		\|F\|_{\LonetLtwo} \lesssim \sqrt{T}\Rp \Big(\|\opstar\|_{\LtwotHone}+\|\opstart\|_{\LtwotLtwo}+\|\othetastar\|_{\LtwotHone}\Big).
	\end{aligned}
\end{equation}
	\end{lemma}
	\begin{proof}
The proof follows by exploiting the properties of functions $\frakm^{(1,2)}$ and $\pstar^{(1), (2)}$. For completeness, we proved the details in Appendix~\ref{Appendix: Estimate F}.
	\end{proof}
	We will show that  $\calT$ is a contraction with respect to the norm of the space $\Yp \times \Ytheta \supseteq  \Xp \times \Xtheta$, where $\Yp$ is defined as
	\begin{equation} \label{def_Yp}
		\begin{aligned}
			\Yp
			= \Big\{
			p \in L^{\infty}(0, T; \Honezero):\, &p_t \in L^{\infty}(0, T;  \Ltwo), 
			\, \ \frakK*\nabla  \pt \in \LtwoLtwo \Big\},
		\end{aligned}
	\end{equation}
endowed with the norm 
	\begin{equation}
		\begin{aligned}
			\|p\|_{\Yp} = \left(\Vert p\Vert^2_{L^\infty(\Hone)}+\Vert p_t\Vert_{L^\infty(\Ltwo)}^2 + \|\frakK* \nabla  \pt\|^2_{\LtwoLtwo}\right)^{1/2}.
		\end{aligned}
	\end{equation}
	Further the space $\Ytheta$ is given by
	\begin{equation}
		\begin{aligned}
			\Ytheta =  \LinfTLtwo \cap \LtwoTHonez,
		\end{aligned}
	\end{equation}
	with the norm
	\begin{equation}
		\|\theta\|_{\Ytheta} = \left(\|\theta\|^2_{\LinfLtwo} + \|\nabla \theta\|^2_{\LtwoLtwo}\right)^{1/2}.
	\end{equation}
	We then set
	\begin{equation}
		\|(p, \theta)\|_{\Yp \times \Ytheta} = \left(\|p\|_{\Yp}^2+\|\theta\|^2_{\Ytheta}\right)^{1/2}
	\end{equation}
	and proceed to proving contractivity.
	\begin{proposition} \label{prop: contractivity}
		Let the assumptions of Theorem~\ref{thm: West Pennes} hold.  For sufficiently small radii $\Rp>0$ and $\Rtheta>0$, the mapping $\mathcal{T}$ is strictly contractive with respect to  $\|(\cdot, \cdot)\|_{\Yp \times \Ytheta}$.  
	\end{proposition} 
	\begin{proof}
The proof is conducted by testing the difference problem with $(\opt, \otheta)$. Testing equation \eqref{Equation_West_diff} with $\opt$, integrating with respect to $t$ and using \eqref{assumption coercivity kernel}, we obtain 
\begin{equation}\label{Energy_Ident_p}
\begin{aligned}
&\frac{1}{2}\Big(\big\|\sqrt{\frakmone} \opt \big\|_{\Ltwo}^2 + \big\|\sqrt{\fraknone} \nabla\op \big\|_{\Ltwo}^2\Big) +\CfrakK \int_0^{t} \big\|(\frakK* \nabla \opt )(s)\big\|^2_{\Ltwo} \ds \\
\leq&\,
\begin{multlined}[t]
\frac{1}{2}\inttO \frakmone_t |\opt|^2\dxs+\frac{1}{2}\inttO \fraknone_t |\nabla\op|^2\dxs + \inttO  \opt \nabla \fraknone \cdot \nabla  \op\dxs\\
+\inttO \fraklone |\opt |^2\dxs+\inttO F\,\opt \dxs.
\end{multlined}
\end{aligned}
\end{equation}    
We estimate the first term on the right-hand side of \eqref{Energy_Ident_p} as follows: 
\begin{equation}
\begin{aligned}
\frac{1}{2}\inttO \frakmone_t |\opt|^2\dxs\leq&\,\frac12 \intt \big\|\frakmone_t\big\|_{\Linf}  \| \opt\|^2_{\Ltwo} \ds\\
		\leq &\, \frac12\intt \frac{1}{\sqrt{s}} \big\|\sqrt{s}\frakmone_t\big\|_{\Linf}  \| \opt\|^2_{\Ltwo} \ds.
	\end{aligned}
\end{equation}
Similarly, 
\begin{equation}
\begin{aligned}
\frac{1}{2}\inttO \fraknone_t |\nabla\op|^2\dxs\leq&\,
		 \frac12\intt \frac{1}{\sqrt{s}} \big\|\sqrt{s}\fraknone_t\big\|_{\Linf}  \| \nabla\op\,\|^2_{\Ltwo} \ds.
	\end{aligned}
\end{equation}
Next, we estimate the third term on the right-hand side of \eqref{Energy_Ident_p} as
\begin{equation}
\inttO  \opt \nabla \fraknone \cdot \nabla  \op\dxs \leq \frac12 \int_0^t \big\|\nabla\fraknone \big\|_{\Linf} \Big(\|\opt\|_{\Ltwo}^2+\|\nabla\op\|_{\Ltwo}^2\Big)\ds.
\end{equation}
The fourth term is estimated as 
\begin{equation}
\inttO \fraklone |\opt |^2\dxs \lesssim \intt \big\|\fraklone \big\|_{\Linf}\|\opt\|_{\Ltwo}^2\ds.
\end{equation}
We estimate the last term as
	\begin{equation}
	\inttO F\opt \dxs\leq \intt \|F\|_{\Ltwo}\|\opt\|_{\Ltwo}\ds \lesssim \|F\|^2_{\LonetLtwo} + \eps \|\opt\|_{\LinftLtwo}^2
\end{equation}
for any $\eps>0$. 
By plugging the above estimates into \eqref{Energy_Ident_p}, making use of \eqref{nondegeneracy assumption frakm frakn}, taking the supremum over $t \in (0,\tau)$ for $\tau \in (0,T)$, and fixing $\eps>0$ small enough, we obtain 
\begin{equation}\label{estimate op}
\begin{aligned}
&\underbrace{\| \opt(t) \|_{\Ltwo}^2 + \| \nabla\op(t) \|_{\Ltwo}^2}_{\displaystyle:=E_0[\op](t)} + \,\CfrakK \int_0^{t} \big\|(\frakK* \nabla \opt )(s)\big\|^2_{\Ltwo} \ds \\
\lesssim &\,
\begin{multlined}[t]
\|F\|^2_{\LonetLtwo} + \Big(1+\big\|\sqrt{t}\frakmone_t\big\|_{\LinfLinf} + \big\|\sqrt{t}\fraknone_t\big\|_{\LinfLinf}\Big)\\
\times\intt \left(\frac{1}{\sqrt{s}} + \big\|\nabla\fraknone \big\|_{\Linf} + \big\|\fraklone \big\|_{\Linf}\right)E_0[\op](s)\ds,
\end{multlined}
\end{aligned}
\end{equation}
for $t \in [0,T]$.  
Let 
\begin{equation}
\begin{aligned}
\overline{\Lambda}(t):= & \,
\begin{multlined}[t]
\Big(1 + \big\|\sqrt{t}\frakmone_t\big\|_{\LinfLinf} + \big\|\sqrt{t}\fraknone_t\big\|_{\LinfLinf}\Big)\\
\times \Big(\big\|\nabla\fraknone \big\|_{\LoneLinf} + \big\|\fraklone \big\|_{\LoneLinf}\Big).
\end{multlined}
\end{aligned}  
\end{equation}
By Lemma~\ref{estimates frakm frakn frakl},
\[
\overline{\Lambda}(t) \leq C(T), \ t \in [0,T]. 
\]
Thanks to Lemma~\ref{lemma: est F}, we have
\begin{equation}
	\begin{aligned}
		\|F\|_{\LonetLtwo} \lesssim \sqrt{T}\Rp (\|\opstar\|_{\LtwotHone}+\|\opstart\|_{\LtwotLtwo}+\|\othetastar\|_{\LtwotHone}).
	\end{aligned}
\end{equation}
To estimate the difference in the temperatures, we test the equation for $\otheta$  in \eqref{Equation_West_diff}   with $\otheta$, which leads to
\begin{equation} \label{estimate otheta}
	\begin{aligned}
		&\frac12 \|\otheta(t)\|^2_{\Ltwo}+ \frac{\kappa}{2}\intt \|\nabla \otheta\|^2_{\Ltwo}\ds  \\
		=&\,- \nu \inttO \wbonestar \otheta^2 \dxs + \inttO G \otheta \dxs \\
		\lesssim&\, \intt \|\wbonestar\|_{\Linf}\|\otheta\|^2_{\Ltwo}\ds +\|G\|_{\LonetLtwo}^2+\eps \|\otheta\|_{\LinftLtwo}
	\end{aligned}
\end{equation}
for any $\eps>0$, where we recall that $G$ is defined in \eqref{def G} as
	\begin{equation} 
	\begin{aligned}
		G=&\,  -\nu (\wbonestar-\wbtwostar) \thetatwostar + \zeta \frac{\bonestar}{(\qonestar)^2}(\pone_{t})^2- \zeta \frac{\btwostar}{(\qtwostar)^2}(\ptwo_{t})^2.
	\end{aligned}
\end{equation}
Since we can rewrite $G$ as
	\begin{equation} 
	\begin{aligned}
		G
		=&\, \begin{multlined}[t]-\nu (\wbonestar-\wbtwostar) \thetatwostar + \zeta \frac{\bonestar}{(\qonestar)^2}\opt(\pone_{t}+\ptwo_t)+ \zeta \left(\frac{\bonestar}{(\qonestar)^2}-\frac{\btwostar}{(\qtwostar)^2} \right)(\ptwo_{t})^2,
		\end{multlined}
	\end{aligned}
\end{equation}
we have
\begin{equation}
	\begin{aligned}
		\|G\|_{\LonetLtwo} \lesssim \sqrt{T} (\Rtheta \|\othetastar\|_{\LtwotLtwo} + \|\opt\|_{\LtwotLtwo})
	\end{aligned}
\end{equation} 
By adding \eqref{estimate op} and \eqref{estimate otheta} and then applying \Gronwall's inequality, we obtain
\begin{equation}
	\begin{aligned}
		\|\op\|^2_{\Yp} + \|\otheta\|^2_{\Ytheta} \leq C(T) T (\Rp + \Rtheta) (\|\opstar\|^2_{\Yp}+ \|\othetastar\|^2_{\Ytheta}).
	\end{aligned}
\end{equation}
Thus the contractivity of $\calT$ follows by reducing $\Rp$ and $\Rtheta$.
\end{proof}
	
	We now have all the ingredients to prove the well-posedness of the nonlocal Westervelt--Pennes system \eqref{coupled_problem West Pennes}.
	
	\begin{proof}[proof of Theorem~\ref{thm: West Pennes}]
	The statement now follows by employing Lemma~\ref{lemma: self-mapping} and Proposition \ref{prop: contractivity}, together with Banach's fixed-point theorem.  
	\end{proof}
		\section{Estimates of the concentration} \label{sec: WestPennes concentration}
The previous section established the existence of a unique solution of the Westervelt--Pennes system. We now consider the full system modeling ultrasound-enhanced drug delivery, which we restate here for convenience:  
\begin{equation} \tag{\ref{coupled_system West Pennes concentration}}
				\boxed{
					\begin{aligned}
						&\ p_{tt}-q(\theta+\Thetaa)\Delta p - b(\theta+\Thetaa) \Delta p_t = k(\theta+\Thetaa)\left(p^2\right)_{tt}+f_p \quad &&\text{in} \ \Omega \times (0,T), \\[2mm]
						&\ \theta_t -\kappa\Delta \theta+  \wb(\theta+\Thetaa) \nu \theta = \calG(\pt, b(\theta+\Thetaa), q(\theta+\Thetaa))+f_\theta &&\text{in} \ \Omega \times (0,T), \\[2mm]
						&\ \ct+ \nabla \cdot (\bfv(p, \nabla p) c)- \nabla \cdot (D(p, \theta+\Thetaa) \nabla c) = \fc \,&&\text{in} \ \Omega \times (0,T),
					\end{aligned} 
				}
\end{equation}
			with
			\begin{equation} \tag{\ref{coupled_problem_BC_WPc}}
				p\vert_{\partial \Om}=0, \qquad \theta\vert_{\partial \Om}= 0, \qquad c\vert_{\partial \Om}=0,
			\end{equation}
			and the initial data 
			\begin{equation} \tag{\ref{coupled_problem_IC_WPc}}
				(p, p_t)\vert_{t=0}= (p_0, p_1), \qquad \theta \vert_{t=0}=\Theta_0- \Thetaa:= \theta_0, \qquad c\vert_{t=0}=c_0.
			\end{equation}
				Given $(p, \theta) \in \solspaceWestPennes$ that solves the Westervelt--Pennes subsystem, we now focus on 
				\begin{equation}\label{c_Equation}
					 \ct + \nabla \cdot (\bfv(p, \nabla p) c)- \nabla \cdot (D(p, \theta+\Thetaa) \nabla c) = \fc 
				\end{equation}
				with $c\vert_{\partial \Om}=0$ and $c\vert_{t=0}=c_0$. We next prove the main result of this work.
	
				\begin{proof}[Proof of Theorem~\ref{thm: West Pennes concentration}]
				
			The well-posedness of the Westervelt--Pennes system in \eqref{coupled_problem West Pennes concentration} follows by Theorem~\ref{thm: West Pennes}.	The well-posedness for the concentration follows similarly to \cite[Theorem 6.3]{hunterLectureNotesPartial} using a Faedo--Galerkin approach;   we include here the derivation of the estimate.	Multiplying the third equation by $c$ and integrating over $\Omega$, using integration by parts, we obtain 
				\begin{equation}\label{Energy_Est_c}
					\frac{1}{2}\ddt \|c\|^2_{\Ltwo}+\int_\Omega D(p, \theta+\Thetaa)\nabla c\cdot \nabla c\dx=-\int_\Omega \bfv(p, \nabla p) c\cdot \nabla c\dx+\int_\Omega c \fc \dx
				\end{equation}
				a.e. in time. We have
				\begin{equation}
				\int_\Omega \bfv(p, \nabla p) c\cdot \nabla c\dx \leq \|\bfv(p, \nabla p)\|_{\Lfour}\|c\|_{\Lfour}\|\nabla c\|_{\Ltwo}.
				\end{equation}
					Using Ladyzhenskaya's inequality:
				\begin{equation}\label{Lady_Inequality_2}    
					\| c\|_{L^4(\Omega)} \leq C \|\nabla c\|_{L^2(\Omega)}^{d/4}\|c\|_{L^2(\Omega)}^{1-d/4}, 
				\end{equation}
				we obtain 
				\begin{equation}
					\begin{aligned}
						\left|\int_\Omega \bfv(p, \nabla p) c\cdot \nabla c\dx\right|\leq\,&\, \|\bfv(p, \nabla p)\|_{\Lfour}\|c\|_{\Ltwo}^{d/4}\|\nabla c\|_{\Ltwo}^{2-d/4}.
					\end{aligned}
				\end{equation}
		 		Then Young's inequality implies    
				\begin{equation}
					\begin{aligned}
						\left|\int_\Omega \bfv(p, \nabla p) c\cdot \nabla c\dx\right|\lesssim \,&\, \|\bfv(p, \nabla p)\|_{\Lfour}^4\|c\|_{\Ltwo}^{2}+\eps_0 \|\nabla c\|_{\Ltwo}^{2}
					\end{aligned}
				\end{equation}
			for any $\eps_0>0$.	We can further bound the velocity
				\begin{equation} 
					\begin{aligned}
					 \|\bfv(p, \nabla p)\|_{\Lfour} \lesssim&\, 1+\|p\|_{\Lfour} + \|\nabla p\|_{\Lfour}
					 \lesssim\, 1+\|p\|_{\Hone} + \| p\|_{\Htwo} \\
					 \lesssim&\, 1+ \Rp.
					\end{aligned}
				\end{equation} 
		Furthermore, again on account of Young's inequality, we have
				\begin{equation}
					\begin{aligned}
						\intt \int_\Omega c \fc \dxs \leq  \frac{1}{4 \eps}\|\fc\|^2_{\LonetLtwo}+ \eps \|c\|^2_{\LinftLtwo}
					\end{aligned}
					\end{equation}
				for any $\eps>0$. Thus, after also integrating in time in \eqref{Energy_Est_c} and employing 
				\begin{equation} \label{splitting}
					\begin{aligned}
						\inttO D(p, \theta+\Thetaa) \nabla c \cdot \nabla c \dxs= 	\inttO (D_0 + \calD(p, \theta)) \nabla c \cdot \nabla c \dxs,
					\end{aligned}
				\end{equation}
			on account of Assumption~\ref{assumptions: D v},	 we arrive at
				\begin{equation} \label{interim est concentration}
					\begin{aligned}
						&\|c(s)\|^2_{\Ltwo} \big \vert_0^t + \intt D_0 \|\nabla c(s)\|^2_{\Ltwo} \ds \\
						\lesssim&\, \begin{multlined}[t]\intt (1+\Rp)\|c\|^2_{\Ltwo}\ds + \eps_0 \intt\|\nabla c\|^2_{\Ltwo}\ds+\|\fc\|^2_{\LoneLtwo}+ \eps \|c\|^2_{\LinftLtwo}\\+	\intO \calD(p, \theta) \nabla c \cdot \nabla c \dx.
							\end{multlined}
					\end{aligned}
				\end{equation}
				We can then further estimate the $\calD$ term on the right-hand side as follows:
				\begin{equation}\label{D_estimate}
					\begin{aligned}
		\intO	 \calD(p, \theta) \nabla c \cdot \nabla c \dx \leq&\, C (\|p\|_{\Linf}+\|\theta\|_{\Linf}) \intO \nabla c \cdot \nabla c \dx \\
		\leq&\,  C(\Rp+\Rtheta) \intO  |\nabla c|^2\dx.
		\end{aligned}
				\end{equation}
			Above we have relied on Assumption~\ref{assumptions: D v} to conclude that
				\[
				\begin{aligned}
					\|\calD(p, \theta) - \calD(0,0)\|_{\Linf} \lesssim \|(p, \theta)\|_{\Linf}.
				\end{aligned}
				\]

		For sufficiently small $\Rp$, $\Rtheta$, and $\eps_0$, we can guarantee that
				\begin{equation}
					\begin{aligned}
						\intO D_0 \nabla c \cdot \nabla c \dx \geq d_0 \intO |\nabla c|^2 \dx> C(\Rp+\Rtheta) \intO  |\nabla c|^2\dx+ \eps_0 \intt\|\nabla c\|^2_{\Ltwo}\ds.
					\end{aligned}
				\end{equation}
				After making use of this estimate in \eqref{interim est concentration}, we arrive at  
				\begin{equation}
					\begin{aligned}
						&\|c(s)\|^2_{\Ltwo} \big \vert_0^t + \intt \|\nabla c(s)\|^2_{\Ltwo} \ds \\
						\lesssim&\, \intt (1+\Rp)\|c\|^2_{\Ltwo}\ds+\|\fc\|^2_{\LoneLtwo}+ \eps \|c\|^2_{\LinftLtwo}.
					\end{aligned}
				\end{equation} 
				Taking the supremum over $t \in (0,\tau)$ for $\tau \in (0,T)$ and then fixing $\eps$ small and employing \Gronwall's inequality leads to
				\begin{equation}\label{c_Main_Estimate}
					\begin{aligned}
						\|c(\tau)\|^2_{\Ltwo} + \int_0^{\tau} \|\nabla c(s)\|^2_{\Ltwo} \ds \lesssim_T \|\czero\|^2_{\Ltwo} + \|\fc\|^2_{\LoneLtwo}.
					\end{aligned}
				\end{equation}
By using this estimate (in a semi-discrete Galerkin setting), it can be shown that there exist a unique $c \in \Xc$ that solves the advection-diffusion equation supplemented with \eqref{coupled_problem_BC_WPc} and \eqref{coupled_problem_IC_WPc}. Note that also $\ct \in \LoneTHneg$ can be shown using duality arguments, analogously to e.g.~\cite[Ch.\ 7]{evansPartialDifferentialEquations2010}. 
\end{proof}

\newcommand{\testw}{\phi_{h}}
\newcommand{\testh}{\phi_{h}}
\newcommand{\testc}{\phi_{h}}
\newcommand{\femspace}{\mathbb{P}_l^{\rm cont}(\Omega)}

\section{Numerical simulation of the Westervelt--Pennes--concentration system} \label{sec: Numerics}

In this section, we discuss numerical approximation of the studied pressure-temperature-concentration system, and present several numerical experiments that highlight different nonlinear phenomena. The memory kernel in all experiments is taken to be $\frakK(t) = \frac{1}{\Gamma(1-\alpha)}t^{-\alpha}$, so that the acoustic dissipation involves the Caputo--Djrbashian fractional derivative. \\
\indent To illustrate the properties of system~\eqref{coupled_problem original system}, we consider a two-dimensional spatial setting. For semi-discretization in space, we employ an $H^1$-conforming finite element  approximation for the pressure, temperature, and concentration using a regular mesh triangulation of the domain $\Omega$. For the time integration, we employ a fixed-point iteration for the wave equation together with a Newmark scheme, analogously to~\cite[Sec.\ 5.4]{kaltenbacherNumericalSimulationMechatronic2015}. In combination, to discretize the Caputo--Djrbashian derivative, we employ an L1-type scheme, following~\cite{kaltenbacherFractionalTimeStepping2022}. The temperature-pressure terms are treated in a semi-implicit manner so that the heat equation can be solved sequentially.
We then employ a semi-implicit discretization in time for the heat  and implicit Euler discretization for the concentration equation.  The details are provided below in Section~\ref{Sec: details scheme}. The numerical experiments are implemented in the finite element library FEniCSx~\cite{barattaDOLFINxNextGeneration2023}. \footnote{All codes are available at
\href{https://github.com/juliocareaga/wave-heat-mass.git}{https://github.com/juliocareaga/wave-heat-mass.git}
}
\subsection{Numerical scheme} \label{Sec: details scheme} For completeness, we describe here in more detail the fully discrete numerical scheme we adopt. Let $\femspace$ be the space of piecewise polynomials of total degree less than or equal to $l\geq 1$ on the triangulation of $\Omega$, which are continuous across the edges of each triangle. In addition, given a time step size $\tau>0$, we define the discrete time $t^{n} = n\tau$ for all $n\in\mathbb{N}$, and use below the superscript notation $(\cdot)^n$  to denote time evaluations. \\

\noindent $\bullet$ {Westervelt's equation}.  Following~\cite[Sec.\ 5.4]{kaltenbacherNumericalSimulationMechatronic2015}, we use a predictor-corrector form of the Newmark time stepping scheme for the Westervelt equation. The predictor variables for the Newmark scheme with positive parameters $(\beta, \gamma)$ at time $t = t^{n+1}$ are defined as follows:
\begin{align}
	\widetilde{p}_h^{n+1}     & = p_h^{n} + \tau p_{h,t}^n + (\tfrac{1}{2} - \beta)\tau^2  p_{h,tt}^{n}, \\
	\widetilde{p}_{h,t}^{n+1} & = p_{h,t}^{n} + (1-\gamma)\tau p_{h,tt}^{n},
\end{align}
and the respective corrector variables are
\begin{align}
	p_h^{n+1}     & = \widetilde{p}_h^{n+1} + \beta \tau^2 p_{h,tt}^{n+1}, \label{eq:corrector:p}\\
	p_{h,t}^{n+1} & = \widetilde{p}_{h,t}^{n+1} + \gamma\tau p_{h,tt}^{n+1} \label{eq:corrector:pt}.
\end{align}
 Following \cite{kaltenbacherFractionalTimeStepping2022}, the non-local acoustic dissipation term in \eqref{coupled_system West Pennes concentration} is approximated by an L1-type scheme based on
 \begin{equation}
 	\begin{aligned}
 		\Dtalpha \ph(t_{n+1}) =&\, \frac{1}{\Gamma(1-\alpha)} \sum_{j=0}^{n}\int_{t_j}^{t_{j+1}} \pht(s)(t_{n+1}-s)^{-\alpha} \ds \\
 		\approx&\,  \frac{1}{\Gamma(1-\alpha)} \sum_{j=0}^{n} \frac{\pht(t_j)+\pht(t_{j+1})}{2}\int_{t_j}^{t_{j+1}}(t_{n+1}-s)^{-\alpha} \ds \\
 		=&\,  \frac{1}{2\Gamma(2-\alpha)} \sum_{j=0}^{n} \left(\pht(t_j)+\pht(t_{j+1})\right)\left((t_{n+1}-t_j)^{1-\alpha}-(t_{n+1}-t_{j+1})^{1-\alpha}\right),
 	\end{aligned}
 \end{equation}
 from which we have
  \begin{equation}
 	\begin{aligned}
 		\Dtalpha \ph(t_{n+1})
 		\approx &\, \begin{multlined}[t] \frac{1}{2\Gamma(2-\alpha)} \sum_{j=0}^{n} \pht(t_j)\left((t_{n+1}-t_j)^{1-\alpha}-(t_{n+1}-t_{j+1})^{1-\alpha}\right)\\
 			+	\frac{1}{2\Gamma(2-\alpha)} \sum_{j=1}^{n+1} \pht(t_{j})\left((t_{n+1}-t_{j-1})^{1-\alpha}-(t_{n+1}-t_{j})^{1-\alpha}\right). \end{multlined}
 	\end{aligned}
 \end{equation}
 We thus employ the following discrete convolution:
\begin{align}
	\tau^{\alpha-1}\Delta \Dtalpha \ph(t_{n+1})
	\approx&\,  \zeta_0^{n+1}\Delta p_{h,t}^{n+1} + \zeta_1^{n+1}\Delta p_{h,t}^{n} + \zeta_2^{n+1}\Delta p_{h,t}^{n-1}
	+ \cdots + \zeta_{n+1}^{n+1}\Delta p_{h,t}^{0},
\end{align}
where
\begin{equation}
\begin{aligned}
	\zeta_j^{n+1} := \dfrac{1}{2\Gamma(2-\alpha)}
\scriptscriptstyle
\begin{cases}
\	1  &\text{ if }j= 0,\\
\	(j+1)^{1-\alpha} - (j-1)^{1-\alpha}& \text{ if }j = 1,\dots,n,\\
\	(n+1)^{1-\alpha} - n^{1-\alpha} & \text{ if }j=n+1,
\end{cases}
\end{aligned}
\end{equation}
Using the corrector expression in \eqref{eq:corrector:pt}, we have
\begin{align}
	\tau^{\alpha-1}\Delta \Dtalpha \ph(t_{n+1})
	\approx&\,  \zeta_0^{n+1}\Delta\big(\widetilde{p}^{n+1}_{h,t} + \tau \gamma p_{h,tt}^{n+1}\big) + \sum_{j=1}^{n+1}\zeta_j^{n+1}\Delta p_{h,t}^{n+1-j}\\
	=&\, \tau\gamma\zeta_0^{n+1} \Delta p_{h,tt}^{n+1} + \Delta\big(\zeta_0^{n+1}\widetilde{p}_{h,t}^{n+1} +  \Upsilon_{h}^n\big),
\end{align}
where we have introduced the short-hand notation
\begin{align}
	& \Upsilon_h^n := \Upsilon_h^n\big(p^0_{h,t},p^1_{h,t},\dots,p^n_{h,t}\big)=  \sum_{j=1}^{n+1}\zeta_j^{n+1} p_{h,t}^{n+1-j}. %
\end{align}
Then, given $p_h^{n}, \theta_{h}^{n}, c_{h}^{n}\in\femspace$ and additionally $ p_{h,t}^{n}, p_{h,tt}^{n}\in \femspace$, the iterative procedure from time $t^n$ to $t^{n+1}$, is the following. In the first step, we solve the nonlinar wave equation in \eqref{coupled_system West Pennes concentration} through a fixed-point iteration procedure. For the $\ell$-th iteration, we set the auxiliary variables $p_*^{\ell} = \widetilde{p}_{h}^{n+1}$ and $p_{*t}^{\ell} = \widetilde{p}_{h,t}^{n+1}$ and solve the following linearized equation for the second derivative $p_{*tt}^{\ell+1}\in \femspace$: 
\begin{align}
	\begin{aligned}
		&\begin{multlined}[t]	\left(  \widetilde{\frakm}(\theta_h^{n},p_*^{\ell}) p_{*tt}^{\ell+1},  \testw\right)_{L^2} +
			\beta	\tau^2\left(  \nabla p_{*tt}^{\ell +1}, \nabla \left(
			\widetilde{q}(\theta_h^n) \testw \right)\right)_{L^2}
			+\gamma	\tau^{2-\alpha}\left(  \nabla p_{*tt}^{\ell +1}, \nabla \left(
			(  \widetilde{b}(\theta_h^n) \zeta_0^{n+1} )\testw \right)\right)_{L^2}
		\end{multlined} \\
		&=\, \begin{multlined}[t] \Big(2\widetilde{k}(\theta_h^n)(p_{*t}^{\ell})^2 + \widetilde{b}(\theta_h^n)f_{p}^{n+1}, \testw\,\Big)_{L^2}
			- \Big(\nabla \widetilde{p}_h^{n+1},  \nabla\big(\widetilde{q}(\theta_h^{n})\testw\big)\,\Big)_{L^2}\\ \hspace*{3.2cm}
			- \tau^{1-\alpha} \Big(\nabla\big(\zeta_0^{n+1}\widetilde{p}_{h,t}^{n+1} + \Upsilon_{h}^n\big), \nabla\big(\widetilde{b}(\theta_h^n)\testw\big)\, \Big)_{L^2} + \int_{\partial\Omega} \big( \mathbf{\Phi}_{p,h}^{n+1}\cdot\boldsymbol{n}\big)\testw\,{\rm d}\sigma
		\end{multlined}
	\end{aligned}
\end{align}
for all $\testw\in\femspace$, where we have used the short-hand notation $\widetilde{\frakm}(\theta,p) = 1 - 2\widetilde{k}(\theta)p$, with $\widetilde{b}(\theta) = b(\theta+\Thetaa)$, $\widetilde{q}(\theta) = q(\theta+\Thetaa)$, and $\widetilde{k}(\theta) = k(\theta + \Thetaa)$.
The function $\mathbf{\Phi}_{p,h}^{n+1}$, defined on the boundary $\partial\Omega$, groups the remaining boundary terms related to the integration by parts, and is specified later according to the Neumann boundary conditions stated in Example 1. Vector $\boldsymbol{n}$ is the outward unit normal vector on $\partial\Omega$. Note that the numerical scheme for the Westervelt equation is consistent with \eqref{Westervelt_eq} prior to dividing by $b(\Theta)$. \\
\indent Then, the auxiliary variables are updated using the corrector formulae \eqref{eq:corrector:p} and \eqref{eq:corrector:pt}, replacing $p_{h,tt}^{n+1}$ by $p_{*tt}^{\ell +1}$:
\begin{equation}
	p_*^{\ell + 1}  = \widetilde{p}_{h}^{n+1} + \beta \tau^2 p_{*tt}^{\ell + 1},\qquad
	p_{*t}^{\ell + 1} = \widetilde{p}_{h,t}^{n+1} + \gamma \tau p_{*tt}^{\ell + 1}.
\end{equation}
The fixed-point iterations are carried out until the following stopping criterion is reached:
\begin{equation}
	\dfrac{\|p_{*tt}^{\ell +1} - p_{*tt}^{\ell}\|_{\Ltwo }}{\|p_{*tt}^{\ell+1}\|_{\Ltwo}} < \texttt{tol},
\end{equation}
and the values updated by setting $p_{h}^{n+1} = p_{*}^{\ell+1}$, $p_{h,t}^{n+1} = p_{*t}^{\ell+1}$, and $p_{h,tt}^{n+1} = p_{*tt}^{\ell +1}$.\\

\noindent $\bullet$ {Pennes equation}.  For the Pennes equation, we employ an implicit Euler approximation for the time derivative of $\theta$ combined with a semi-implicit discretization for the $\theta$-dependent functions $\omega_{\rm b}$ and $\mathcal{G}$.
Then, making use of $p_{h,t}^{n+1}$, we solve the temperature equation for $\theta_{h}^{n+1}\in \femspace$ as follows:
\begin{align}
	&( \theta_h^{n+1},\testh\,)_{L^2} + \tau \kappa \left( \nabla\theta_h^{n+1},\nabla\testh \,\right)_{L^2} + \tau\left( \nu \omega_{\rm b}(\theta_h^n + \Theta_{\rm a})\theta_h^{n+1},\testh\,\right)_{L^2}\\
	 =&\,  \left( \theta_h^{n} + \tau \mathcal{G}\left(p_{h,t}^{n+1}, \theta_h^n\right) + \tau f_{\theta}^{n+1},\testh\,\right)_{L^2}
\end{align}
for all $\testh\in\femspace$, where we have assumed homogeneous Neumann boundary conditions for $\theta$ on $\partial\Omega$. The same homogeneous Neumann boundary condition used in all numerical examples.\\

\noindent $\bullet$  {Equation for the concentration}.   To solve the concentration equation, we once again use an implicit Euler approximation for the time derivative of $c$, and use the updated heat and pressure to solve the discrete mass balance for $c_h^{n+1}\in \femspace$ as follows:
\begin{align}
	&( c_h^{n+1},\testc\,)_{L^2} + \tau\left(c_h^{n+1}\boldsymbol{v}(p_h^{n+1},\nabla p_h^{n+1}) -  D(p_h^{n+1}, \theta_h^{n+1} + \Theta_{\rm a})\nabla c_h^{n+1}, \nabla\testc\,\right)_{L^2}\\
	 =&\,  \left(c_h^{n} + \tau f_{c}^{n+1}, \testc\,\right)_{L^2} + \int_{\Gamma_{\rm b}} \big(\mathbf{\Phi}_{c,h}^{n+1}\cdot\boldsymbol{n}\big)\testc\,{\rm d}\sigma
\end{align}
for all $\testc\in\femspace$, where $\mathbf{\Phi}_{c,h}^{n+1}$ contains the boundary terms on $\partial\Omega$ related to the concetration $c$ to be specified in Example 1.\\

In all the numerical examples presented in this section, we use the Newmark parameters $\beta = 0.45$ and $\gamma = 0.85$, piecewise linear polynomials as basis functions (that is, we take $l=1$), and the fractional order of differentiation $\alpha = 0.8$. The tolerance for the fixed-point method is set to $\texttt{tol} = 10^{-12}$.

\subsection{Example 1: Effects of high-intensity focused ultrasound} \label{numex:1}

In this example, we simulate the transport of a chemical substance (drug) enhanced by ultrasound waves through soft tissue. The medium in this scenario is assumed to be liver tissue, and the two-dimensional domain $\Omega$ is delimited by the rectangle $[-0.04\,{\rm m},0.04\,{\rm m}]\times[0, 0.12\,{\rm m}]$ joint to a circular section as shown in Figure~\ref{fig:domain}. The curved part of the boundary, denoted by $\Gamma_{\rm b}$, corresponds to a circle with center at $x_{\rm c} = (0.04\,{\rm m}, 0.03 \,{\rm m})$ and radius $R = 0.05\,{\rm m}$. The domain is discretized with a mesh  of $40192$ triangular elements, and the time step chosen for the example is $\tau = 6.67\times 10^{-8}\,\rm s$. \vspace*{-4mm}
\begin{figure}[h]
	\includegraphics[scale=0.6]{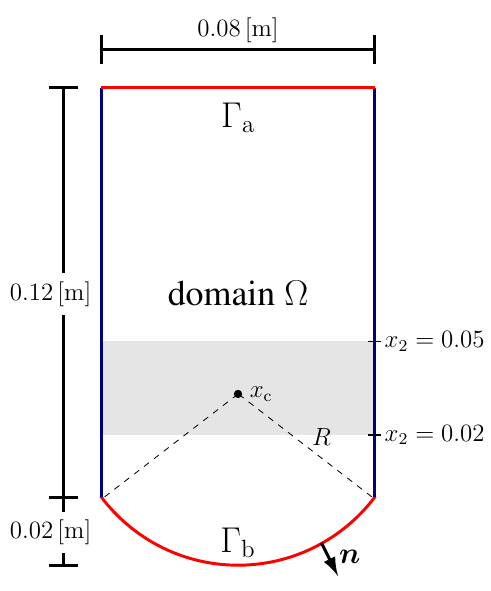}
	\caption{Computational domain $\Omega$: The top boundary at $x_2 = 0.12\,\rm m$ is denoted by $\Gamma_{\rm a}$, and the curved bottom boundary for $x_2\leq 0$ is symbolized by $\Gamma_{\rm b}$, and $\boldsymbol{n}$ corresponds to the unit normal exterior to $\Omega$. The gray zone corresponds to the region for $0.02\,{\rm m}\leq x_2 \leq 0.05\,\rm m$.\label{fig:domain}}
\end{figure}

\indent Compared to the theoretical setting, we choose here boundary conditions that allow for a more realistic simulation scenario.
On the curved boundary, a high frequency excitation is applied according to the Neumann boundary condition
\begin{align}
	\nabla p\cdot\boldsymbol{n} = g,\quad \mbox{on }\Gamma_{\rm b},
\end{align}
dictated by the time-dependent function
\begin{align}
	g(t) = \begin{cases}
		g_0 \sin(\omega t) (1 + \sin(\omega t/4)) & \text{ if } t >2\pi/\omega,\\
		g_0 \sin(\omega t) & \text{ if } t \le 2\pi/\omega,
	\end{cases}
\end{align}
where $\omega = 2\pi \frakf$ is the angular frequency, with $\frakf = 100\, {\rm kHz}$ the frequency of excitation, and $g_0 = 1000\,{\rm MPa}$ is the amplitude of oscillation. The source of this boundary condition could be, for example, a transducer located outside the curved boundary. 
The Neumann boundary condition is then incorporated to the discrete weak formulation of the Westervelt equation through the function $\mathbf{\Phi}_{p,h}^{n+1}$, which is defined according to
\begin{align}
 \mathbf{\Phi}_{p,h}^{n+1}\cdot \boldsymbol{n} = g(t^{n+1})q(\Theta_{\rm a}) + \tau^{1-\alpha} b(\Theta_{\rm a})\sum_{j=0}^{n+1}\zeta_{j}^{n+1}g'(t^{n+1-j})\quad \text{on }\Gamma_{\rm b},
\end{align}
and $\mathbf{\Phi}_{p,h}^{n+1}\cdot\boldsymbol{n} = 0$ on $\partial\Omega\setminus\Gamma_{\rm b}$.
In addition, we assume that the drug concentration is being constantly injected on the curved boundary, through the inflow boundary condition
\begin{align}
	 (c\boldsymbol{v} - D\nabla c)\cdot \boldsymbol{n} = \widetilde{g}\quad \mbox{on }\Gamma_{\rm b},
\end{align}
where $\widetilde{g} = 0.01\,{\rm kg\,m^{-2}\,s}$. Then, the corresponding discrete flux is $\mathbf{\Phi}_{c,h}^{n+1}\cdot \boldsymbol{n} = \widetilde{g}$ on $\Gamma_{\rm b}$, and zero-flux boundary conditions are used outside $\Gamma_{\rm b}$, i.e., $\mathbf{\Phi}_{c,h}^{n+1}\cdot\boldsymbol{n} = 0$ on $\partial\Omega\setminus\Gamma_{\rm b}$.
\\
\indent The temperature-dependent speed of sound squared, sound diffusivity, and mass flow rate of blood per unit volume of tissue are respectively taken to be
\begin{align}
	q(\Theta) & = \big(1529.3 + 1.6856\Theta + 6.1131\times 10^{-2} \Theta^2\\
	&\qquad - 2.2967 \times 10^{-3} \Theta^3 + 2.2657 \times 10^{-5}\Theta^4 - 7.1795 \times 10^{-8} \Theta^5\big)^2,\\
	b(\Theta) & = \dfrac{2\alpha_0}{\omega^2}(q(\Theta))^{3/2},\\
	\omega_{\rm b}(\Theta) & = 10^{-4}(5 + \Theta).
\end{align}
Function $q$ corresponds to the squared speed of sound in liver tissue reported in \cite{connorBioacousticThermalLensing2002}, and the linear blood perfusion function $\omega_{\rm b}$ is taken from \cite{dengParametricStudiesPhase2000}. The ambient temperature is set to $\Theta_{\rm a} = 37\, ^\circ{\rm C}$, and the rest of parameters corresponding to the wave and heat equations are taken from \cite[Table~3]{connorBioacousticThermalLensing2002} to be
\begin{align}
	&\beta_{\rm a} = 6\,{\rm kg^3\,m^{-4}\,s^{-2}}, 
	\quad \rho_{\rm a} =1050\, {\rm kg\,m^{-3}}, \quad \rho_{\rm b} = 1030\, {\rm kg\,m^{-3}},\quad  \alpha_0  = 4.5\times 10^{-6}f\,{\rm Np \, m^{-1}}, \\
	&C_{\rm a}  = 3600\, {\rm J\,kg\, K^{-1}},\quad C_{\rm b} = 3620\,{\rm J\,kg\, K^{-1}},\quad \kappa_{\rm a} = 0.512\, {\rm W\,m^{-1}\,K^{-1}},
\end{align}
and the remaining constants in the heat equation are set to
\begin{align}
	\kappa & =  \dfrac{\kappa_{\rm a}}{\rho_{\rm a}C_{\rm a}},\qquad \nu = \dfrac{\rho_{\rm b}C_{\rm b}}{\rho_{\rm a}C_{\rm a}}, \qquad \tilde{\zeta} = 2.
\end{align}
Concerning the concentration equation, we note that the time scale
reported in related works (see, e.g., \cite{Sefidgar2015,azhdariFickianNonFickianTransports2022}) is several orders of magnitude larger than the wave-heat sub-system; that is, the wave equation requires a much smaller time step to be used than the concentration equation. In \cite{Sefidgar2015}, the convective velocity, which is described by Darcy's law, posses a factor of $10^{-15}\,{\rm m^2\, Pa^{-1}\, s^{-1}}$ multiplying the gradient of the pressure, while the constant diffusion is $10^{-12}\,\rm m^2/s$. Hence, in order to capture the influence of the pressure on the transport equation, we take the following convective velocity and diffusion functions:
\begin{align}
	\boldsymbol{v}(p,\nabla p) & = -k_{\rm D}\nabla p,\qquad
	D(p,\theta)  = D_0,
\end{align}
where $ k_{\rm D} = 10^{-6}\,{\rm m^2\, Pa^{-1}\, s^{-1}}$ and $D_0 = 5\,{\rm m^2\, s^{-1}}$.

Figure~\ref{fig:example1:1} shows the numerical simulation of the pressure field until the final time $T= 0.0001\, \rm s$. We observe the self-focusing as the ultrasound wave travel with the peak values, whose precise modeling is important in therapeutic applications, in the focal region.  \\

\begin{figure}[h]
	\centering
	\begin{tabular}{cc}
		\includegraphics[scale=0.8]{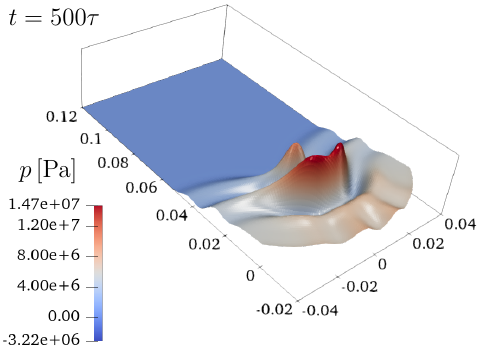} &
		\includegraphics[scale=0.8]{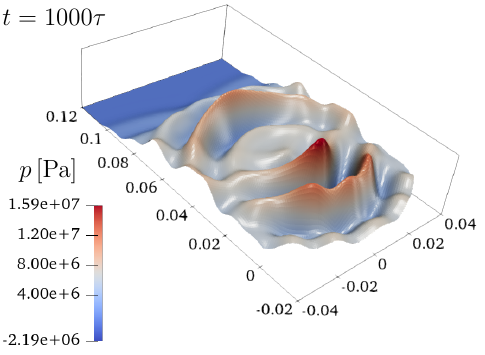}\\[2ex]
		\includegraphics[scale=0.8]{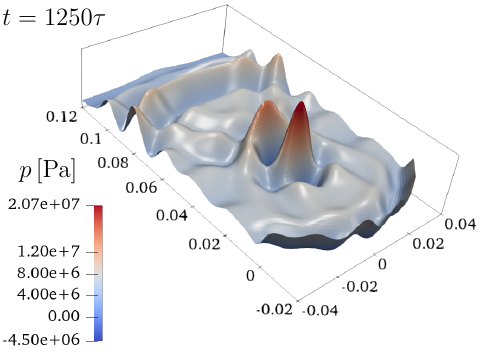} &
		\includegraphics[scale=0.8]{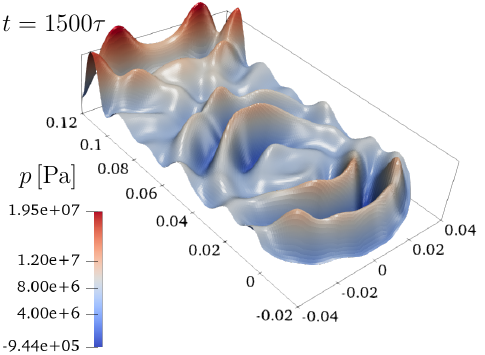} 
	\end{tabular}
	\caption{Example 1: Pressure waves at four times $t=500i\tau$, where $i = 1,2,3/2,3$ with $\tau = 6.67\times 10^{-8}\,\rm s$\label{fig:example1:1}}
\end{figure}
Figure~\ref{fig:example1:1 axis} illustrates the difference of having $k=0$ and $k=k(\Theta)$ in the acoustic wave equation. We observe that the latter setting with Westervelt-type acoustic nonlinearities leads to higher positive peak pressure values. 

\begin{figure}
	\centering
	\begin{tabular}{cc}
			\multicolumn{2}{c}{\includegraphics[scale=0.4]{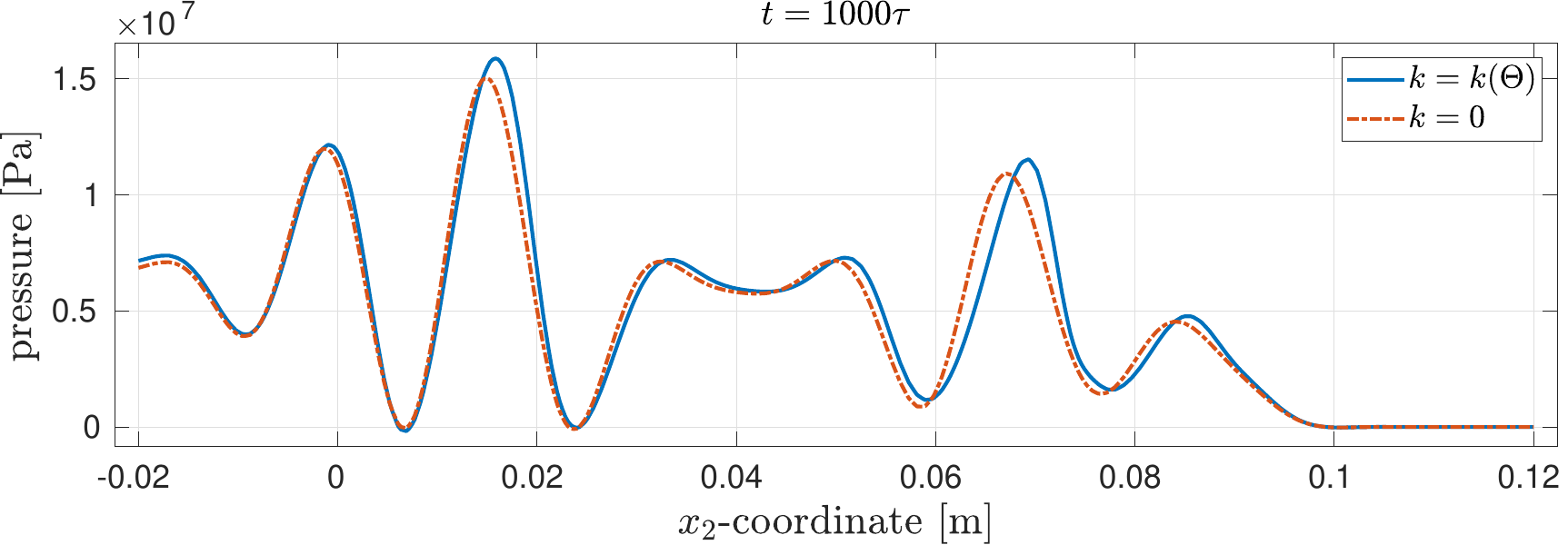}}
	\end{tabular}
\caption{Example 1: Pressure over the axis of symmetry $x_1 = 0$ at the time $t = 1000\,\tau$ obtained using the nonlinear Westervelt equation ($k = k(\Theta)$) and the linear wave equation ($k=0$), respectively.\label{fig:example1:1 axis}}

\end{figure}

\begin{figure}[h]
	\centering
	\begin{tabular}{ccc}
		\includegraphics[scale=0.78]{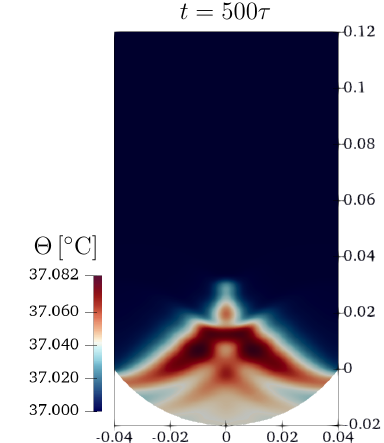} &
		\hspace{-0.3cm}\includegraphics[scale=0.78]{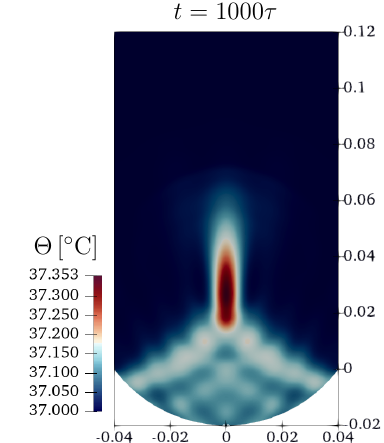} &
		\hspace{-0.3cm}\includegraphics[scale=0.78]{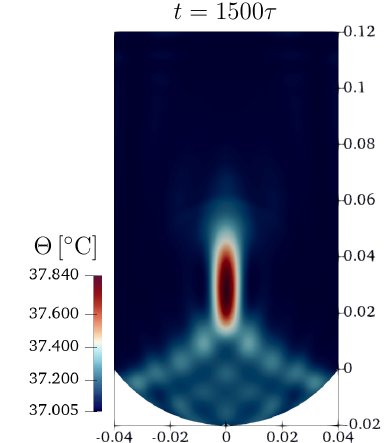}\\[1ex]
		\includegraphics[scale=0.78]{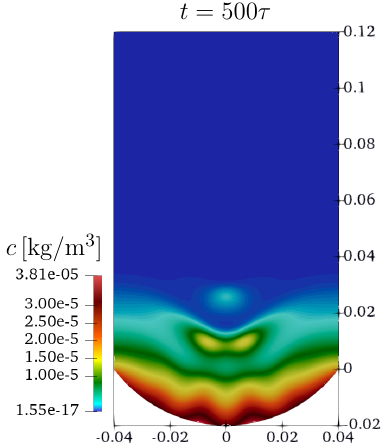} &
		\hspace{-0.3cm}\includegraphics[scale=0.78]{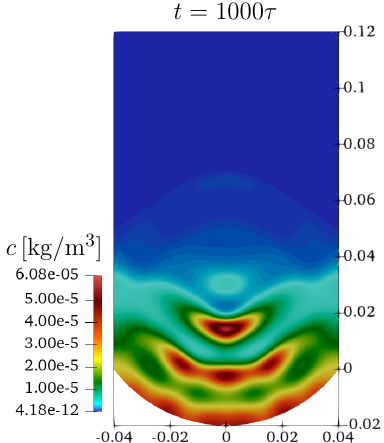} &
		\hspace{-0.3cm}\includegraphics[scale=0.78]{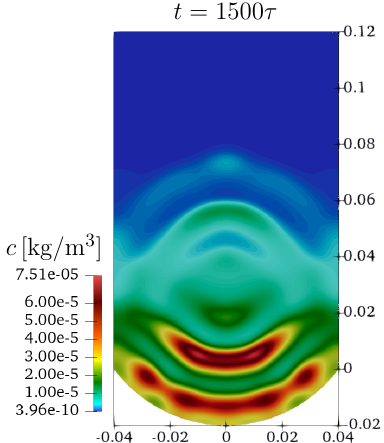}\\[1ex]
	\end{tabular}
	\caption{Example 1: Snapshots of temperature and concentration profiles (top and bottom row, respectively) at times $t= 500i\tau$ for $ i= 1,2,3$ with $\tau=6.67\times 10^{-8}\,\rm s$.\label{fig:example1:2}}
\end{figure}

\indent  The temperature and concentration fields are shown in Figure~\ref{fig:example1:2}. We see that they increase from $\Gamma_{\rm b}$ toward the focal region and both tend to decrease afterward. This observation aligns with the fact that ultrasound intensity is the highest in the focal area.

\subsection{Example 2:  Thermal lensing} \label{numex:2}
In the second example, we study the influence of having the temperature-dependent coefficients in the  wave equation. More precisely, we consider two cases, the Westervelt equation with temperature-dependent coefficients, and the case of no temperature variations, i.e., $\Theta(x,t) \equiv \Theta_{\rm a}$ and $q \equiv q(\Thetaa)$, $b\equiv b(\Thetaa)$, and $k \equiv k(\Thetaa)$. \\
\indent For the two cases, we use the same parameters $\omega$, $\frakf$, $g_0$, $\widetilde{g}$, $\beta_{\rm a}$, $\rho_{\rm a}$, $\rho_{\rm b}$, $\alpha_{0}$, $C_{\rm a}$, $C_{\rm b}$, $\kappa_{\rm a}$, $\tilde{\zeta}$, $k_{\rm D}$, $D_0$, and $\tau$, and functions introduced in Example 1. \\
\indent Figure~\ref{fig:example2} shows the acoustic pressure plots at six different time points, once every 250 time iterations, for the two cases. The left half of the domain corresponds to the case under temperature variations, i.e., the fully coupled PDE system \eqref{coupled_problem original system}, while the right half is the case of constant temperature. From Figure~\ref{fig:example2}, it is clearly seen that the wave propagates faster in the case of variable temperature. As a consequence of that, thermal lensing can be observed; that is, a shift in the focal are of the ultrasound waves. 

\begin{figure}[h]
	\centering
	\begin{tabular}{ccc}
		\includegraphics[scale=0.78]{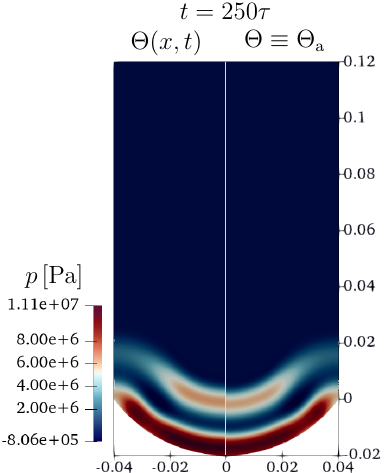} &
		\hspace{-0.3cm}\includegraphics[scale=0.78]{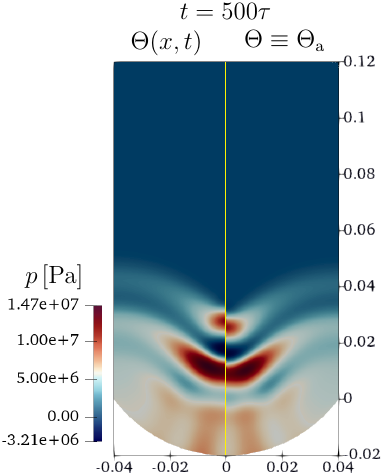} &
		\hspace{-0.3cm}\includegraphics[scale=0.78]{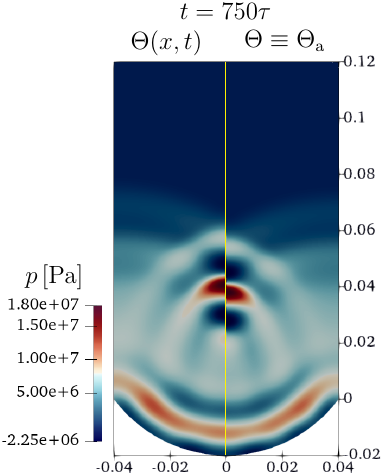}\\[1ex]
		\includegraphics[scale=0.78]{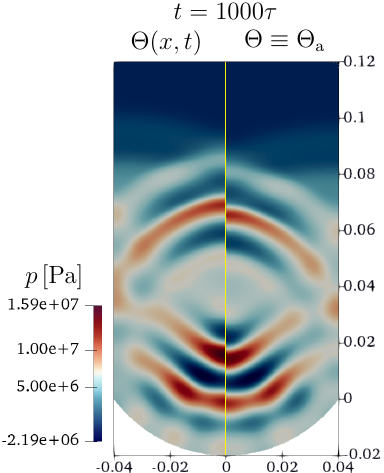} &
		\hspace{-0.3cm}\includegraphics[scale=0.78]{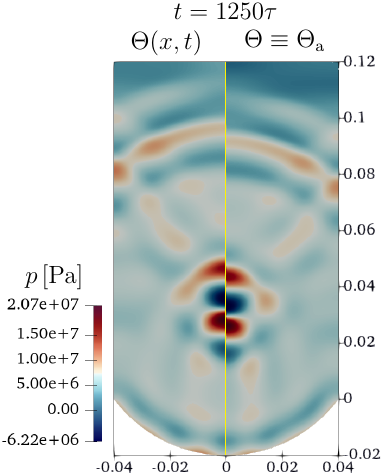} &
		\hspace{-0.3cm}\includegraphics[scale=0.78]{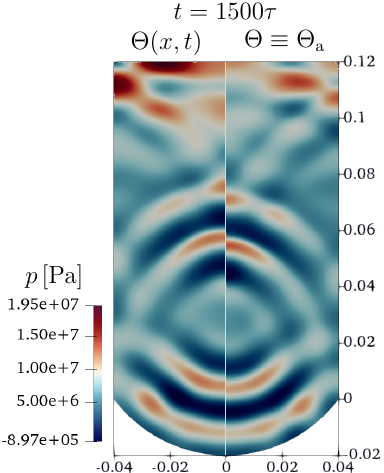}\\[1ex]
	\end{tabular}
	\caption{Example 2 (Thermal lensing): Snapshots of pressure profiles for two different scenarios, with variation in temperature (left halves) and constant temperature $\Theta = \Theta_{\rm a}$ (right halves)
		at times $t= 250i\tau$ for $i=1,2,\dots,6$, with $\tau=6.67\times 10^{-8}\,\rm s$.\label{fig:example2}}
\end{figure}

\subsection{Example 3: Changes in the mass variation} \label{numex:3}

In the spirit of the numerical study of the pressure-concentration system in~\cite{azhdariFickianNonFickianTransports2022}, we examine also the total mass variation defined by
\[
m(t) = \int_{\tilde{\Omega}} c(\cdot,t) \dx
\]
 in two cases, including the ultrasound effect and without ultrasound. For comparison purposes, we compute the mass on the whole domain $\tilde{\Omega}=\Omega$ and on a subset $\tilde{\Omega}$ of $\Omega$ containing the focal region, taken to be the rectangle where $x_2\in (0.02, 0.05)\,\rm m$; see the gray zone in Figure~\ref{fig:domain}. We consider the same scenario as in Example 1, but now we have a different boundary condition for the concentration equation at the top boundary $\Gamma_{\rm a}$; see Figure~\ref{fig:domain}.
 We set 
\begin{align}
	\mathbf{\Phi}_{c,h}^{n+1} \cdot \boldsymbol{n} = -100c_h^{n+1} \quad \text{on }\Gamma_{\rm a}.
\end{align}
For the inflow boundary condition at $\Gamma_{\rm b}$, we set $\widetilde{g} = 5\times 10^{-3}\, \rm kg\,m^{-2}$. The effect of having the two chosen boundary conditions is that mass is being injected at the bottom as well as being extracted at the top. For the convective velocity $\boldsymbol{v}$, we incorporate a constant term $v_0$ here, such that
\begin{align}
	\boldsymbol{v}(\nabla p) = \boldsymbol{v}_0 - k_{\rm D}\nabla p,
\end{align}
with $\boldsymbol{v}_0 = (0,10)^{T}\,{\rm m/s}$, and $k_{\rm D}$ as in Example 1. For the diffusion term we use a constant function with $ D_0 = 5.0\,\rm m^2/s$. Unlike the previous examples, the initial condition for the concentration is taken as $c_0(x,0) = 10^{-4}\,\rm kg/m^3$ for all $x\in \Omega$. Moreover, we set the time step to $\tau = 10^{-6}\,\rm s$ and final time of simulation to $T = 0.0005\,\rm s$. The constants and functions for the wave and heat equations are the same as in Example 1.

\begin{figure}[h]
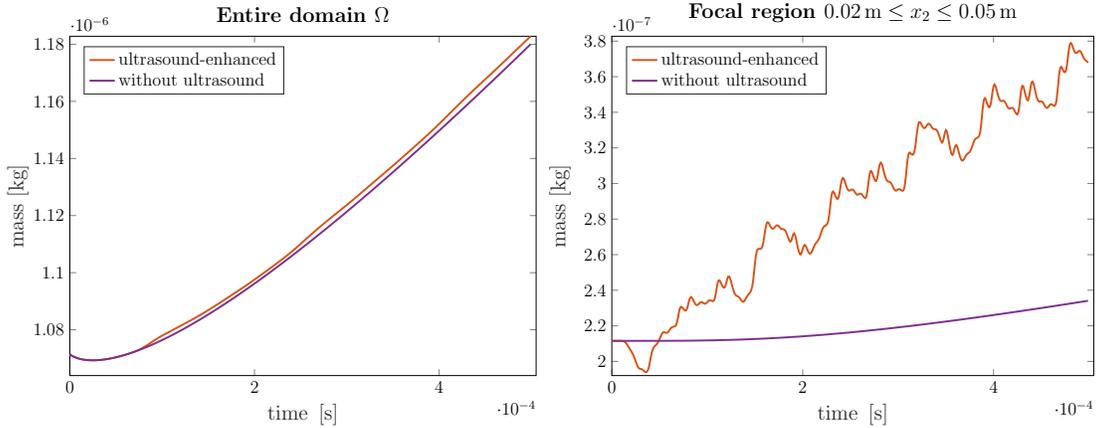

	\centering
	\begin{tabular}{cc}
		\scalebox{0.56}{\input{IMG/plotmass1_2curves.tex} }\hspace*{-0.1cm}
		\scalebox{0.56}{\input{IMG/plotmass7_2curves.tex}}
	\end{tabular}
	\caption{Example 3 (The influence of ultrasound):  Plots of mass $m=m(t)$ over the entire domain $\Omega$ (left) and focal region $0.02\,{\rm m}\leq x_2\leq 0.05\,{\rm m}$ (right) with respect to time for the case with ultrasound enhancement and without ($\boldsymbol{v} = \boldsymbol{v}_0$).\label{fig:example3}}
\end{figure}

The two scenarios we compare are the coupled case with nonlinear ultrasound enhancement modeled by the studied system and the case without ultrasound enhancement, in which $p\equiv 0$ and $\nabla p = \boldsymbol{0}$ in the concentration equation.  Figure~\ref{fig:example3} shows the total mass in the entire domain $\Omega$ as a function of time and the mass in the focal region $\tilde{\Omega}$ where $0.02{\,\rm m}\leq x_2\leq 0.05{\,\rm m}$. Due to the boundary conditions on $\Gamma_{\rm b}$ and $\Gamma_{\rm a}$, the total mass (Figure~\ref{fig:example3}; left plot) in both cases increases over time with the exception of a small interval of decrease at the beginning. When comparing the ultrasound-enhanced and non-ultrasound setting, there is a small increase in the total mass with ultrasound-enhancement when mass is measured over the whole domain $\Omega$. However, if we consider the focal region, which is of practical interest in therapeutic applications and where the ultrasound is focused on, this increase is considerably larger compared to the case without ultrasound effect.


\section*{Conclusions and outlook}	
In this work, we have investigated the Westervelt--Pennes--concentration system with nonlocal acoustic attenuation, motivated by ultrasound-enhanced drug delivery. The local existence analysis of this nonlinear multiphysics problem was based on a fixed-point argument for the Westervelt--Pennes sub-problem combined with an-energy based analysis of its linearization.  As the nonlocal acoustic damping is relatively weak, sufficiently smooth acoustic test functions have to be employed in the analysis, leading to a higher regularity requirement on the initial pressure data compared to the setting of strong acoustic damping. On the other hand, parabolic evolution in the Pennes equation allowed us to devise, in part, a time-weighted testing strategy, thereby reducing the smoothness assumptions on the initial temperature. \\
\indent Numerical experiments have further reinforced the importance of accounting for different nonlinear effects in modeling. As we have seen, having Westervelt-type nonlinearities in the pressure equation resulted in higher positive peak pressures in the focal region. The temperature dependency in the acoustic parameters allows for capturing the thermal lensing effect where the position of the focal region may shift as a result of ultrasound-induced heating of the medium. Finally, the experiments also demonstrate that the ultrasound has a significant effect on the concentration mass in the focal region. \\ 
\indent The present work provides a sound basis for further research on nonlinear modeling of ultrasound-enhanced transport phenomena. There are various theoretical and computational open questions of interest to pursue.
\begin{enumerate}[leftmargin=*] 
\item From the theoretical side, extending the analysis to the setting of Neumann acoustic boundary conditions used here in numerics or Neuman-absorbing boundary conditions which minimize spurious reflections from the computational boundary is important, but technically more intricate. The latter mixed acoustic setting is particularly challenging as it prevents the use of high smoothness arguments in the pressure field that are invoked by the time-nonlocal damping.  Allowing for heterogeneous media to model healthy and non-healthy tissue is likewise of practical interest but poses a challenge for similar reasons, as  the theoretical framework employed here would not allow for low spatial regularity of acoustic parameters. 
Determining whether the local solution in Theorem \ref{thm: West Pennes concentration} can be extended to a  global one and examining the asymptotic behavior in time is of theoretical interest as well. Establishing global behavior of time-fractional evolution equations is a challenging task in general; see, e.g.,~\cite{vergaraOptimalDecayEstimates2015, kemppainen2016decay} for such results for certain single-physics problems.  Time-nonlocal acoustic damping in the Westervelt--Pennes-concentration system could turn out to be too weak for this purpose, however, and devising stabilization mechanisms might be needed.  
\item Numerically, the different time scales within the wave and heat equations imply having to solve the Westervelt equation a large number of times for each time iteration of the concentration equation.  Combined with the nonlocal term in the wave equation, this may lead to excessive memory consumption  in realistic scenarios.  Thus, developing advanced numerical schemes for handling fractional derivatives in order to reduce memory consumption is particularly challenging and of great practical interest.  Furthermore, advanced multi-rate time-stepping approaches might be considered as a way of tackling different time scales.  In addition, establishing rigorous \emph{a priori} estimates of the developed schemes is worth considering for future works.
Finally, a comparison of numerical results with experimental data is also among the computational challenges to be tackled in the future.
\end{enumerate}

\section*{Acknowledgments}
J.C. is supported by ANID through Fondecyt project 3230553.
\begin{appendices}
	\section{Linear pressure estimates} \label{Appendix: linear pressure estimates}
\noindent We include in this appendix the derivation of estimate \eqref{est pn} for the Galerkin approximations  used in the proof of Proposition~\ref{prop: linear pressure higher order}.	
\begin{proof}[Derivation of estimate \eqref{est pn}]
	In deriving \eqref{est pn}, the arguments are analogous to those of~\cite[Proposition 3.1]{kaltenbacher2024limiting} with modifications needed to incorporate having the variable coefficient $\frakn$ and to include time weights when estimating the arising $\frakm$ and $\frakn$ terms. For simplicity of notation, we omit the superscript $n$ below and use $p$ instead of $\pn$ everywhere. \\
	\indent By testing the semi-discrete problem with  $\Delta^2\pt$, we obtain
	\begin{equation} \label{identity_1}
		\begin{aligned} 
			(\frakm \ptt -\frakn \Delta p-  \frakK*\Delta \pt-\frakl_t \pt,\,  \Delta^2 \pt)_{L^2} = (\fp, \Delta^2 \pt)_{L^2}.
		\end{aligned}
	\end{equation}
	We can then rely on the following identity:
	\begin{equation}\label{Identity_m}
		\begin{aligned}
			&(\frakm \ptt, \Delta^2 \pt)_{L^2} \\
			=&\,  (\frakm\D \ptt, \Delta \pt)_{L^2}+(\ptt\, \D\frakm +2\nabla \ptt \cdot\nabla\frakm, \Delta \pt)_{L^2} \\
			=&\,  \frac12\ddt(	\frakm\D \pt, \Delta \pt)_{L^2}-\frac12(\frakm_t\D \pt, \Delta \pt)_{L^2}+(\ptt \, \D\frakm+2\nabla \ptt \cdot\nabla\frakm, \Delta \pt)_{L^2} 
		\end{aligned}
	\end{equation}
	because $\ptt \vert_{\partial \Omega} =\Delta \pt \vert_{\partial \Omega}=0$.  Similarly, we have
	\begin{equation}\label{n_estimate}
		\begin{aligned}
			(\frakn \Delta p, \Delta^2 \pt)_{L^2}
			=&\,  	( \nabla(\frakn \Delta p), \nabla \Delta \pt)_{L^2}\\
			=&\, ( \frakn \nabla \Delta p, \nabla \Delta \pt)_{L^2}+( \Delta p \nabla \frakn , \nabla \Delta \pt)_{L^2}\\
			=&\,
			\frac12\ddt\|	\sqrt{\frakn}\nabla\D p\|_{L^2}^2-\frac12(\frakn_t\nabla\D p, \nabla \Delta p)_{L^2}+( \Delta p \nabla \frakn , \nabla \Delta \pt)_{L^2}.			
		\end{aligned}
	\end{equation}
	Further, we can treat the arising $\frakl$ terms as follows:
		\begin{equation}\label{l_identity}
		\begin{aligned}
			(  \frakl \pt, \Delta^2 \pt )_{L^2}=&\,  (\D (\frakl \pt), \Delta \pt )_{L^2}\\
			=&\, ( \frakl\D \pt +2 \nabla \frakl \cdot \nabla \pt +  \pt \D \frakl, \Delta \pt )_{L^2}.
		\end{aligned}
	\end{equation}
	Thus, by integrating over $(0,t)$ for $t \in (0,T)$ \eqref{identity_1} and employing the coercivity assumption on $\frakK$ in \eqref{assumption coercivity kernel}, we arrive at the following energy inequality:
	\begin{equation} \label{energy_ineq_1}
		\begin{aligned}
			&\frac12\nLtwo{\sqrt{\frakm} \Delta \pt}^2 \Big \vert_0^t + \frac{1}{2} \nLtwo{\sqrt{\frakn}\nabla\D p}^2 \Big \vert_0^t +
			\CfrakK \int_0^{t} \|(\frakK*\nabla \D \pt )(s)\|^2_{\Ltwo} \ds \\
			\leq&\,\begin{multlined}[t] 
				\frac12 \int_0^t (\frakm_t\D \pt, \Delta \pt)_{L^2}\ds+ \intt(\ptt \, \D\frakm+2\nabla \ptt \cdot\nabla\frakm, \Delta \pt)_{L^2} \ds \\
				+ \intO \nabla \fp \cdot \nabla \D p \dx \Big \vert_0^t - \inttO \nabla \fpt \cdot \nabla \Delta p \dxs- \inttO \Delta p \nabla \frakn \cdot \nabla \Delta p \dxs \\
				+\frac12\intt (\frakn_t\nabla\D p, \nabla \Delta p)_{L^2} \ds+ \intt ( \frakl\D \pt +2 \nabla \frakl \cdot \nabla \pt +  \pt \D \frakl, \Delta \pt )_{L^2} \ds,
			\end{multlined}
		\end{aligned}
	\end{equation}
	where we have also used that $\fp|_{\partial \Om} = 0.$
	We can estimate the first integral on the right-hand side by relying on the embedding $\Htwo \hookrightarrow \Linf$ as follows:
	\begin{equation} \label{est_aaa_t}
		\begin{aligned}
			\frac12 \int_0^t (\frakm_t\D \pt, \Delta \pt)_{L^2}\ds
			\leq&\, \frac12 \intt \|\frakm_t\|_{\Linf}  \|\Delta \pt\|^2_{\Ltwo} \ds\\
			\leq&\, \frac12\intt \frac{1}{\sqrt{s}} \|\sqrt{s}\frakm_t\|_{\Linf}  \|\Delta \pt\|^2_{\Ltwo} \ds\\
		\end{aligned}
	\end{equation} 
	where the resulting term will be handled by Gr\"onwall's inequality later on. Note that here we have introduced $\|\sqrt{s}\frakmt\|_{\Linf}$ as $\frakm$ will in the fixed-point argument involve the temperature field which is handled using time weights; see~Proposition~\ref{prop: linear temperature higher order}. Similarly,
	\begin{equation}
		\begin{aligned}
			&\intt ( \frakl\D \pt +2 \nabla \frakl \cdot \nabla \pt +  \pt \D \frakl, \Delta \pt )_{L^2} \ds \\
			\lesssim&\,\begin{multlined}[t]  \|\frakl\|_{\LinfLinf} \|\Delta \pt\|^2_{\LtwotLtwo} + \|\nabla \frakl\|_{\LinfLthree} \|\nabla \pt\|_{L^2_t(\Lsix)} \|\Delta \pt\|_{\LtwotLtwo} \\+ \|\D \frakl\|_{\LinfLtwo} \|\pt\|_{L^2_t(\Linf)}\|\Delta \pt\|_{\LtwotLtwo}
			\end{multlined} \\
			\lesssim&\, \|\frakl\|_{\Xfrakl} \|\Delta \pt\|^2_{\LtwotLtwo}.
		\end{aligned}
	\end{equation}
	Further,
	\begin{equation}
		\begin{aligned}
			\frac12\intt (\frakn_t\nabla\D p, \nabla \Delta p)_{L^2} \ds \leq&\,\frac12 \intt \frac{1}{\sqrt{s}}\|\sqrt{s} \frakn_t\|_{\Linf} \|\nabla \Delta p\|^2_{\Ltwo}\ds. 
		\end{aligned}
	\end{equation}
	We can use H\"older's and Young's inequalities  $f$ terms in \eqref{energy_ineq_1}:
	\begin{equation} \label{estimates_f}
		\begin{aligned}
			&
			\intO \nabla \fp \cdot \nabla \D p \dx \Big \vert_0^t - \inttO \nabla f_t \cdot \nabla \Delta p \dxs\\
			\leq&\, \begin{multlined}[t]  \frac{1}{4 \eps}\|\nabla \fp(t)\|^2_{\Ltwo}+\eps \|\nabla \Delta p(t)\|_{\Ltwo}^2+\frac12 \|\nabla \fp(0)\|^2_{\Ltwo}+ \frac12\|\nabla \Delta p_0\|_{\Ltwo}^2 \\+ 	\|\nabla \fpt\|^2_{L^2(\Ltwo)}+  \|\nabla \Delta p\|_{L^2_t(\Ltwo)}^2, \end{multlined}
		\end{aligned}
	\end{equation}
	where the two resulting $\nabla \D p$ terms above can be either absorbed by the left-hand side of \eqref{energy_ineq_1} or handled by Gr\"onwall's inequality. Additionally, 
	\begin{equation}
		\begin{aligned}
			\inttO \Delta p \nabla \frakn \cdot \nabla \Delta p \dxs \lesssim&\, \|\nabla \frakn\|_{\LinfLthree}\|\Delta p\|_{L^2_t(\Lsix)} \|\nabla \Delta p\|_{\LtwotLtwo} \\
			\lesssim&\, \|\frakn\|_{\Xfrakn} \|\nabla \Delta p\|^2_{\LtwotLtwo},
		\end{aligned}
	\end{equation}
	which can be also tackled via \Gronwall's inequality. \\
	\indent	It remains to estimate the $\ptt$ terms on the right-hand side of \eqref{energy_ineq_1}. To this end, we first use H\"older's and Young's inequalities:
	\begin{equation} \label{est_utteps}
		\begin{aligned}
			&\intt(\ptt \, \D\frakm+2\nabla \ptt \cdot\nabla\frakm, \Delta \pt)_{L^2} \ds \\
			\leq&\,  \left(\|\ptt\|_{L^2(\Lfour)}\|\Delta \frakm\|_{L^\infty(\Lfour)} + \|\nabla \ptt\|_{L^2(\Ltwo)}\|\nabla \frakm\|_{L^\infty(\Linf)} \right)\|\Delta \pt\|_{L^2(\Ltwo)}.
		\end{aligned}
	\end{equation}
	By the embeddings $\Hone \hookrightarrow \Lfour$ and $\Htwo \hookrightarrow \Linf$, we further have
	\begin{equation}
		\begin{aligned}
			\|\Delta \frakm\|_{L^\infty(\Lfour)} +	\|\nabla \frakm\|_{L^\infty(\Linf)} \lesssim 	\|\Delta \frakm\|_{L^\infty(\Hone)} +	\|\nabla \frakm\|_{L^\infty(\Htwo)} \lesssim \|\frakm\|_{\Xfrakm}
		\end{aligned}
	\end{equation}
	and thus
	\begin{equation} \label{est_utt_interim}
		\begin{aligned}
			\intt(\ptt \, \D\frakm+2\nabla \ptt \cdot\nabla\frakm, \Delta \pt)_{L^2} \ds 
			\leq\,   \|\frakm\|_{\Xfrakm} \|\ptt\|_{L^2(\Hone)} \|\Delta \pt\|_{L^2(\Ltwo)}.
		\end{aligned}
	\end{equation}
	From here, we can use the (semi-discrete) PDE to further bound the $\ptt$ term. We first have by Poincar\'e's inequality
	\begin{equation}
		\begin{aligned}
			\|\ptt\|_{L^2(\Hone)}  \lesssim \| \nabla \ptt\|_{L^2(\Ltwo)}.
		\end{aligned}
	\end{equation}
	Then to estimate the right-hand side term, we use the identity, obtained after testing with $-\Delta \ptt$:
	\begin{equation} \label{id nabla ptt}
		\begin{aligned}
			(\frakm	\nabla \ptt+ \nabla \frakm \ptt, \nabla \ptt)_{L^2}
			=&\, \begin{multlined}[t]  (\frakK*\nabla \Delta \pt +\frakn\nabla \Delta p+ \nabla \frakn \Delta p + \nabla [\frakl \pt]+\nabla \fp, \nabla \ptt)_{L^2}.
			\end{multlined}
		\end{aligned}
	\end{equation}
	By testing the semi-discrete PDE with $\ptt$, we obtain
	\[
	\|\ptt\|^2_{\LtwoLtwo} \lesssim \| \frakK*\Delta \pt +\frakn\Delta p - \frakl \pt+ \fp\|^2_{L^2(\Ltwo)}.
	\]
	Then from \eqref{id nabla ptt}, we have
	\begin{equation} \label{ineq1}
		\begin{aligned}
			\|\nabla  \ptt\|^2_{L^2(\Ltwo)} 
			\lesssim&\,  \begin{multlined}[t] 
				\|\nabla \frakm\|^2_{L^\infty(\Linf)}\| \frakK*\Delta \pt +\frakn\Delta p - \frakl \pt+ \fp\|^2_{L^2(\Ltwo)}\\
				+ \| \frakK*\nabla \Delta \pt +\frakn\nabla \Delta p+\nabla \frakn \Delta p-  \nabla [\frakl \pt]+\nabla \fp\|^2_{L^2(\Ltwo)}.
			\end{multlined}
		\end{aligned}
	\end{equation}
Poincar\'e's inequality implies that
	\begin{equation}
		\| \frakK*\Delta \pt +\frakn\Delta p - \frakl \pt+ \fp\|_{L^2(\Ltwo)} \lesssim \|\nabla( \frakK*\Delta \pt +\frakn\Delta p - \frakl \pt+ \fp)\|_{L^2(\Ltwo)}.
	\end{equation}
	Further estimating the right-hand side terms in \eqref{ineq1} thus leads to
	\begin{equation}
		\begin{aligned}
			&\|\nabla  \ptt\|_{L^2(\Ltwo)} \\
			\lesssim&\,  \begin{multlined}[t] 
				(1+	\| \frakm\|_{\Xfrakm} )\Big ( \| \frakK*\nabla \Delta \pt\|_{L^2(\Ltwo)} + \|\frakn\|_{\Xfrakn}\| \nabla \Delta p\|_{L^2(\Ltwo)}+ 
				\| \frakl\|_{\Xfrakl} \|\pt\|_{L^2(\Lfour)}\\+ \|\frakn\|_{\Xfrakn}\|\Delta p\|_{\LtwotLtwo}
				+\|\nabla \fp\|_{L^2(\Ltwo)} \Bigr).
			\end{multlined}
		\end{aligned}
	\end{equation}
	Thanks to assumptions on $\frakK$, we have the  bound $\|\frakK\|_{\calM(0,t)} \leq C$.
	Going back to \eqref{est_utt_interim} and using the estimate on $\nabla \ptt$ thus yields
	\begin{equation} \label{est_utt_final}
		\begin{aligned}
			&\intt(\ptt \, \D\frakm+2\nabla \ptt \cdot\nabla\frakm, \Delta \pt)_{L^2} \ds \\
			\lesssim&\,  \begin{multlined}[t]   \|\frakm\|_{\Xfrakm}(1+\|\frakm\|_{\Xfrakm})(1+\|\frakn\|_{\Xfrakn}+\|\frakl\|_{\Xfrakl}) \Big\{ \| \nabla \Delta p\|_{L^2(\Ltwo)}+   \|\D \pt \|_{L^2(\Ltwo)}\\
				+ \| \frakK*\nabla \Delta \pt\|_{L^2(\Ltwo)} + \|\pt\|_{L^2(\Hone)}+\| \fp\|_{L^2(\Hone)}\Big\} \|\Delta \pt\|_{L^2(\Ltwo)}.
			\end{multlined}
		\end{aligned}
	\end{equation}
	Let $\calL(\frakm, \frakn, \frakl) =   \|\frakm\|_{\Xfrakm}(1+\|\frakm\|_{\Xfrakm})(1+\|\frakn\|_{\Xfrakn}+\|\frakl\|_{\Xfrakl})$.
	By employing estimates \eqref{est_aaa_t}, \eqref{estimates_f}, and \eqref{est_utt_final}  in \eqref{energy_ineq_1}, we arrive at
	\begin{equation} \label{energy_ineq_2}
		\begin{aligned}
			&\frac12\nLtwo{\sqrt{\frakm} \Delta \pt}^2 \Big \vert_0^t + \frac{1}{2} \nLtwo{\sqrt{\frakn}\nabla\D p}^2 \Big \vert_0^t +\CfrakK \int_0^{t} \|(\frakK*\nabla \D \pt )(s)\|^2_{\Ltwo} \ds \\
			\lesssim &\,\begin{multlined}[t] (1+\|\frakn\|_{\Xfrakn}+\calL^2(\frakm, \frakn, \frakl))\Big( \|\Delta \pt\|^2_{L^2_t(\Ltwo)} 
				\\+ \|\nabla \Delta p\|^2_{L^2_t(\Ltwo)}+   \|\D \pt \|^2_{L^2_t(\Ltwo)} + \|\pt\|^2_{L^2_t(\Hone)}\Bigr)+\|\nabla \Delta p_0\|^2_{\Ltwo} \\+\|\fp\|^2_{{H^1(\Hone)}} 
				+	 \calL(\frakm, \frakn, \frakl)\| \frakK*\nabla \Delta \pt\|_{L^2_t(\Ltwo)}  \|\Delta \pt\|_{L^2_t(\Ltwo)}\\
				+\intt \frac{1}{\sqrt{s}}( \|\sqrt{s}\frakm_t\|_{\Linf} +\|\sqrt{t}\frakn_t\|_{\Linf}) \|\Delta \pt\|^2_{\Ltwo} \ds.
			\end{multlined}
		\end{aligned}
	\end{equation} 
	Other than the last term in \eqref{energy_ineq_2}, all other terms on the right-hand side can be tackled using Gr\"onwall's inequality. To treat the last term, we employ Young's inequality:
	\begin{equation} \label{est_Keps_gamma}
		\begin{aligned} 
			&  \calL(\frakm, \frakn, \frakl)\|  \frakK*\nabla \Delta \pt\|_{L^2(\Ltwo)}  \|\Delta \pt\|_{L^2(\Ltwo)} \\
			\leq&\,  \gamma	\| \nabla \Delta \frakK* \pt\|^2_{L^2_t(\Ltwo)}+\frac{1}{4 \gamma } \calL^2(\frakm, \frakn, \frakl)\|\Delta \pt\|^2_{L^2_t(\Ltwo)} .
		\end{aligned}
	\end{equation}
	If $\tilde{C}>0$ is the hidden constant within $\lesssim$ in \eqref{energy_ineq_2}, we can choose $\gamma$ as
	\[
	\gamma = \frac{1}{\tilde{C}}\CfrakK /2.
	\]
	The term $\gamma	\| \nabla \Delta \frakK* \pt\|^2_{L^2_t(\Ltwo)}$ can be absorbed by the left-hand side of \eqref{energy_ineq_2}. 
	We then use Gr\"onwall's inequality for the second term on the right-hand side of \eqref{est_Keps_gamma} to arrive at the estimate. 
\end{proof}	
\section{Estimate of the right-hand side in the difference equation for the pressure} \label{Appendix: Estimate F}
We provide here the estimate of the right-hand side of the difference equation for the pressure in \eqref{Equation_West_diff}.
\begin{proof}[proof of Lemma~\ref{lemma: est F}]
	Recall that
		\begin{equation} \tag{\ref{rhs_contractivity}}
		\begin{aligned}
			F = \begin{multlined}[t] -(\frakmone-\frakmtwo) \ptwostartt 
				+(\fraknone-\frakntwo) \Delta \ptwostar   
				+(\fraklone-\frakltwo)\ptwostart.	
			\end{multlined}
		\end{aligned} 
	\end{equation} 
		By employing H\"older's inequality, we find that 
	\begin{equation}
		\begin{aligned}
			\|F\|_{\LoneLtwo}
			\lesssim &\,\sqrt{T} \|F\|_{\LtwoLtwo}\\
			\lesssim&\, \begin{multlined}[t] \sqrt{T} \Bigl(\|\frakmone-\frakmtwo\|_{\LinftLfour} \|\ptwostartt\|_{\LtwoLfour}+ \ \|\fraknone-\frakntwo\|_{\LtwotLfour} \|\Delta \ptwostar\|_{\LinfLfour} \\
				+ \|\fraklone-\frakltwo\|_{\LtwotLtwo}\|\ptwostart\|_{\LinfLinf} \Bigr). \end{multlined}
		\end{aligned}
	\end{equation}
	\noindent Since $(\ptwo, \thetatwo) \in \ball$, we know that $\|\ptwostar\|_{\Xp}\leq \Rp$ and thus from here  we further have 
	\begin{equation}
		\begin{aligned}
			\|F\|_{\LoneLtwo} 
			\lesssim&\, \begin{multlined}[t] \sqrt{T} \Rp \bigl(\|\frakmone-\frakmtwo\|_{\LinftHone} +  \|\fraknone-\frakntwo\|_{\LtwotHone}  \\
				+ \|\fraklone-\frakltwo\|_{\LtwotLtwo} \bigr). \end{multlined}
		\end{aligned}
	\end{equation}
We should thus bound the difference terms on the right-hand side further. To this end, we first note that the following identity holds:
	\begin{equation}
		\begin{aligned}
			\frakmone-\frakmtwo =&\, \frac{1}{\bonestar}\left(1-2 \konestar \ponestar \right)-\frac{1}{\btwostar}\left(1-2 \ktwostar \ptwostar \right) \\
			=&\,- \frac{2}{\bonestar} \left(\konestar \ponestar - \ktwostar \ptwostar \right) + \left(\frac{1}{\bonestar}-\frac{1}{\btwostar} \right) \left(1-2 \ktwostar \ptwostar \right)\\
			=&\,- \frac{2}{\bonestar} \left(\konestar \opstar +(\konestar - \ktwostar) \ptwostar \right) +\frac{\btwostar-\bonestar}{\bonestar \btwostar} \left(1-2 \ktwostar \ptwostar \right).
		\end{aligned}
	\end{equation}
	From here, using Lemmas~\ref{lemma: estimates bstar qstar} and~\ref{lemma: estimates kstar} and the fact that $\|\ptwostar\|_{\Xp} \leq \bRp$, we conclude that
	\begin{equation}
		\begin{aligned}
			\|	\frakmone-\frakmtwo \|_{\LtwotLtwo} \lesssim&\, \begin{multlined}[t]\|\opstar\|_{\LtwotLtwo}+ \|\konestar - \ktwostar\|_{\LtwotLtwo} \|\ptwostar\|_{\LinfLinf}
				\\
				+ \|\btwostar-\bonestar\|_{\LtwotLtwo}(1+ \|\ptwostar\|_{\LinfLinf})
			\end{multlined}\\
			\lesssim&\, \|\opstar\|_{\LtwotLtwo}+ \|\konestar - \ktwostar\|_{\LtwotLtwo} + \|\btwostar-\bonestar\|_{\LtwotLtwo}\\
			\lesssim&\, \|\opstar\|_{\LtwotLtwo}+ \|\othetastar\|_{\LtwotLtwo}.
		\end{aligned}
	\end{equation}
	Next, we can exploit the identity
	\begin{equation}  \label{id nabla diff frakm}
		\begin{aligned}
			&\nabla( \frakmone-\frakmtwo)\\
			=&\, \begin{multlined}[t]
				\frac{2}{(\bonestar)^2} (\bonestar)' \nabla \thetaonestar \left(\konestar \opstar +(\konestar - \ktwostar) \ptwostar \right) \\
				- \frac{2}{\bonestar} \left((\konestar)' \nabla \thetaonestar \opstar + \konestar \nabla \opstar +(\konestar - \ktwostar)' \nabla \othetastar \ptwostar+ (\konestar-\ktwostar) \nabla \ptwostar \right) 
				\\ +\nabla \frac{\btwostar-\bonestar}{\bonestar \btwostar} \left(1-2 \ktwostar \ptwostar \right) +\frac{\btwostar-\bonestar}{\bonestar \btwostar} \left(-2 (\ktwostar)'\nabla \thetatwostar \ptwostar+ \ktwostar \nabla \ptwostar \right)
			\end{multlined}
		\end{aligned}
	\end{equation}
	with
	\begin{equation}
		\begin{aligned} 
			&\nabla \frac{\btwostar-\bonestar}{\bonestar \btwostar}\\
			=&\, \frac{((\btwostar)'\nabla \thetatwostar-(\bonestar)'\nabla \thetaonestar)\bonestar \btwostar- (\btwostar-\bonestar)((\bonestar)' \nabla \thetaonestar \btwostar+\bonestar (\btwostar)' \nabla \thetatwostar)}{(\bonestar \btwostar)^2}.
		\end{aligned}
	\end{equation}
	From here, by additionally rewriting 
	\begin{equation}
		(\btwostar)'\nabla \thetatwostar-(\bonestar)'\nabla \thetaonestar=((\btwostar)'-(\bonestar)')\nabla \thetatwostar+(\nabla \thetatwostar-\nabla\thetaonestar)(\bonestar)'
	\end{equation}
	and, since $\|\thetaonestar\|_{\LinfHthree}$, $\|\thetatwostar\|_{\LinfHthree}  \leq \bRtheta$, we have
	\begin{equation}
		\begin{aligned}
		\left	\|	\nabla \frac{\btwostar-\bonestar}{\bonestar \btwostar} \right \|_{\LtwotLtwo}\lesssim&\, \|\othetastar\|_{\LtwotLtwo}+ \|\nabla \othetastar\|_{\LtwotLtwo}.
		\end{aligned}
	\end{equation}	
	By using this bound in \eqref{id nabla diff frakm} and the fact that $\|\ponestar\|_{\LinfHthree}$, $\|\ptwostar\|_{\LinfHthree} \leq \bRp$, we obtain
	\begin{equation}
		\begin{aligned}
			\|\nabla(\frakmone-\frakmtwo)\|_{\LtwotLtwo} \lesssim \|\opstar\|_{\LtwotHone} + \|\othetastar\|_{\LtwotHone}.
		\end{aligned}
	\end{equation}
	To estimate the difference of $\frakn$ terms, we first observe that the following identity holds:
	\begin{equation}
		\begin{aligned}
			\fraknone-\frakntwo = \frac{\qonestar}{\bonestar}-\frac{\qtwostar}{\btwostar} =&\, \frac{\qonestar \btwostar-\qtwostar \bonestar}{\bonestar \btwostar} \\
			=&\, \frac{(\qonestar-\qtwostar) \bonestar-\qonestar( \bonestar-\btwostar)}{\bonestar \btwostar}.
		\end{aligned}
	\end{equation}
	From here, we have
	\begin{equation}
		\begin{aligned}
			\|	\fraknone-\frakntwo\|_{\LtwotLtwo} \lesssim \|\bonestar-\btwostar\|_{\LtwotLtwo}+\|\qonestar-\qtwostar\|_{\LtwotLtwo} \lesssim \|\othetastar\|_{\LtwotLtwo}.
		\end{aligned}
	\end{equation}
	Further, the difference of gradients can be rewritten as follows:
	\begin{equation}
		\begin{aligned}
			&\nabla (\fraknone-\frakntwo)\\
			=&\,\begin{multlined}[t] \frac{1}{(\bonestar \btwostar)^2} \Bigl\{ [(\qonestar-\qtwostar) \bonestar-\qonestar( \bonestar-\btwostar)]'\bonestar \btwostar\\-[(\qonestar-\qtwostar) \bonestar-\qonestar( \bonestar-\btwostar)]((\bonestar)'\nabla \thetaonestar \btwostar+\bonestar (\btwostar)'\nabla \thetatwostar)\Bigr\},
			\end{multlined}
		\end{aligned}
	\end{equation}
	where
	\begin{equation}
		\begin{aligned}
			&[(\qonestar-\qtwostar) \bonestar-\qonestar( \bonestar-\btwostar)]' \\
			=&\, \begin{multlined}[t]
				(	(\qonestar)' \nabla \thetaonestar-(\qtwostar)'\nabla \thetatwostar) \bonestar+ (\qonestar-\qtwostar)(\bonestar)'\nabla \thetaonestar\\-(\qonestar)'\nabla \thetaonestar( \bonestar-\btwostar) -\qonestar \left((\bonestar)'\nabla \thetaonestar- (\btwostar)'\nabla \thetatwostar\right).
			\end{multlined}
		\end{aligned}
	\end{equation}
	Thus, similarly to before by relying on the uniform boundedness of $\thetaonestar$ and $\thetatwostar$, we obtain
	\begin{equation}
		\|\nabla (\fraknone-\frakntwo)\|_{\LtwotLtwo} \lesssim \|\otheta\|_{\LtwotHone}.
	\end{equation}
	To estimate the difference of $\frakl$ terms, we note that
	\begin{equation}
		\begin{aligned}
			\fraklone-\frakltwo  =&\, 2\frac{\konestar}{\bonestar}\ponestart-2\frac{\ktwostar}{\btwostar}\ptwostart 
			=\, 2 \frac{\konestar}{\bonestar} \opstart + 2 \frac{\konestar \btwostar-\ktwostar \bonestar}{\bonestar \btwostar} \ptwostart \\
			=&\, 2 \frac{\konestar}{\bonestar} \opstart + 2 \frac{(\konestar-\ktwostar) \btwostar-\ktwostar (\bonestar-\btwostar)}{\bonestar \btwostar} \ptwostart. 
		\end{aligned}
	\end{equation}
	From here using a similar reasoning to before we find that
	\begin{equation}
		\begin{aligned}
			\|\fraknone-\frakntwo\|_{\LtwotLtwo} \lesssim \|\opstart\|_{\LtwotLtwo}+\|\othetastar\|_{\LtwotLtwo}.
		\end{aligned}
	\end{equation}
	By putting the estimates together, we conclude that \eqref{est F} holds.
\end{proof}
\end{appendices}
\bibliography{references_new}{}
\bibliographystyle{siam} 
			\end{document}